%% file: main.tex
\documentclass[12pt]{amsart}
\usepackage[top=80pt,bottom=80pt,left=70pt,right=70pt]{geometry}
\usepackage{graphicx} 
\usepackage{amsmath}
\usepackage{enumitem}
\usepackage{amssymb}
\usepackage{xcolor}
\usepackage{tikz}
\usepackage{amsthm}
\usepackage{hyperref}

\newtheorem{theorem}{Theorem}[section]

\newtheorem{lemma}{Lemma}[section]

\newtheorem{claim}{Claim}[section]

\newtheorem{mainTheorem}{Theorem}

\theoremstyle{definition}
\newtheorem{definition}{Definition}[section]

\newtheorem{remark}[theorem]{Remark}

\counterwithin{equation}{section}
\renewcommand\qedsymbol{$\blacksquare$}

\newcommand{\norm}[1]{\left\lVert#1\right\rVert}
\newcommand{\intd}{{\rm d}}
\newcommand{\interior}[1]{%
  {\kern0pt#1}^{\mathrm{o}}%
}

\title[Rigidity of some partially hyperbolic automorphisms]{Centralizer classification and rigidity for some partially hyperbolic toral automorphisms}
\author{Sven Sandfeldt}
\address{Sven Sandfeldt, Department of mathematics, Kungliga Tekniska högskolan, Lindstedtsvägen 25, SE-100 44 Stockholm, Sweden.}
\email{svensan@kth.se}
\date{July 2023}

\begin{document}

\begin{abstract}
In this paper, we consider local centralizer classification and rigidity of some toral automorphisms. In low dimensions, we classify up to finite index possible centralizers for volume preserving diffeomorphisms $f$ $C^{1}-$close to an ergodic irreducible toral automorphism $L$. Moreover, we show a rigidity result in the case that the centralizer of $f$ is large: if the smooth centralizer $Z^{\infty}(f)$ is virtually isomorphic to that of $L$ then $f$ is $C^{\infty}-$conjugate to $L$. In higher dimensions, we show a similar rigidity result for certain irreducible toral automorphisms. We also classify up to finite index all possible centralizers for symplectic diffeomorphisms $C^{5}-$close to a class of irreducible symplectic automorphisms on tori of any dimension.
\end{abstract}

\maketitle

\input{Introduction}

\input{MainResult}

\input{Background}

\input{FakeConjugacy}

\input{HomologyRepInj}

\input{TopConjForHypCenter}

\input{SymplecticPerturbations}

\input{PropertiesOfPHAut}

\newpage
\bibliography{main.bib}{}
\bibliographystyle{abbrv}

\end{document}

%% file: Introduction.tex
\section{Introduction}

The $C^{r}-$smooth centralizer of a diffeomorphism $f:M\to M$ is the collection of all $g:M\to M$ $C^{r}-$diffeomorphisms such that $fg = gf$. That is, the centralizer of $f$ consists of the (global) smooth change of coordinates that keeps $f$ unchanged. We denote by $Z^{r}(f)$, $r = 0,1,2,...,\infty$, the $C^{r}-$smooth centralizer of $f$. Given any $C^{r}-$diffeomorphism $f$ and $n\in\mathbb{Z}$, we can always produce elements in $Z^{r}(f)$ by taking iterates of $f$, $f^{n}\in Z^{r}(f)$. We denote by $\langle f\rangle = \{f^{n}\}_{n\in\mathbb{Z}}\subset Z^{r}(f)$ and call this the trivial part of the centralizer. For $r\geq 1$ it was conjectured by Smale that a generic $C^{r}-$diffeomorphism has $C^{r}-$centralizer consisting entirely of the trivial part \cite{Smale1,Smale2}. This conjecture was solved in the $C^{1}-$topology by Bonatti, Crovisier and Wilkinson \cite{SmaleConjecture}. In the Anosov setting Smale's conjecture is known to hold for $C^{\infty}-$diffeomorphisms by Palis and Yoccoz \cite{SmaleConjectureYoccozPalis}, and for $C^{r}-$diffeomorphisms by Rocha and Varandas \cite{SmaleConjectureRochaVarandas}. For certain Anosov diffeomorphisms on tori (or more generally nilmanifolds) we can be more precise in describing the centralizer. Let $f:\mathbb{T}^{d}\to\mathbb{T}^{d}$ be Anosov, with the induced map on homology $f_{*}\in{\rm GL}(d,\mathbb{Z})$ irreducible\footnote{A matrix $A\in{\rm GL}(d,\mathbb{Z})$ is irreducible if its characteristic polynomial $p_{A}(t)\in\mathbb{Z}[t]$ is irreducible in $\mathbb{Q}[t]$.}. The results of Manning \cite{FranksManningConjugacy} and Adler Palais \cite{ConjRigidity} imply that $f$ has discrete $C^{\infty}-$centralizer isomorphic to some abelian subgroup $Z_{{\rm GL}(d,\mathbb{Z})}(f_{*})\subset{\rm GL}(d,\mathbb{Z})$. Moreover, the global rigidity result of F. Rodriguez Hertz and Wang \cite{HigherRankRigidity} implies that the diffeomorphism $f$ is smoothly conjugated to $f_{*}$ if $Z^{\infty}(f)$ has rank $\geq 2$. That is, either $f$ has a virtually trivial $C^{\infty}-$centralizer or $f$ is $C^{\infty}-$conjugated to $f_{*}$. So, for Anosov diffeomorphisms on tori with irreducible induced map on homology, we have a complete \textit{centralizer classification}: either $Z^{\infty}(f)$ is virtually trivial or $Z^{\infty}(f)\cong Z_{{\rm GL}(d,\mathbb{Z})}(f_{*})$. In the latter case we also obtain \textit{centralizer rigidity} when ${\rm rank}(Z_{{\rm GL}(d,\mathbb{Z})}(f_{*})) > 1$ (which is guaranteed, for example, if $d\geq 5$): if $Z^{\infty}(f)\cong Z_{{\rm GL}(d,\mathbb{Z})}(f_{*})$ then $f$ is $C^{\infty}-$conjugated to $f_{*}$. A step towards weakening the Anosov assumption in the centralizer classification statement was taken by Damjanović, Wilkinson and Xu in \cite{CentralizerRigidity}. The authors classify possible centralizers for (volume preserving and ergodic) perturbations of some trivial circle extensions of Anosov automorphisms. In \cite{CentralizerRigidity} the authors also raise the question of whether a similar type of centralizer classification is possible near other algebraic partially hyperbolic systems. In this paper, we answer this question for certain ergodic partially hyperbolic automorphisms on tori.

Let ${\rm Diff}_{\rm vol}^{\infty}(\mathbb{T}^{d})$ denote the volume preserving $C^{\infty}-$diffeomorphisms of $\mathbb{T}^{d}$.
\begin{mainTheorem}\label{Thm:MainTheorem1}
Let $L\in{\rm GL}(4,\mathbb{Z})$ be irreducible and induce an ergodic automorphism of $\mathbb{T}^{4}$. If $f\in{\rm Diff}_{\rm vol}^{\infty}(\mathbb{T}^{4})$ is $C^{1}-$close to $L$ then we have a dichotomy
\begin{enumerate}[label = (\roman*)]
    \item either $Z^{\infty}(f)$ is virtually trivial,
    \item or $f$ is $C^{\infty}-$conjugate to $L$.
\end{enumerate}
\end{mainTheorem}
In dimension $6$ we show a similar result for automorphisms with complex eigenvalues.
\begin{mainTheorem}\label{Thm:MainTheorem2}
Let $L\in{\rm GL}(6,\mathbb{Z})$ be irreducible with only complex eigenvalues and induce an ergodic automorphism of $\mathbb{T}^{6}$. If $f\in\rm Diff_{vol}^{\infty}(\mathbb{T}^{6})$ is a $C^{1}-$small perturbation of $L$ then we have a dichotomy
\begin{enumerate}[label = (\roman*)]
    \item either $Z^{\infty}(f)$ is virtually trivial,
    \item or $f$ is $C^{\infty}-$conjugate to $L$.
\end{enumerate}
\end{mainTheorem}
\begin{remark}
In both Theorems, the conclusion follows from the global rigidity of higher rank Anosov actions \cite{HigherRankRigidity} when $L$ is hyperbolic, so the novelty in our results is in the case when $L$ is not hyperbolic, but only partially hyperbolic.
\end{remark}
Theorems \ref{Thm:MainTheorem1} and \ref{Thm:MainTheorem2} can be put in the more general framework of centralizer classification and centralizer rigidity. If $f_{0}:M\to M$ is some model system with a known centralizer (for example, an irreducible automorphism of the torus) then one can study the centralizers of perturbations $f$ of $f_{0}$. We ask: what are the possible centralizers of perturbations $f$ of $f_{0}$? This is the problem of local centralizer classification at $f_{0}$. Second, assuming that the centralizer of $f$ is (virtually) isomorphic to the centralizer of $f_{0}$, what can be said about the dynamics of $f$? This is the problem of centralizer rigidity. Theorems \ref{Thm:MainTheorem1} and \ref{Thm:MainTheorem2} both show that the possible centralizers of $f$ close to $L$ are either $\mathbb{Z}$ (which should be the generic situation by Smale's conjecture) or $Z_{\rm aff}(L)$. We also answer the centralizer rigidity problem: if $f$ has a centralizer virtually isomorphic to the centralizer of $L$ then $f$ is smoothly conjugate to $L$, so the dynamics is smoothly equivalent to that of $L$.

Theorems \ref{Thm:MainTheorem1} and \ref{Thm:MainTheorem2} are the strongest possible in the sense that they give a dichotomy between the generic situation (which by Smale's conjecture corresponds to a trivial centralizer) and a smooth conjugacy to an algebraic system. Both theorems follow from general results on tori of any dimension, which we state in the next section.

The question of centralizer classification and centralizer rigidity has been formulated and addressed by Damjanović, Wilkinson and Xu in \cite{CentralizerRigidity}. In the same paper, a centralizer classification result for perturbations of geodesic flows on negatively curved surfaces was obtained. The authors also consider ergodic volume preserving perturbations of circle extensions of hyperbolic automorphisms of the torus. In particular, for product maps $A\times R_{\omega}$ where $A$ is a hyperbolic automorphism and $R_{\omega}$ a circle translation, it follows from results in \cite{CentralizerRigidity} that an ergodic volume preserving $C^{1}-$small perturbation can only have centralizer of a certain prescribed type. The main feature of the toral systems considered in \cite{CentralizerRigidity} is that they are fibered (the center foliation has uniformly compact leaves) and that the center foliation has dimension $1$. Centralizer classification and rigidity have also been studied for derived-from-Anosov maps on $\mathbb{T}^{3}$ in \cite{CentralizerOn3Torus}. In \cite{CentralizerOn3Mfd} centralizer classification was studied for partially hyperbolic diffeomorphisms on $3-$manifolds. In all these works the dimension of the center direction is $1$.

The main novelty in our work is that the partially hyperbolic systems we consider here are not fibered partially hyperbolic systems (as are the systems on the torus considered in \cite{CentralizerRigidity}) and the center dimension is not $1$, so the approach in \cite{CentralizerRigidity} can not be applied. 

In the next section, we formulate results for automorphisms on tori of any dimension, which have Theorems \ref{Thm:MainTheorem1} and \ref{Thm:MainTheorem2} as a consequence and we explain the strategy of proofs.

%% file: MainResult.tex
\subsection{Main results}

The systems that we will consider are irreducible toral automorphisms $L:\mathbb{T}^{d}\to\mathbb{T}^{d}$ with $2-$dimensional isometric center. These automorphisms have been heavily studied and can be shown to be, among other things, partially hyperbolic with the essential accessibility property, Bernoulli and exponentially mixing. We denote by $Z_{\text{Aut}}(L)$ the centralizer of $L\in{\rm GL}(d,\mathbb{Z})$ in ${\rm GL}(d,\mathbb{Z})$. Equivalently, we define by $Z_{\text{Aut}}(L)$ the centralizer of $L$ in the automorphism group of $\mathbb{T}^{d}$. Since any automorphism of the torus is smooth we have a natural injective map $Z_{\text{Aut}}(L)\to Z^{\infty}(L)$, it is well-known that this map is virtually an isomorphism \cite{ConjRigidity}. Moreover, in \cite{CentAut} it is shown that virtually $Z_{\text{Aut}}(L) = \mathbb{Z}^{r_{1} + r_{2} - 1}$ where $r_{1} = r_{1}(L)$ is the number of real eigenvalues of $L$ and $r_{2} = r_{2}(L)$ is the number of pairs of complex eigenvalues of $L$. Irreducibility of $L$ implies that any subgroup $\Gamma\leq Z_{\text{Aut}}(L)$, containing $L$, with $\rm rank(\Gamma) > 1$ does not factor through a $\mathbb{Z}-$action on some subtorus. A $\mathbb{Z}^{k}-$action by toral automorphisms is said to be of \textit{higher rank} if the action does not factor through a non-trivial $\mathbb{Z}-$action (see Definition \ref{Def:HigherRank}). So any $\Gamma\leq Z_{\text{Aut}}(L)$, $L\in\Gamma$, with ${\rm rank}(\Gamma) > 1$ defines a higher rank action.
\begin{theorem}\label{Thm:InjectivityOfHom}
Let $L\in{\rm GL}(d,\mathbb{Z})$ be an irreducible automorphism with isometric center of dimension $2$. Let $f\in\rm Diff_{vol}^{\infty}(\mathbb{T}^{d})$ be a $C^{1}-$close to $L$. The natural homomorphism $Z^{\infty}(f)\to Z_{\rm Aut}(L)$ given by the homology representation is virtually injective. In particular, if $Z^{\infty}(f)$ is not virtually trivial then the action of $Z^{\infty}(f)$ on $\mathbb{T}^{d}$ is higher rank.
\end{theorem}
Theorem \ref{Thm:InjectivityOfHom} puts an upper bound on the size of the smooth centralizer for any volume preserving $f$ that is $C^{1}-$close to $L$. Indeed, the homology representation is injective so we can define $\text{rank}(Z^{\infty}(f))$ as the rank of the image of $Z^{\infty}(f)$ in $Z_{\text{Aut}}(L)$ and immediately obtain $\text{rank}(Z^{\infty}(f))\leq r_{1}(L) + r_{2}(L) - 1$. When we have equality $\text{rank}(Z^{\infty}(f)) = r_{1}(L) + r_{2}(L) - 1$ (or equivalently, when $Z^{\infty}(f)\to Z_{\text{Aut}}(L)$ is virtually an isomorphism) we obtain a rigidity result for $f$.
\begin{theorem}\label{Thm:FullCentralizer}
Let $L\in{\rm GL}(d,\mathbb{Z})$ be an irreducible automorphism with isometric center of dimension $2$ and no three eigenvalues of the same modulus. If $f\in{\rm Diff_{vol}^{\infty}}(\mathbb{T}^{d})$ is $C^{1}-$close to $L$ and ${\rm rank}(Z^{\infty}(f)) = r_{1} + r_{2} - 1$, or equivalently the homology representation is virtually an isomorphism, then $f$ is $C^{\infty}-$conjugate to $L$.
\end{theorem}
By Theorem \ref{Thm:FullCentralizer} we obtain centralizer rigidity: if $f$ has a centralizer virtually isomorphic to the centralizer of $L$, then $f$ is $C^{\infty}-$conjugate to $L$. We deduce Theorems \ref{Thm:MainTheorem1} and \ref{Thm:MainTheorem2} from Theorem \ref{Thm:FullCentralizer}.
\begin{proof}[Proof of Theorem \ref{Thm:MainTheorem1}]
When $d = 4$, any ergodic, irreducible $L\in\rm GL(4,\mathbb{Z})$ is either Anosov or irreducible with $2-$dimensional isometric center. The case where $L$ is Anosov follows from \cite{HigherRankRigidity}. If $L:\mathbb{T}^{4}\to\mathbb{T}^{4}$ is irreducible with $2-$dimensional isometric center, then $r_{1} = 2$ and $r_{2} = 1$ so the centralizer of $L$ is virtually $\mathbb{Z}^{2}$. It follows that a volume preserving $f$ that is $C^{1}-$close to $L$ either has a virtually trivial centralizer or a centralizer virtually isomorphic to that of $L$. So Theorem \ref{Thm:MainTheorem1} follows from Theorem \ref{Thm:FullCentralizer}.
\end{proof}
\begin{proof}[Proof of Theorem \ref{Thm:MainTheorem2}]
Similarly, if $L\in\text{GL}(6,\mathbb{Z})$ is irreducible with only complex eigenvalues then $L$ is either Anosov or $L$ has $2-$dimensional isometric center and the rank of the centralizer of $L$ is $2$. As in the case of $d = 4$, we see that Theorem \ref{Thm:MainTheorem2} follows from Theorem \ref{Thm:FullCentralizer}.
\end{proof}
We turn to the question of local centralizer classification. That is, describing possible centralizers for all perturbations of $L$. Since this question is answered in the volume preserving setting for $d = 4$ in Theorem \ref{Thm:MainTheorem1} we will restrict to the case $d\geq6$.

To state the next result, we need a property that is stronger than irreducibility. We say that that an automorphism $L$ has property $(P)$ if $L$ is irreducible, has precisely two eigenvalues on the unit circle and all other eigenvalues are real, see Definition \ref{Def:PropertyP}. The conjugacy in Theorem \ref{Thm:FullCentralizer} follows from the existence of an element $g\in Z^{\infty}(f)$ that is homotopic to a hyperbolic automorphism. By studying the maximal rank of a subgroup of $Z_{\text{Aut}}(L)$ that does not contain any hyperbolic matrix we obtain the following classification of possible centralizers.
\begin{theorem}\label{Thm:HomotopicToHyperbolic}
Let $L\in\text{GL}(d,\mathbb{Z})$, $d\geq6$, have property $(P)$ and let $f\in{\rm Diff_{vol}^{\infty}}(\mathbb{T}^{d})$ be $C^{5}-$close to $L$. Then
\begin{enumerate}[label = (\roman*)]
    \item either ${\rm rank}(Z^{\infty}(f))\leq (d-2)/2$,
    \item or $f$ is $C^{\infty}-$conjugate to $L$.
\end{enumerate}
\end{theorem}
\begin{remark}
For $L$ with property $(P)$ we have ${\rm rank}(Z_{\text{Aut}}(L)) = d-2$, so the theorem says that if $f$ is not conjugated to $L$ then the rank of the smooth centralizer of $f$ can be at most half of the rank of the centralizer of $L$.
\end{remark}
\begin{remark}
Considering this theorem, it seems reasonable to ask if one has a dichotomy similar to Theorems \ref{Thm:MainTheorem1} and \ref{Thm:MainTheorem2} in any dimension: either $Z^{\infty}(f)$ is virtually trivial or $f$ is smoothly conjugate to $L$.
\end{remark}
If $L\in\text{GL}(d,\mathbb{Z})$ has property $(P)$ then $L$ naturally preserves a symplectic form $\omega$. We denote by ${\rm Diff}_{\omega}^{\infty}(\mathbb{T}^{d})$ the space of diffeomorphisms that preserve $\omega$. If we assume that the perturbation $f$ of $L$ is still symplectic, then more can be said about the centralizer of $f$. If $f:M\to M$ and $g:N\to N$ are partially hyperbolic diffeomorphisms with (uniquely integrable) center foliations $W_{f}^{c}$ and $W_{g}^{c}$, then we say that $f$ and $g$ are $C^{r}-$leaf conjugated if they are $C^{r}-$conjugated modulo their center foliations (see Definition \ref{Def:LeafConjugacy}). We have the following theorem showing that $f$ is always smoothly leaf conjugated to $L$ when the centralizer of $f$ is not virtually trivial.
\begin{theorem}\label{Thm:LeafConjugacy}
Let $L\in{\rm GL}(d,\mathbb{Z})$ have property $(P)$. If $f\in\rm Diff_{\omega}^{\infty}(\mathbb{T}^{d})$ is $C^{5}-$close to $L$ and $Z^{\infty}(f)$ is not virtually trivial then $f$ is $C^{1+\alpha}-$leaf conjugated to $L$.
\end{theorem}
If $L$ has property $(P)$ then we say that $L$ has $r-$spread spectrum (see Definition \ref{Def:SpredSpec}) if for two eigenvalues $\lambda,\mu$ of $L$ such that $|\lambda| > |\mu| > 1$ we have $|\lambda| > |\mu|^{r}$. If we know that $L$ has property $(P)$ and $3-$spread spectrum then we obtain full local centralizer classification and rigidity for symplectic perturbations of $L$.
\begin{theorem}\label{Thm:SymplecticRigidity}
Let $L\in{\rm GL}(d,\mathbb{Z})$ have property $(P)$ and $3-$spread spectrum. If $f\in{\rm Diff}_{\omega}^{\infty}(\mathbb{T}^{d})$ is $C^{5}-$close to $L$ then we have a dichotomy
\begin{enumerate}[label = (\roman*)]
    \item either $Z^{\infty}(f)$ is virtually trivial,
    \item or $f$ is $C^{\infty}-$conjugate to $L$.
\end{enumerate}
\end{theorem}
\begin{remark}
In Lemma \ref{L:ExistenceSpreadSpectrum} we show that for any $r\in\mathbb{N}$ and even $d$ there is $L\in\text{GL}(d,\mathbb{Z})$ with property $(P)$ and $r-$spread spectrum.
\end{remark}

\subsubsection{Strategy of proofs}

Our results rely in an essential way on the results and methods developed by F. Rodriguez Hertz for studying stable ergodicity of partially hyperbolic toral automorphisms \cite{StabelErgodicity}. From \cite{StabelErgodicity} we have a good understanding of the stable and unstable foliations for perturbations $f$ of some irreducible automorphism $L$ with isometric $2-$dimensional center. In particular, we have a dichotomy: if $f\in{\rm Diff}_{\rm vol}^{\infty}(\mathbb{T}^{d})$ is sufficiently $C^{r}-$close (where $r = 5$ for $d\geq 6$ and $r = 22$ for $d = 4$) to $L$ then $f$ is non-accessible if and only if $f$ is topologically conjugated to $L$. So, to find a topological conjugacy from $f$ to $L$ it suffices to show that $f$ is not accessible. This is used directly in the proof of Theorem \ref{Thm:HomotopicToHyperbolic} and indirectly in the proof of Theorem \ref{Thm:SymplecticRigidity}. The conjugacy produced in \cite{StabelErgodicity} requires control of the perturbation in high regularity, so to deal with $C^{1}-$small perturbations in Theorem \ref{Thm:FullCentralizer} we develop a different method to produce a topological conjugacy.

The starting point of all proofs is Theorem \ref{Thm:InjectivityOfHom} which is proved in Section \ref{Sec:InjectivityOfHom}. By Theorem \ref{Thm:InjectivityOfHom} the assumption ${\rm rank}(Z^{\infty}(f)) > 1$ implies that $Z^{\infty}(f)$ acts on $\mathbb{T}^{d}$ as a higher rank action. The centralizer $Z^{\infty}(f)$ might not contain an Anosov element, so it is not immediate that the action of $Z^{\infty}(f)$ is bi-Hölder conjugated to its linearization (as it is in \cite{HigherRankRigidity,FisherKalininSpatzier}). It follows that the approach proving rigidity in \cite{HigherRankRigidity,FisherKalininSpatzier} can not be applied. Instead, in Section \ref{Sec:FranksManninCoordinates} we construct, following Manning \cite{FranksManningConjugacy}, Hölder coordinates along the stable and unstable foliations that conjugate $f$ to its linearization along the leaves. Using these coordinates we can control the behavior of all elements in $Z^{\infty}(f)$ along the stable and unstable leaves. This is used in two different ways. In Section \ref{Sec:FullCentralizer} we prove that if $Z^{\infty}(f)$ has maximal rank then there is an element $g\in Z^{\infty}(f)$ such that $g$ is homotopic to a hyperbolic automorphism and therefore is semi-conjugated to its linearization \cite{KatokIntroDynSyst}. We show that this semi-conjugacy is, in fact, a bi-Hölder homeomorphism. Once it is established that the action of $Z^{\infty}(f)$ is bi-Hölder conjugate to its linearization, it follows that any element $g\in Z^{\infty}(f)$ homotopic to some hyperbolic automorphism is Anosov. Once we have an Anosov element in $Z^{\infty}(f)$, the global rigidity result by F. Rodriguez Hertz and Wang \cite{HigherRankRigidity} can be applied. In Section \ref{Sec:SymplecticPerturbations} we use methods from Katok and Lewis \cite{KatokLewis} to show that the coordinates constructed in Section \ref{Sec:FranksManninCoordinates} are more regular than Hölder when the perturbation is symplectic. The regularity of the coordinates along stable and unstable foliations implies that the stable and unstable foliation themselves are more regular and we can show that $f$ is $C^{\infty}-$conjugated to $L$ by applying results by Gogolev, Kalinin and Sadovskaya \cite{CenterFoliationRig}.

\subsubsection{Structure of paper}

In Section \ref{Sec:DefBackground} we recall some definitions and results from partially hyperbolic dynamics and higher rank actions. In Section \ref{Sec:FranksManninCoordinates} we show the existence of special coordinates along some invariant foliations, which are used as a substitute for a Franks-Manning conjugacy \cite{FranksManningConjugacy}. These coordinates are then used to control the volume growth of elements $g\in Z^{\infty}(f)$ along the center direction. In Section \ref{Sec:InjectivityOfHom} we show that the homology representation is injective proving Theorem \ref{Thm:InjectivityOfHom}. In Section \ref{Sec:FullCentralizer} we prove Theorem \ref{Thm:FullCentralizer} and \ref{Thm:HomotopicToHyperbolic}. In Section \ref{Sec:SymplecticPerturbations} we prove Theorem \ref{Thm:LeafConjugacy} and \ref{Thm:SymplecticRigidity}. Finally in Appendix \ref{Sec:AlgebraicProperties} we recall and prove some basic properties of irreducible partially hyperbolic automorphisms of the torus.

\subsubsection*{Acknowledgement}
The author is supported by Swedish Research Council grant
2019-67250. I would like to thank Pennsylvania State University for their hospitality during the visit where a large part of this paper was written. I would like to thank Federico Rodriguez Hertz, Joshua Paik and Danijela Damjanović for useful discussions while writing this paper. I would also like to thank Boris Kalinin for explaining his result \cite{RealNormalFormsExist} and Davi Obata for pointing out an error with the application of Journé's lemma in a previous version of the paper. I would also like to thank the anonymous referee for useful comments, that improved the quality of the paper.

%% file: Background.tex
\section{Definitions and background}
\label{Sec:DefBackground}

\subsection{Partially hyperbolic diffeomorphisms}

Let $M$ be a closed smooth manifold and let $f\in{\rm Diff}^{r}(M)$. We say that $f$ is (absolutely) partially hyperbolic if there is an $f-$invariant continuous decomposition
\begin{align*}
TM = E^{u}\oplus E^{c}\oplus E^{s}
\end{align*}
such that $E^{u}$ is uniformly expanded, $E^{s}$ is uniformly contracted and the behaviour on $E^{c}$ is dominated by the behaviour along $E^{u}$ and $E^{s}$. That is, we ask that $D_{x}fE^{\sigma}(x) = E^{\sigma}(fx)$, for $\sigma = s,c,u$, and there exists $\lambda,\nu,\widehat{\lambda},\widehat{\nu}\in(0,1)$, $\lambda < \widehat{\lambda}$ and $\nu < \widehat{\nu}$, such that for unit vectors $v^{u}\in E^{u}(p)$, $v^{c}\in E^{c}(p)$ and $v^{s}\in E^{s}(p)$ we have
\begin{align*}
& \norm{D_{p}f^{n}(v^{s})}\leq C\cdot\lambda^{n}, \\
& \frac{1}{C}\cdot\widehat{\lambda}^{n}\leq\norm{(D_{p}f^{n})^{-1}(v^{c})}^{-1}\leq\norm{D_{p}f^{n}(v^{c})}\leq C\cdot\widehat{\nu}^{-n}, \\
& \frac{1}{C}\cdot\nu^{-n}\leq\norm{\left(D_{p}f^{n}\right)^{-1}(v^{u})}^{-1}
\end{align*}
where $C\geq1$ is some uniform constant.

The distributions $E^{s}$ and $E^{u}$ are always uniquely integrable to Hölder foliations with uniformly $C^{r}-$leaves. We denote these foliations by $W^{s}$ and $W^{u}$ or $W_{f}^{s}$ and $W_{f}^{u}$. It is not always true that the bundle $E^{c}$ is integrable, and even if it is, it is not always true that $E^{c}$ is uniquely integrable. A sufficient condition to guarantee that $W^{c}$ is integrable is that $f$ is \textit{Dynamically coherent}.
\begin{definition}\label{Def:DynamicalCoherence}
We say that a partially hyperbolic map $f\in{\rm Diff}^{r}(M)$, $r\geq1$, is dynamically coherent if $E^{u}\oplus E^{c}$ and $E^{c}\oplus E^{s}$ are integrable to foliations $W^{cu}$ and $W^{cs}$. In this case $E^{c}$ is integrable to the foliation $W^{c} = W^{cu}\cap W^{cs}$.
\end{definition}
\begin{remark}
Note that Definition \ref{Def:DynamicalCoherence} does not assert that $W^{cs}$ and $W^{cu}$ are \textit{uniquely integrable}, that is there might exist more than one foliation $W^{cs}$ tangent to $E^{s}\oplus E^{c}$ and similarly for $W^{cu}$.  In this paper, however, all partially hyperbolic systems considered will be dynamically coherent with $W^{cs}$ and $W^{cu}$ uniquely integrable, see Section \ref{Sec:AutomorphismsOnTori} and in particular Theorem \ref{Thm:RegularityOfCenter}.
\end{remark}
We denote by $W^{\sigma}_{\delta}(x)$ the $\delta-$ball inside $W^{\sigma}(x)$. Let $f:M\to M$ and $g:N\to N$ be two partially hyperbolic diffeomorphsims with (uniquely integrable) center foliations $W_{f}^{c}$ and $W_{g}^{c}$. We say that $f$ and $g$ are \textit{leaf conjugate} if they are conjugate modulo the center foliations. More precisely:
\begin{definition}\label{Def:LeafConjugacy}
We say that $f$ and $g$ are leaf conjugate if there is a homeomorphism $H:M\to N$ such that
\begin{align}
H(W_{f}^{c}(fx)) = W_{g}^{c}(gH(x))
\end{align}
for $x\in M$.
\end{definition}
\begin{definition}
Let $(M,\mathcal{F})$ be a foliated manifold where $\mathcal{F}$ is a continuous foliation with uniformly $C^{1}-$leaves. Let $f:M\to M$ preserve the foliation $\mathcal{F}$. We say that $(f,\mathcal{F})$ is $r-$normally hyperbolic if the tangent bundle has a $Df-$invariant splitting
\begin{align*}
T_{p}M = E^{s}(p)\oplus E^{c}(p)\oplus E^{u}(p)
\end{align*}
where $E^{c}(p) = T_{p}\mathcal{F}(p)$, $E^{s}$ is uniformly contracted, $E^{u}$ is uniformly expanded and there is some $k\geq1$ such that
\begin{align*}
& \norm{(D_{p}f^{k}|_{E^{u}})^{-1}}\cdot\norm{D_{p}f^{k}|_{E^{c}}}^{r} < 1, \\
& \norm{D_{p}f^{k}|_{E^{s}}}\cdot\norm{(D_{p}f^{k}|_{E^{c}})^{-1}}^{r} < 1
\end{align*}
for all $p\in M$.
\end{definition}
We have the following results from \cite{InvariantManifolds}.
\begin{theorem}\label{Thm:HPS1}
Let $f\in{\rm Diff}^{\infty}(M)$ and let $\mathcal{F}$ be a $r-$normally hyperbolic foliation for $f$. Then the leaves of $\mathcal{F}$ are uniformly $C^{r}$.
\end{theorem}
\begin{theorem}\label{Thm:HPS2}
Let $f\in{\rm Diff}^{\infty}(M)$ and let $\mathcal{F}$ be a $r-$normally hyperbolic foliation for $f$. If $g$ is sufficiently $C^{1}-$close to $f$ then there is some foliation $\mathcal{G}$ which is $r-$normally hyperbolic for $g$. Moreover $T_{x}\mathcal{G}$ can be made to be close to $T_{x}\mathcal{F}$ by letting $g$ be sufficiently $C^{1}-$close to $f$.
\end{theorem}
Let $L:\mathbb{T}^{d}\to\mathbb{T}^{d}$ be a partially hyperbolic toral automorphism with a partially hyperbolic invariant splitting $V\oplus V'\oplus V''$ (here we do not assume that $V'$ is isometric). Any $f$ that is $C^{1}-$close to $L$ is partially hyperbolic with locally uniquely integrable $W^{\sigma}$, $\sigma = s,c,u,cs,cu$, \cite{BrinDynCoh} or \cite{InvariantManifolds}. Combining this with Theorems \ref{Thm:HPS1} and \ref{Thm:HPS2}, the following theorem follows.
\begin{theorem}\label{Thm:RegularityOfCenter}
If $L:\mathbb{T}^{d}\to\mathbb{T}^{d}$ is a partially hyperbolic toral automorphism with invariant splitting $V\oplus V'\oplus V''$ (with $L$ not necesserily isometric along $V'$) then any $f$ that is $C^{1}-$close to $L$ is partially hyperbolic, dynamically coherent with $E_{f}^{c}$ close to $V'$ and each $W_{f}^{\sigma}$ is locally uniquely integrable. If $\norm{(L|_{V''})^{-1}}\cdot\norm{L|_{V'}}^{r} < 1$ and $\norm{L|_{V}}\cdot \norm{(L|_{V'})^{-1}}^{r} < 1$ then $W_{f}^{c}$ has uniformly $C^{r}-$leaves. Moreover if $g\in\rm Diff^{1}(\mathbb{T}^{d})$ commute with $f$ then $g$ preserves all foliations $W_{f}^{\sigma}$, $\sigma = s,c,u,cs,cu$.
\end{theorem}
\begin{remark}
The last claim follows by the local uniqueness.
\end{remark}
Let $f\in{\rm Diff}^{\infty}(M)$ be a partially hyperbolic diffeomorphism. We say that a curve $\gamma:[0,1]\to M$ is a $su-$path if there are $0 = t_{0} < t_{1} < ... < t_{N} = 1$ such that ${\rm Im}(\gamma|_{[t_{i-1},t_{i}]})$ is contained entirely in a $W^{s}-$leaf or $W^{u}-$leaf.
\begin{definition}
Let $f\in{\rm Diff}^{\infty}(M)$ be partially hyperbolic. We say that $f$ is accessible if every $x,y\in M$ can be connected by an $su-$path.
\end{definition}
\begin{definition}
Let $f\in{\rm Diff}^{r}(M)$ be partially hyperbolic. Let $E\subset M$ be a subset. We say that $E$ is $s-$saturated if $x\in E$ implies that $W^{s}(x)\subset E$. Similarly we say that $E$ is $u-$saturated if $x\in E$ implies that $W^{u}(x)\subset E$. We say that $E\subset M$ is $su-$saturated if $E$ is $s-$saturated and $u-$saturated.
\end{definition}
The following lemma is immediate from the definitions.
\begin{lemma}
A partially hyperbolic diffeomorphism $f\in{\rm Diff}^{\infty}(M)$ is accessible if and only if the only $su-$saturated sets are $M$ and $\emptyset$.
\end{lemma}

\subsection{Automorphisms of the torus}\label{Sec:AutomorphismsOnTori}

Let $\mathbb{T}^{d} = \mathbb{R}^{d}/\mathbb{Z}^{d}$ be the $d-$dimensional torus. We denote by ${\rm GL}(d,\mathbb{Z})$ the integer $d\times d-$matrices with determinant $\pm1$. Any $L\in{\rm GL}(d,\mathbb{Z})$ induce an automorphism $L:\mathbb{T}^{d}\to\mathbb{T}^{d}$. We will in the remainder abuse notation and not make a distinction between the automorphism $L:\mathbb{T}^{d}\to\mathbb{T}^{d}$, the automorphism $L:\mathbb{R}^{d}\to\mathbb{R}^{d}$ and the matrix $L\in{\rm GL}(d,\mathbb{Z})$.
\begin{theorem}
An automorphism $L:\mathbb{T}^{d}\to\mathbb{T}^{d}$ is volume preserving and $L$ is ergodic with respect to volume if and only if $L$ has no root of unity as an eigenvalue.
\end{theorem}
Given $L\in{\rm GL}(d,\mathbb{Z})$ we denote by $p_{L}(t)$ the characteristic polynomial of $L$.
\begin{definition}
We say that $L:\mathbb{T}^{d}\to\mathbb{T}^{d}$ is irreducible if $p_{L}(t)$ is irreducible over $\mathbb{Q}$ (or equivalently, by Gauss lemma, over $\mathbb{Z}$).
\end{definition}
Given $L\in{\rm GL}(d,\mathbb{Z})$ we define the Lyapunov exponents of $L$ as the numbers that we can obtain by
\begin{align*}
\chi(L) = \chi = \lim_{n\to\infty}\frac{1}{n}\log\norm{L^{n}v},\quad v\neq0
\end{align*}
and denote by ${\rm Lyap}(L)$ the set of Lyapunov exponents of $L$. One notes that $\chi$ is a Lyapunov exponent if $\chi = \log|\lambda|$ for some eigenvalue $\lambda$ of $L$. Associated to $0\neq\chi\in{\rm Lyap}(L)$ we define the $L-$invariant subspace $E_{L}^{\chi}$ by
\begin{align*}
E_{L}^{\chi}\setminus\{0\} =\left\{v\text{ : }\lim_{n\to\pm\infty}\frac{1}{n}\log\norm{L^{n}v} = \chi\right\}.
\end{align*}
We also define the center space
\begin{align*}
E_{L}^{c}\setminus\{0\} =\left\{v\text{ : }\lim_{n\to\pm\infty}\frac{1}{n}\log\norm{L^{n}v} = 0\right\}.
\end{align*}
These spaces, $E_{L}
^{\chi}$ and $E_{L}^{c}$, produce a decomposition
\begin{align*}
\mathbb{R}^{d} = E_{L}^{c}\oplus\bigoplus_{0\neq\chi\in{\rm Lyap}(L)}E_{L}^{\chi}
\end{align*}
into Lyapunov subspaces. We fix an inner product on $\mathbb{R}^{d}$ making the decomposition into Lyapunov subspaces orthogonal and in the case where $L$ has no non-trivial Jordan blocks we also choose the inner product such that for $v\in E_{L}^{\chi}$ we have $\norm{Lv} = e^{\chi}\norm{v}$. Given $v\in\mathbb{R}^{d}$ denote by $v^{\chi}$ and $v^{c}$ the projection of $v$ onto $E_{L}^{\chi}$ and $E_{L}^{c}$ respectively:
\begin{align}\label{Eq:NotationForProjections}
\mathbb{R}^{d}\ni v\mapsto v^{\chi}\in E_{L}^{\chi},\quad\mathbb{R}^{d}\ni v\mapsto v^{c}\in E_{L}^{c}
\end{align}
We also define
\begin{align*}
E_{L}^{u} = \bigoplus_{\chi > 0}E_{L}^{\chi},\quad E_{L}^{s} = \bigoplus_{\chi < 0}E_{L}^{\chi},\quad E_{L}^{cs} = E_{L}^{s}\oplus E_{L}^{c},\quad E_{L}^{cu} = E_{L}^{c}\oplus E_{L}^{u}
\end{align*}
and similarly we denote by $v^{\sigma}$, $\sigma = c,u,s,cs,cu$, the projection of $v$ onto $E_{L}^{\sigma}$.

Let $L\in{\rm GL}(d,\mathbb{Z})$ be an ergodic automorphism of the torus. If $f\in{\rm Diff}^{\infty}(\mathbb{T}^{d})$ is sufficiently $C^{1}-$close to $L$ then there is an $f-$invariant, Hölder continuous dominated splitting
\begin{align*}
T_{x}\mathbb{T}^{d} = E_{f}^{c}(x)\oplus\bigoplus_{0\neq\chi\in{\rm Lyap}(L)}E_{f}^{\chi}(x).
\end{align*}
Each distribution $E_{f}^{\chi}(x)$ is $C^{0}-$close to $E_{L}^{\chi}$ and integrable to a Hölder foliation $W_{f}^{\chi}(x)$ with $C^{1+\alpha}-$leaves for some $\alpha > 0$. The integrability follows from repeated use of Theorem \ref{Thm:RegularityOfCenter}, see also \cite{LyapSpecRigidity} where the same construction is used for the slow foliations. We will denote the distance along $W_{f}^{\chi}$ by $\intd_{\chi}$. More generally, given $0 < \chi_{0} < \chi_{1}$ we define 
\begin{align*}
E_{f}^{\chi_{0} < \chi < \chi_{1}} := \bigoplus_{\chi_{0} < \chi < \chi_{1}}E_{f}^{\chi}
\end{align*}
which also integrate some some foliation $W_{f}^{\chi_{0} < \chi < \chi_{1}}$ that is subfoliated by each $W_{f}^{\eta}$ with $\chi_{0} < \eta < \chi_{1}$.

For any homeomorphism $f:\mathbb{T}^{d}\to\mathbb{T}^{d}$ we will denote by $F:\mathbb{R}^{d}\to\mathbb{R}^{d}$ some lift of $f$. If $f\in{\rm Diff}^{\infty}(\mathbb{T}^{d})$ is $C^{1}-$close to some ergodic linear automorphism $L\in{\rm GL}(d,\mathbb{Z})$ then $f$ has a unique fixed point $x_{0}$ that is close to $0\in\mathbb{T}^{d}$. After conjugating $f$ with a translation close to $0$ we may assume without loss of generality that $x_{0} = 0$. We will make the convention that a lift $F$ of $f$ is chosen such that $F(0) = 0$. We can lift the distributions $E_{f}^{\sigma}$ and foliations $W_{f}^{\sigma}$ to $\mathbb{R}^{d}$, we denote these distributions and foliations by $E_{F}^{\sigma}$ and $W_{F}^{\sigma}$ respectively (here $\sigma = s,u,c,cu,su$ or $\sigma = \chi$ for some $0\neq\chi\in{\rm Lyap}(L)$). From the construction of invariant foliations in \cite{InvariantManifolds}, see also \cite{StabelErgodicity}, we can write
\begin{align*}
W_{F}^{\sigma}(x) = x + \text{Graph}(\gamma_{x}^{\sigma}) = \{x + v + \gamma_{x}^{\sigma}(v)\text{ : }v\in E_{L}^{\sigma}\}
\end{align*}
where $\gamma_{x}^{\sigma}:E_{L}^{\sigma}\to (E_{L}^{\sigma})^{\bot}$ and the orthogonal complement is taken with respect to the decomposition into Lyapunov subspaces. The functions $(x,v)\mapsto\gamma_{x}^{\sigma}(v)$ are Hölder continuous in $x$ and as smooth as the leaves of $W_{F}^{\sigma}$ in $v$. Moreover, since the foliations $W_{F}^{\sigma}$ are lifted from $\mathbb{T}^{d}$ we have $\gamma_{x+n}(v) = \gamma_{x}(v)$ for every $n\in\mathbb{Z}^{d}$. We obtain estimates on $\gamma_{x}^{\sigma}(v)$, see \cite[Lemma 2.1]{StabelErgodicity}.
\begin{lemma}\label{L:InvMfdAreGraph}
There is a constant $\kappa = \kappa(f)$ such that $\kappa(f)\to 0$ as $f\to L$ in $C^{1}$ such that for any $\sigma = s,c,u,cs,cu,\chi$
\begin{align*}
\norm{\gamma_{x}^{\sigma}(v)}\leq\kappa\norm{v}
\end{align*}
and for $\sigma = c,cs,cu$ we have
\begin{align*}
\norm{\gamma_{x}^{\sigma}(v)}\leq\kappa
\end{align*}
uniformly in $x$ and $v$.
\end{lemma}
\begin{remark}
In \cite{StabelErgodicity} this is only shown for $\sigma = s,c,u,cs,cu$ but it also holds for $\sigma = \chi$. To see this, note that if $f$ is sufficiently $C^{1}-$close to $L$ then we have an estimate $\norm{D_{v}\gamma_{x}^{\chi}}\leq\delta$ for some small $\delta$ (since $E_{f}^{\chi}$ is $C^{0}-$close to $E_{L}^{\chi}$) and if we integrate this we obtain the lemma.
\end{remark}
We also have the following global product structure \cite[Lemma 2.2]{StabelErgodicity}.
\begin{lemma}
For any $p,q\in\mathbb{R}^{d}$
\begin{align*}
& \#W_{F}^{cs}(p)\cap W_{F}^{u}(q) = 1, \\
& \#W_{F}^{s}(p)\cap W_{F}^{cu}(q) = 1.
\end{align*}
\end{lemma}
Consider $L\in{\rm GL}(d,\mathbb{Z})$ irreducible with precisely $2$ eigenvalues on the unit circle. That is, $L$ is irreducible and partially hyperbolic with $2-$dimensional isometric center. The fact that $L$ is irreducible with an eigenvalue on the unit circle implies that $d$ must be even, see \cite[Corollary A.4]{StabelErgodicity}, so we will restrict attention to $d$ even. The following results are due to F. Rodriguez Hertz \cite{StabelErgodicity}.
\begin{theorem}\label{Thm:RHsuMinimality}
If $f\in{\rm Diff}_{\rm vol}^{\infty}(\mathbb{T}^{d})$ is $C^{1}-$close to $L$ then any $f-$invariant $su-$saturated set is dense.
\end{theorem}
\begin{theorem}\label{Thm:RHaccessibilityDichom}
If $f\in{\rm Diff_{vol}^{\infty}}(\mathbb{T}^{d})$ is $C^{r}-$close to $L$ (with $r = 22$ for $d = 4$ and $r = 5$ for $d\geq6$) then we have a dichotomy
\begin{enumerate}[label = (\roman*)]
    \item either $f$ is accessible,
    \item or $f$ is bi-Hölder conjugate to $L$.
\end{enumerate}
\end{theorem}
\begin{remark}
This theorem is not stated explicitly in \cite{StabelErgodicity} but is key in the proof of the main result. These results have also been used in, for example, \cite{CenterFoliationRig,AvilaViana}.
\end{remark}
\begin{remark}
In \cite{StabelErgodicity} it is assumed that $L$ has an extra property called \textit{pseudo Anosov}, where one also assumes that $p_{L}(t)$ is not a polynomial of $t^{n}$ with $n\geq 2$. However, this is implied if $L$ has precisely two eigenvalues on $S^{1}$, see Lemma \ref{L:PseudoIsAllTrue}.
\end{remark}
Let $f\in{\rm Diff}^{\infty}(\mathbb{T}^{d})$ be  $C^{1}-$close to $L$. For $x,y\in\mathbb{R}^{d}$ we define the stable and unstable holonomy $\pi_{x,y}^{\sigma}$, $\sigma = s,u$, by
\begin{align*}
\pi_{x,y}^{s}:W_{F}^{cu}(x)\to W_{F}^{cu}(y),\quad \pi_{x,y}^{s}(z) = W_{F}^{s}(z)\cap W_{F}^{cu}(y),
\end{align*}
see Figure \ref{Fig:DefinitionOfHolonomies1}. Similarly we define $\pi_{x,y}^{u}$.
\begin{figure}[h!]
\begin{tikzpicture}
\draw [red] plot [smooth, tension=0.9] coordinates {(0,-3) (0,-2) (0.5,0) (0.2,2) (0.1,3)} node[anchor=west] {\scriptsize$W_{F}^{cu}(x)$};
\draw [red] plot [smooth, tension=1.3] coordinates {(3.1,-3) (3,-2.1) (3.5,-0.1) (3.2,1.95) (3.1,3)} node[anchor=west] {\scriptsize$W_{F}^{cu}(y)$};
\draw [blue] plot [smooth, tension=0.85] coordinates {(-1,0) (-0.5,0.1) (0,0.2) (0.5,0.3) (1,0.21) (1.5,0.12) (2,0.02) (2.5,0) (3,-0.1) (3.5,-0.2) (4,-0.1) (4.5,0)} node[anchor=west] {\scriptsize$W_{F}^{s}(x)$};
\draw [blue] plot [smooth, tension=0.85] coordinates {(-1,-2.1) (-0.5,-2) (0,-1.9) (0.5,-1.8) (1,-1.75) (1.5,-1.83) (2,-1.95) (2.5,-2) (3,-2.05) (3.5,-2.1) (4,-2.1) (4.5,-2)} node[anchor=west] {\scriptsize$W_{F}^{s}(x')$};
\draw [blue] plot [smooth, tension=0.85] coordinates {(-1,2.05) (-0.5,2.1) (0,2.15) (0.5,2.19) (1,2.21) (1.5,2.13) (2,2.07) (2.5,2) (3,1.9) (3.5,1.8) (4,1.9) (4.5,2)} node[anchor=west] {\scriptsize$W_{F}^{s}(x'')$};
\filldraw [purple] (0.5,0.3) circle (2pt) node[anchor=south east] {\tiny$x$};
\filldraw [purple] (0.03,-1.905) circle (2pt) node[anchor=south east] {\tiny$x'$};
\filldraw [purple] (0.18,2.17) circle (2pt) node[anchor=south east] {\tiny$x''$};
\filldraw [purple] (3.5,-0.2) circle (2pt) node[anchor=north west] {\Tiny$\pi_{x,y}^{s}(x)$};
\filldraw [purple] (3,-2.05) circle (2pt) node[anchor=north west] {\Tiny$\pi_{x,y}^{s}(x')$};
\filldraw [purple] (3.22,1.84) circle (2pt) node[anchor=north west] {\Tiny$\pi_{x,y}^{s}(x'')$};
\end{tikzpicture}
\caption{Stable holonomies from $x$ to $y\in W_{F}^{s}(x)$.}\label{Fig:DefinitionOfHolonomies1}
\end{figure}
Given $n\in\mathbb{Z}^{d}$ we define a point $x_{n}$ and two maps $\pi_{n}^{u}:W_{F}^{c}(n)\to W_{F}^{c}(x_{n})$, $\pi_{n}^{s}:W_{F}^{c}(x_{n})\to W_{F}^{c}(0)$, see Figure \ref{Fig:DefinitionOfHolonomies2}, by
\begin{align*}
x_{n} = W_{F}^{u}(n)\cap W_{F}^{cs}(0),\quad\pi_{n}^{u} = \pi_{n,x_{n}}^{u},\quad\pi_{n}^{s} = \pi_{x_{n},0}^{s}.
\end{align*}
\begin{figure}[h!]
\begin{tikzpicture}
\draw [black!60!green] plot [smooth, tension=0.9] coordinates {(0.2,-3) (0,-2) (0.3,0) (0.2,2) (0.1,3)} node[anchor=west] {\scriptsize$W_{F}^{c}(0)$};
\draw [black!60!green] plot [smooth, tension=0.9] coordinates {(4.1,-3) (4,-2) (4.2,0) (4.3,2) (4.2,3)} node[anchor=west] {\scriptsize$W_{F}^{c}(n)$};
\draw [red] plot [smooth, tension=0.9] coordinates {(1.5,2) (2,1) (3,-0.5) (4,-1.5) (5,-2)} node[anchor=west] {\scriptsize$W_{F}^{u}(x+n)$};
\draw [blue] plot [smooth, tension=0.9] coordinates {(2.5,2) (1.5,1) (0.5,-0.5) (-0.5,-1.5) (-1,-2)} node[anchor=east] {\scriptsize$W_{F}^{s}(x_{n})$};
\filldraw [gray] (0.29,1) circle (2pt) node[anchor=south west] {\Tiny$0$};
\filldraw [gray] (0.31,0.3) circle (2pt) node[anchor=south west] {\Tiny$x$};
\filldraw [gray] (4.1232,-0.8) circle (2pt) node[anchor=west] {\Tiny$n$};
\filldraw [gray] (4.03,-1.51) circle (2pt) node[anchor=south west] {\Tiny$x+n$};
\filldraw [gray] (1.815,1.33) circle (2pt) node[anchor=west] {\Tiny$\pi_{n}^{u}(x+n)$};
\filldraw [gray] (0.194,-0.87) circle (2pt) node[anchor=west] {\Tiny$T_{n}(x)$};
\end{tikzpicture}
\caption{Definition of the map $T_{n}:W_{F}^{c}(0)\to W_{F}^{c}(0)$.}\label{Fig:DefinitionOfHolonomies2}
\end{figure}
\begin{definition}\label{Def:TnMapsFromFederico}
We define the $C^{1}-$map $T_{n}:W_{F}^{c}(0)\to W_{F}^{c}(0)$, see Figure \ref{Fig:DefinitionOfHolonomies2},  by
\begin{align*}
T_{n}(x) := \pi_{n}^{s}\circ\pi_{n}^{u}(x + n)
\end{align*}
where $x + n\in W_{F}^{c}(n)$ since the foliation $W_{F}^{c}$ is lifted from the torus.
\end{definition}
\begin{remark}
The fact that each $T_{n}$ is $C^{1}$ follows from \cite[Corollary 2.4]{StabelErgodicity} (note that in \cite{StabelErgodicity} the maps $T_{n}$ are defined on $E_{L}^{c}$ by projecting $W_{F}^{c}(0)$ onto $E_{L}^{c}$).
\end{remark}
If $g:\mathbb{T}^{d}\to\mathbb{T}^{d}$ is homotopic to a hyperbolic matrix $A\in{\rm GL}(d,\mathbb{Z})$, then $f$ is semiconjugatd to $A$ \cite[Theorem 2.6.1]{KatokIntroDynSyst}. If $g$ is Anosov then the semiconjugacy is a conjugacy, the \textit{Franks-Manning conjugacy} \cite{Franks1969}. Since $f$ is not Anosov, and is not homotopic to a hyperbolic matrix there may not exists a semiconjugacy from $f$ to $L$. However, if a conjugacy exists it shares properties with the Franks-Manning conjugacy. In particular, if $H:\mathbb{T}^{d}\to\mathbb{T}^{d}$ is a semiconjugacy $H\circ f = L\circ H$, then $H$ maps $W_{f}^{\sigma}$, $\sigma = s,c,u,cs,cu$, to $W_{L}^{\sigma}$. The following two lemmas helps us to study accesibility classes of $f$. The first lemma is \cite[Lemma 6.1]{StabelErgodicity}, and the second follows since the semiconjugacy takes stable manifolds to stable manifolds and unstable manifolds to unstable manifolds.
\begin{lemma}
If $f$ is not accessible then $T_{n}:W_{F}^{c}(0)\to W_{F}^{c}(0)$ generate a $\mathbb{Z}^{d}-$action. That is $T_{n+m}(z) = T_{n}(T_{m}(z))$. In particular the map $T_{n}$ is a $C^{1}-$diffeomorphism for $n\in\mathbb{Z}^{d}$.
\end{lemma}
Recall that in Equation \ref{Eq:NotationForProjections} we defined $v^{\sigma}$ to be the projection of a vector $v\in\mathbb{R}^{d}$ into $E_{L}^{\sigma}$ for $\sigma = s,c,u,cs,cu,\chi$ where $\chi$ is some Lyapunov exponent for $L$.
\begin{lemma}
If $H:\mathbb{T}^{d}\to\mathbb{T}^{d}$ is a semiconjugacy $H\circ f = L\circ H$ chosen such that $H(0) = 0$ then $H|_{W_{F}^{c}(0)}:W_{F}^{c}(0)\to E_{L}^{c}$ and $H(T_{n}(z)) = H(z) + n^{c}$. Moreover
\begin{align*}
H(\pi_{x,y}^{\sigma}(z)) = H(z) + (y-x)^{\sigma}
\end{align*}
for $\sigma = s,u$.
\end{lemma}
\begin{remark}
Since the $\mathbb{Z}^{d}-$action $\rho^{n}:E_{L}^{c}\to E_{L}^{c}$, $\rho^{n}(x) = x + n^{c}$ is minimal it follows, in particular, that if $f$ is topologically conjugated to $L$ then the $\mathbb{Z}^{d}-$action $T_{n}(z)$ is minimal.
\end{remark}
We will be interested in a less general class of automorphisms.
\begin{definition}\label{Def:PropertyP}
Let $L\in{\rm GL}(d,\mathbb{Z})$ be an automorphism where $d\geq4$. We say that $L$ has property $(P)$ if
\begin{enumerate}[label = (\roman*)]
    \item the automorphism $L$ is irreducible,
    \item the automorphism $L$ has precisely one pair of complex eigenvalues on the unit circle $S^{1}$,
    \item all eigenvalues of $L$ that do not lie on $S^{1}$ are real.
\end{enumerate}
\end{definition}
One checks (see Lemma \ref{L:PropertiesOfP}) that $L$ with property $(P)$ has Lyapunov exponents 
\begin{align*}
{\rm Lyap}(L) = \{\chi_{1},...,\chi_{N},0,0,-\chi_{N},...,-\chi_{1}\}
\end{align*}
where $\chi_{1} > \chi_{2} > ... > \chi_{N}$ and $N = (d-2)/2$. In the remainder, we will always use this numbering of the Lyapunov exponents of $L$.
\begin{definition}\label{Def:SpredSpec}
If $L\in{\rm GL}(d,\mathbb{Z})$ has property $(P)$ then we say that $L$ has $r-$spread spectrum if $\chi_{j} > r\chi_{j+1}$ for $j = 1,...,N - 1$.
\end{definition}
\begin{remark}
In Lemma \ref{L:ExistenceSpreadSpectrum} we prove that for any even $d$ there is some $L\in{\rm GL}(d,\mathbb{Z})$ that has property $(P)$ and $r-$spread spectrum.
\end{remark}
Our interest in automorphisms having $r-$spread spectrum comes from the following lemma
\begin{lemma}\label{L:SmoothnesSpreadSpec}
Let $L\in{\rm GL}(d,\mathbb{Z})$ have property $(P)$ and $r-$spread spectrum. If $f\in{\rm Diff}^{\infty}(\mathbb{T}^{d})$ is $C^{1}-$close to $L$ then $f$ is dynamically coherent and the $2-$dimensional center foliation $W_{f}^{c}(x)$ has uniformly $C^{s}-$leaves where $s$ can be made as large as we want by letting $f$ be sufficiently close to $L$. Moreover, the $1-$dimensional leaves of $W_{f}^{\chi}$ are uniformly $C^{r}$. The same holds for $W_{f}^{\chi_{0} < \chi < \chi_{1}}$.
\end{lemma}
\begin{proof}
The lemma follows by repeated use of Theorem \ref{Thm:RegularityOfCenter}.
\end{proof}

\subsubsection{Centralizers of ergodic irreducible automorphisms}

Let $M$ be a closed smooth manifold and $f\in{\rm Diff}^{\infty}(M)$. We define the centralizer of $f$ to be the set of maps that commute with $f$. That is, we define
\begin{align*}
Z^{r}(f) = \{g\in{\rm Diff}^{r}(M)\text{ : }fg = gf\}.
\end{align*}
In the case when $M$ is a homogeneous space and $f$ is an affine map, we also define 
\begin{align*}
& Z_{{\rm aff}}(f) = \{g\text{ is affine }\text{ : }fg = gf\}, \\
& Z_{{\rm Aut}}(f) = \{g\text{ is an automorphism }\text{ : }fg = gf\}.
\end{align*}
Given two finitely generated groups $\Gamma$, $\Lambda$ and a homomorphism $\psi:\Lambda\to\Gamma$ we say that a property virtually holds if it holds up to finite index. In particular, $\psi$ is virtually injective if $\psi$ is injective on a finite index subgroup, and $\psi$ is virtually surjective if ${\rm Im}(\psi)$ has finite index in $\Gamma$. In the case that $\Lambda$ and $\Gamma$ are abelian it holds that $\Lambda$ is virtually isomorphic to $\mathbb{Z}^{\text{rank}(\Lambda)}$ and $\Gamma$ is virtually isomorphic to $\mathbb{Z}^{\text{rank}(\Gamma)}$ (here the $\text{rank}(\Gamma)$ can be defined, for examples, as the dimension of the $\mathbb{Q}-$vector space $\mathbb{Q}\otimes\Gamma$). As a consequence, we see that a homomorphism $\psi:\Lambda\to\Gamma$ between finitely generated abelian groups is a virtual isomorphism if and only if $\text{rank}(\Lambda) = \text{rank}(\Gamma)$ and $\psi$ is virtually injective. We can always define a subgroup $\{f^{n}\text{ : }n\in\mathbb{Z}\} = \langle f\rangle\subset Z^{r}(f)$ which is the \textit{trivial} part of the centralizer. We say that $f$ has a virtually trivial centralizer if the inclusion $\langle f\rangle\to Z^{\infty}(f)$ is a virtual isomorphism. In most cases (the exceptions being when $f^{n} = \text{id}$ for some $n > 0$) the property of having a virtually trivial centralizer is equivalent to $Z^{\infty}(f)$ being virtually isomorphic to $\mathbb{Z}$.

Let $L\in{\rm GL}(d,\mathbb{Z})$ be an automorphism of the torus. For irreducible automorphisms it is particularly simple to describe the centralizer of $L$. The following result is due to Adler and Palais \cite{ConjRigidity}.
\begin{theorem}\label{Thm:AdlerPalais}
If $L:\mathbb{T}^{d}\to\mathbb{T}^{d}$ is ergodic then $Z^{0}(L) = Z_{{\rm aff}}(L)$.
\end{theorem}
\begin{remark}
For any $r\geq0$ we have $Z_{{\rm aff}}(L)\subset Z^{r}(L)\subset Z^{0}(L)$, so this theorem reduces all questions about the centralizer to calculating $Z_{{\rm aff}}(L)$ which is an algebraic question.
\end{remark}
If we know that $L$ is irreducible then we can calculate $Z_{{\rm Aut}}(L)$. Since $Z_{{\rm Aut}}(L)$ has finite index in $Z_{{\rm aff}}(L)$, we can also calculate $Z_{{\rm aff}}(L)$ up to finite index. We denote by $r_{1}(L) = r_{1}$ the number of real eigenvalues of $L$ and by $r_{2}(L) = r_{2}$ the number of pairs of complex eigenvalues of $L$. We have the following result \cite{CentAut}.
\begin{theorem}\label{Thm:RankCentralizer}
Let $L\in{\rm GL}(d,\mathbb{Z})$ be an irreducible automorphism, then $Z_{{\rm Aut}}(L) = \mathbb{Z}^{r_{1} + r_{2} - 1}\times F$ where $F$ is some finite group.
\end{theorem}
\begin{remark}
By combining Theorems \ref{Thm:AdlerPalais} and \ref{Thm:RankCentralizer}, if $L$ is irreducible and ergodic then $Z^{r}(L)$ is virtually isomorphic to $\mathbb{Z}^{r_{1} + r_{2} - 1}$ for all $r\geq0$.
\end{remark}
\begin{definition}
Given an abelian subgroup $\Gamma\leq{\rm GL}(d,\mathbb{Z})$ there is a set, ${\rm Lyap}(\Gamma)$, of homomorphisms $\chi:\Gamma\to\mathbb{R}$ and a decomposition
\begin{align*}
\mathbb{R}^{d} = \bigoplus_{\chi\in{\rm Lyap}(\Gamma)}E^{\chi}
\end{align*}
such that $0\neq v\in E^{\chi}$ if and only if
\begin{align*}
\chi(\gamma) = \lim_{n\to\pm\infty}\frac{1}{n}\log\norm{\gamma^{n}v}.
\end{align*}
We call $\chi\in{\rm Lyap}(\Gamma)$ the Lyapunov functionals of $\Gamma$ and $E^{\chi}$ the Lyapunov subspaces.
\end{definition}
If $L$ is irreducible with precisely one pair of eigenvalues on $S^{1}$ and with no three eigenvalues of the same modulus then all eigenspaces correspond to eigenvalues of distinct moduli and therefore to distinct Lyapunov exponents. Indeed, if we have two distinct eigenvalues of the same modulus $\mu,\lambda\in\mathbb{C}$ then either $\overline{\lambda} = \mu$ or $\lambda = -\mu$ since otherwise $L$ would have more than three eigenvalues of modulus $|\mu|$. If $\lambda = -\mu$ then the characteristic polynomial $p_{L}(t)$ satisfy $p_{L}(-t) = p_{L}(t)$, so we can write $p_{L}(t) = Q(t^{2})$ for some polynomial $Q(x)$, which is a contradiction by Lemma \ref{L:PseudoIsAllTrue}. This implies the following lemma.
\begin{lemma}\label{L:InjectiveLyapunovExponentMap}
Let $L\in{\rm GL}(d,\mathbb{Z})$ be irreducible with $2-$dimensional isometric center and no three eigenvalues of the same modulus. For any abelian subgroup $\Gamma\leq Z_{{\rm Aut}}(L)$ that contains $L$ the natural map ${\rm Lyap}(\Gamma)\to{\rm Lyap}(L)$ defined by $\chi\mapsto\chi(L)$ is a bijection. That is, if $M$ commute with $L$ then we have a unique limit $\chi(M)$ such that for $0\neq v\in E_{L}^{\chi}$
\begin{align*}
\chi(M) = \lim_{n\to\infty}\frac{1}{n}\log\norm{M^{n}v}.
\end{align*}
\end{lemma}
For our purposes, the following lemma will be useful. The proof is contained in Appendix \ref{Sec:AlgebraicProperties}.
\begin{lemma}
Let $L\in{\rm GL}(d,\mathbb{Z})$ be irreducible with affine centralizer $Z_{{\rm aff}}(L)$. For any subgroup $\Gamma\leq Z_{{\rm aff}}(L)$ we define
\begin{align*}
\Lambda:\Gamma\to\mathbb{R}^{r_{1} + r_{2} - 1},\quad\Lambda(\gamma) = (\chi_{1}(\gamma),...,\chi_{r_{1} + r_{2}-1}(\gamma))
\end{align*}
where each $\chi_{j}:\Gamma\to\mathbb{R}$ is the Lyapunov functional along some eigendirection for the action of $\Gamma$. If $\Gamma\leq Z_{{\rm Aut}}(L)$ has rank $r_{1} + r_{2} - 1$ then the image of $\Lambda$ is a lattice. Moreover, if $\chi_{j}:\Gamma\to\mathbb{R}$ is a Lyapunov exponent that corresponds to a real eigenvalue of the action of $\Gamma$ and $\gamma\in\Gamma$ satisfies $\chi_{j}(\gamma) = 0$ then $\gamma = \pm I$.
\end{lemma}

\subsection{Higher rank actions}

Let $\Gamma\cong\mathbb{Z}^{k}$, $k\geq 2$, be some abelian group and $M$ a smooth closed manifold. Given a smooth action $\alpha:\Gamma\times M\to M$ of $\Gamma$ we make the following definition \cite{RodriguezHertzPeriodicPoints}.
\begin{definition}
A point $p\in M$ is periodic for $\alpha$ if the stabilizer $\Gamma_{p} = \{\gamma\text{ : }\alpha(\gamma)p = p\}$ has finite index in $\Gamma$.
\end{definition}
Given a periodic point $p\in M$ for $\alpha$ we can define the Lyapunov functionals at $p$ by
\begin{align*}
\chi_{p}(v,\gamma) = \lim_{n\to\infty}\frac{1}{n}\log\norm{D_{p}\gamma^{n}(v)},\quad v\in T_{p}M.
\end{align*}
There are finitely many $\chi_{p}^{1},...,\chi_{p}^{\ell}:\Gamma\to\mathbb{R}$ such that for every $v\in T_{p}M$ there is some $j$ such that $\chi_{p}(v,\gamma) = \chi_{p}^{j}(\gamma)$ for all $\gamma\in\Gamma$. We say that $\{\chi_{p}^{1},...,\chi_{p}^{\ell}\}$ are the Lyapunov functionals at $p$. If $\alpha$ preserves a foliation $\mathcal{F}$ then we can consider the Lyapunov functionals along $\mathcal{F}$ which consists of those $\chi_{p}^{j}$ such that $\chi_{p}^{j}(\gamma) = \chi_{p}(v,\gamma)$ for some $v\in T_{p}\mathcal{F}$.

In the case of the torus, we can also define \textit{higher rank actions}. We begin by defining a linear higher rank action. An abelian subgroup $\Gamma\leq{\rm GL}(d,\mathbb{Z})$ has a natural action on $\mathbb{T}^{d}$, written $\alpha_{0}$, given by matrix multiplication. A $\mathbb{Z}-$factor, or a rank $1$ factor, of $\alpha_{0}$ is some projection $\pi:\mathbb{T}^{d}\to T$ to a torus and an automorphism $A:T\to T$ such that $\pi\circ\alpha_{0}(\gamma) = A^{n(\gamma)}\circ\pi$ for all $\gamma\in\Gamma$. That is, $\alpha_{0}$ has a $\mathbb{Z}-$factor if it factors through an action of $\mathbb{Z}$. The following standard definition is used in \cite{HigherRankRigidity,FisherKalininSpatzier}.
\begin{definition}\label{Def:HigherRankLinear}
The natural action of some abelian subgroup $\Gamma\leq{\rm GL}(d,\mathbb{Z})$ is of higher rank if $\Gamma$ has no non-trivial $\mathbb{Z}-$factor.
\end{definition}
The following result follows since irreducible automorphisms can not preserve non-trivial proper rational subspaces.
\begin{theorem}\label{Thm:AllSubgroupsHigerRank}
If $L\in{\rm GL}(d,\mathbb{Z})$ is irreducible then any subgroup $\Gamma\leq Z_{{\rm Aut}}(L)$ with $L\in\Gamma$ and ${\rm rank}(\Gamma) > 1$ is higher rank.
\end{theorem}
Given any smooth action $\alpha:\Gamma\times\mathbb{T}^{d}\to\mathbb{T}^{d}$ we denote by $\Gamma\to\rm Diff^{\infty}(\mathbb{T}^{d})$ the associated homomorphism. We define the linearization of $\alpha$ to be the homomorphism $\Gamma\to\rm Diff^{\infty}(\mathbb{T}^{d})\to{\rm GL}(d,\mathbb{Z})$ where the last map is the homology representation. We make the following definition following \cite{HigherRankRigidity}.
\begin{definition}\label{Def:HigherRank}
We say that a smooth action $\alpha:\Gamma\times\mathbb{T}^{d}\to\mathbb{T}^{d}$ is of higher rank if its linearization is of higher rank.
\end{definition}
The following result was shown in \cite{HigherRankRigidity}
\begin{theorem}\label{Thm:RigidityOfHigherRankAnosov}
If $\alpha:\Gamma\times\mathbb{T}^{d}\to\mathbb{T}^{d}$ is of higher rank and there is some $\gamma\in\Gamma$ such that $\alpha(\gamma)$ is Anosov, then $\alpha$ is $C^{\infty}-$conjugated to its linearization.
\end{theorem}

%% file: FakeConjugacy.tex
\section{Coordinates along invariant foliations}
\label{Sec:FranksManninCoordinates}

In this section we show the following result for perturbations.
\begin{claim}\label{Claim:ExistenceFMcoordinates}
Let $L\in{\rm GL}(d,\mathbb{Z})$ be any automorphism of the torus and let $f\in{\rm Diff}^{\infty}(\mathbb{T}^{d})$ be $C^{1}-$close to $L$. For each non-zero Lyapunov exponent $\chi$ of $L$ there exist a unique Hölder continuous $\Phi_{\chi}:\mathbb{R}^{d}\to E_{L}^{\chi}$ such that 
\begin{align*}
& \Phi_{\chi}(Fx) = L\Phi_{\chi}(x), \\
& \Phi_{\chi}(x + n) = \Phi_{\chi}(x) + n^{\chi},\quad n\in\mathbb{Z}^{d}.
\end{align*}
We define $\Phi_{\chi,x}:W_{F}^{\chi}(x)\to E_{L}^{\chi}$ by $\Phi_{\chi,x}(z) = \Phi_{\chi}|_{W_{F}^{\chi}(x)}(z) - \Phi_{\chi}(x)$, then $\Phi_{\chi,x}$ is a bi-Hölder homeomorphism with Hölder exponent close to $1$ and Hölder norm uniformly bounded. The map $x\mapsto\Phi_{\chi,x}$ is continuous and well-defined on the torus in the sense that $\Phi_{\chi,x+n} = \Phi_{\chi,x}$ after we identify $W_{F}^{\chi}(x)$ and $W_{F}^{\chi}(x+n)$. Finally, $\Phi_{\chi,x}(x) = 0$ and for any $g\in Z^{\infty}(f)$ homotopic to $M$
\begin{align}\label{Eq:FunctionalEquation}
\Phi_{\chi,gx}(gy) = M\Phi_{\chi,x}(y).
\end{align}
\end{claim}
\begin{remark}
The uniqueness statement for $\Phi_{\chi}$ holds once we have fixed a lift $F$ of $f$. However, the two defining properties of $\Phi_{\chi}$ immediately implies that two different lifts for $f$ will define functions $\Phi_{\chi}$ that differ by a constant vector in $E_{L}^{\chi}$ since any two lifts differ by an integer vector.
\end{remark}
\begin{proof}\renewcommand{\qedsymbol}{}
One notes that $\Phi_{\chi}(x + n) = \Phi_{\chi}(x) + n^{\chi}$ is equivalent to the fact that we can write $\Phi_{\chi}$ as $\Phi_{\chi}(x) = x^{\chi} + \varphi_{\chi}(x)$ where $\varphi_{\chi}$ is lifted from $\mathbb{T}^{d}$. Writing $(Fx)^{\chi} = Lx^{\chi} + v_{\chi}(x)$ with $v_{\chi}:\mathbb{T}^{d}\to E_{L}^{\chi}$ we obtain
\begin{align*}
\Phi_{\chi}(Fx) = (Fx)^{\chi} + \varphi_{\chi}(Fx) = Lx^{\chi} + v_{\chi}(x) + \varphi_{\chi}(fx) = L(x^{\chi} + \varphi_{\chi}(x)).
\end{align*}
Equivalently, $v_{\chi}(x) = L\varphi_{\chi}(x) - \varphi_{\chi}(fx)$. Since $L$ is hyperbolic along $E_{L}^{\chi}$, this is a hyperbolically twisted cohomological equation and therefore has a unique Hölder solution $\varphi_{\chi}$. Since $f$ is $C^{1}-$close to $L$ the expansion of $f$ along $W_{f}^{\chi}$ is close to that of $L$ along $W_{L}^{\chi}$ so $\varphi_{\chi}$ has Hölder exponent close to $1$ and Hölder norm close to $0$ along $W_{f}^{\chi}$. It follows that $\Phi_{\chi}|_{W_{F}^{\chi}(x)}$ is Hölder with Hölder exponent and Hölder norm close to $1$. Since $\varphi_{\chi}$ is unique, $\Phi_{\chi}$ is unique. Let $g\in{Z^{\infty}(f)}$ be homotopic to $g_{*} = M\in{\rm GL}(d,\mathbb{Z})$ (recall that we can identify $M$ with the induced map on the first homology group, $g_{*}$). Since $L$, also, coincides with the induced map on the first homology group for $f$, $f_{*}$, we have 
\begin{align*}
ML = g_{*}f_{*} = (gf)_{*} = (fg)_{*} = f_{*}g_{*} = LM.
\end{align*}
That is, $M\in Z_{\rm Aut}(L)$. We choose a lift $G$ of $g$ and define $m\in\mathbb{Z}^{d}$ by
\begin{align*}
m = GFx - FGx.
\end{align*}
Let $\Tilde{\Phi}_{\chi}:\mathbb{R}^{d}\to E_{L}^{\chi}$ be defined by
\begin{align*}
\Tilde{\Phi}_{\chi}(x) = M^{-1}\Phi_{\chi}(Gx) - v,\quad v\in E_{L}^{\chi}
\end{align*}
for some $v$ to be decided. It holds that $\Tilde{\Phi}_{\chi}(x + n) = \Tilde{\Phi}_{\chi}(x) + n^{\chi}$ since $G(x + n) = G(x) + Mn$ and $M$ preserves $E_{L}^{\chi}$. Moreover, we have a relation
\begin{align*}
\Tilde{\Phi}_{\chi}(Fx) = & M^{-1}\Phi_{\chi}(GFx) - v = M^{-1}\Phi_{\chi}(FGx + m) - v = \\ &
M^{-1}\Phi_{\chi}(FGx) + M^{-1}m^{\chi} - v = L\left(\Tilde{\Phi}_{\chi}(x) + v\right) + M^{-1}m^{\chi} - v.
\end{align*}
If we choose $v$ to solve $Lv - v + M^{-1}m^{\chi} = 0$ (which can be done since $L - I$ is invertible) we obtain
\begin{align*}
\Tilde{\Phi}_{\chi}(Fx) = L\Tilde{\Phi}_{\chi}(x).
\end{align*}
Uniqueness of $\Phi_{\chi}(x)$ implies that $\Tilde{\Phi}_{\chi}(x) = \Phi_{\chi}(x)$ or
\begin{align*}
\Phi_{\chi}(Gx) = M\Phi_{\chi}(x) + Mv.
\end{align*}
For $y\in W_{F}^{\chi}(x)$ we obtain
\begin{align*}
\Phi_{\chi,Gx}(Gy) = & \Phi_{\chi}(Gy) - \Phi_{\chi}(Gx) = M\Phi_{\chi}(y) + Mv - \left[M\Phi_{\chi}(x) + Mv\right] = \\ &
M\left(\Phi_{\chi}(y) - \Phi_{\chi}(x)\right) = M\Phi_{\chi,x}(y)
\end{align*}
which proves Equation \ref{Eq:FunctionalEquation}.
\end{proof}
Before proceeding we show a lemma that allows us to compare lengths on $W_{F}^{\chi}$ with lengths in $\mathbb{R}^{d}$.
\begin{lemma}\label{L:QuasiIsometryForPert}
For $p,q\in W_{F}^{\chi}(x)$ we have
\begin{align*}
\intd(p,q)\leq\intd_{\chi}(p,q)\leq C\cdot\intd(p,q)
\end{align*}
for some $C\geq1$, that can be chosen close to $1$ by letting $f$ be sufficiently close to $L$. Moreover, we have
\begin{align*}
\intd(p,q)\leq c\norm{(p-q)^{\chi}}
\end{align*}
where $(p-q)^{\chi}$ is the projection of $p-q$ onto $E_{L}^{\chi}$.
\end{lemma}
\begin{proof}
The lemma is an immediate consequence of the descriptions of $W_{F}^{\chi}$ as graphs and Lemma \ref{L:InvMfdAreGraph}.
\end{proof}
We can now finish the proof of the claim.
\begin{proof}[Proof of Claim \ref{Claim:ExistenceFMcoordinates}]
We will assume $\chi > 0$. For the other case, reverse time. Let $q\in W_{F}^{\chi}(p)$ satisfy $\intd_{\chi}(p,q)\leq 1$ and let $0 < \varepsilon < 0.01\cdot\chi$ be some small number. We calculate with Lemma \ref{L:QuasiIsometryForPert}
\begin{align*}
\norm{\Phi_{\chi}(p) - \Phi_{\chi}(q)} = & \norm{L^{-n}\Phi_{\chi}(F^{n}p) - L^{-n}\Phi_{\chi}(F^{n}q)}\geq \\ &
Ce^{-n(\chi + \varepsilon)}\norm{\Phi_{\chi}(F^{n}p) - \Phi_{\chi}(F^{n}q)}\geq \\ &
Ce^{-n(\chi + \varepsilon)}\left(\norm{(F^{n}p)^{\chi} - (F^{n}q)^{\chi}} - 2\norm{\varphi_{\chi}}_{C^{0}}\right)\geq \\ &
Ce^{-n(\chi + \varepsilon)}\left(c\norm{F^{n}p - F^{n}q} - 2\norm{\varphi_{\chi}}_{C^{0}}\right)\geq \\ &
C'e^{-n(\chi + \varepsilon)}\left(c\intd_{\chi}(F^{n}p,F^{n}q) - 2\norm{\varphi_{\chi}}_{C^{0}}\right)
\end{align*}
and since $f$ is $C^{1}-$close to $L$ and $E_{f}^{\chi}$ is $C^{0}-$close to $E_{L}^{\chi}$ we have $\intd_{\chi}(F^{n}p,F^{n}q)\geq C\cdot e^{n(\chi - \varepsilon)}\intd_{\chi}(p,q)$. It follows that we have an estimate
\begin{align}\label{Eq:HölderEstimateInverse1}
\norm{\Phi_{\chi}(p) - \Phi_{\chi}(q)}\geq Ce^{-n(\chi + \varepsilon)}\left(ce^{n(\chi - \varepsilon)}\intd_{\chi}(p,q) - 2\norm{\varphi_{\chi}}_{C^{0}}\right).
\end{align}
We find some constant $C$ such that if
\begin{align*}
n = \left\lfloor-\frac{\log\intd_{\chi}(p,q)}{\chi - \varepsilon} + C\right\rfloor + 1
\end{align*}
then $ce^{n(\chi - \varepsilon)}\intd_{\chi}(p,q) - 2\norm{\varphi_{\chi}}_{C^{0}}\geq1$. With this $n$ in equation \ref{Eq:HölderEstimateInverse1} we find some constant $K > 0$ such that
\begin{align}\label{Eq:HölderEstimateInverse2}
\norm{\Phi_{\chi}(p) - \Phi_{\chi}(q)}\geq K\cdot\intd_{\chi}(p,q)^{\frac{\chi + \varepsilon}{\chi - \varepsilon}},\quad\intd_{\chi}(p,q)\leq 1.
\end{align}
Define $\Phi_{\chi,x}:W_{F}^{\chi}(x)\to E_{L}^{\chi}$, $\Phi_{\chi,x}(z) = \Phi_{\chi}(z) - \Phi_{\chi}(x)$ as in the claim. We have
\begin{align*}
\Phi_{\chi,x+n}(z + n) = & \Phi_{\chi}(z+n) - \Phi_{\chi}(x + n) = \\ &
\left[\Phi_{\chi}(z) + n^{\chi}\right] - \left[\Phi_{\chi}(x) + n^{\chi}\right] = \Phi_{\chi,x}(z)
\end{align*}
which proves that $\Phi_{\chi,x}$ is defined for $x\in\mathbb{T}^{d}$. Moreover, it is immediate from the definition that $x\mapsto\Phi_{\chi,x}$ is a continuous map. From equation \ref{Eq:HölderEstimateInverse2} we have
\begin{align*}
\norm{\Phi_{\chi,x}(p) - \Phi_{\chi,x}(q)}\geq K\cdot\intd(p,q)^{\frac{\chi + \varepsilon}{\chi - \varepsilon}},\quad\intd_{\chi}(p,q)\leq 1
\end{align*}
which, after iterating the functional equation $\Phi_{\chi,Fx}(Fy) = L\Phi_{\chi,x}(y)$, proves that $\Phi_{\chi,x}$ is injective. Since $\Phi_{\chi,x}$ is injective, it follows by invariance of domain that $\Phi_{\chi,x}$ has open image and is a homeomorphism onto its image. From the functional equation and the fact that $L$ is expanding along $E_{L}^{\chi}$, it follows that $\Phi_{\chi,x}$ is surjective. So $\Phi_{\chi,x}$ is a homeomorphism. Since $\Phi_{\chi,x}(x) = 0$ and $x\mapsto\Phi_{\chi,x}$ is defined on $\mathbb{T}^{d}$ we find some $0 < R < \infty$ such that for $v,w\in E_{L}^{\chi}$ if $\norm{v},\norm{w}\leq R$ then $\intd_{\chi}(\Phi_{\chi,x}^{-1}(v),\Phi_{\chi,x}^{-1}(w))\leq 1$. Using equation \ref{Eq:HölderEstimateInverse2} with $\Phi_{\chi,x}^{-1}(v),\Phi_{\chi,x}^{-1}(w)$ where $v,w\in E_{L}^{\chi}$, $\norm{v},\norm{w}\leq R$, we obtain
\begin{align*}
\intd_{\chi}(\Phi_{\chi,x}^{-1}(v),\Phi_{\chi,x}^{-1}(w))\leq k\cdot\norm{v - w}^{\alpha},\quad\alpha = \frac{\chi - \varepsilon}{\chi + \varepsilon}.
\end{align*}
So $\Phi_{\chi,x}^{-1}$ is uniformly Hölder on the $R-$ball in $E_{L}^{\chi}$. In fact, $\Phi_{\chi,x}^{-1}$ is uniformly Hölder on any ball of radius $R$, not necessarily the $R-$ball about $0$ in $E_{L}^{\chi}$. This follows from the following 
\begin{align*}
\Phi_{\chi}(\Phi_{\chi,x}^{-1}(v)) = \Phi_{\chi}(\Phi_{\chi,x}^{-1}(v)) - \Phi_{\chi}(x) + \Phi_{\chi}(x) = v + \Phi_{\chi}(x).
\end{align*}
If $y\in W_{F}^{\chi}(x)$ then we have $\Phi_{\chi,x}^{-1}(v) = \Phi_{\chi,y}^{-1}(v + \Phi_{\chi}(x) - \Phi_{\chi}(y))$. So, we have shown that $\Phi_{\chi,x}^{-1}$ is uniformly Hölder on any ball $B_{R}(v)$ in $E_{L}^{\chi}$.
\end{proof}
We obtain some immediate consequences of Claim \ref{Claim:ExistenceFMcoordinates}.
\begin{lemma}\label{L:PropLyapFunc}
Let $L\in{\rm GL}(d,\mathbb{Z})$ be irreducible with $2-$dimensional isometric center and no three eigenvalues of the same modulus. Let $f\in{\rm Diff}^{\infty}(\mathbb{T}^{d})$ be $C^{1}-$close to $L$, $g\in Z^{\infty}(f)$ homotopic to $M$, $\chi(L)\neq0$. Then there is some small $\varepsilon > 0$ (that can be chosen arbitrarily small by letting $f$ be close to $L$) and $C\geq 1$, depending on how $C^{1}-$close $f$ is to $L$, such that
\begin{align*}
\frac{1}{C}\cdot e^{n(1-\varepsilon)\chi(M)}\leq\norm{D_{x}g^{n}|_{E_{f}^{\chi}}}\leq C\cdot e^{n(1+\varepsilon)\chi(M)},\quad\chi(M)\geq0
\end{align*}
and similarly for $\chi(M) < 0$.
\end{lemma}
\begin{proof}
By Lemma \ref{L:InjectiveLyapunovExponentMap} the Lyapunov exponent $\chi$ of $L$ is also a Lyapunov exponent for the image $\Gamma\subset Z_{\rm Aut}(L)$ of $Z^{\infty}(f)$ in ${\rm GL}(d,\mathbb{Z})$. We assume without loss of generality that $\chi(M) < 0$, the case $\chi(M) > 0$ follows by considering $g^{-1}$. Let $\beta$ be a joint Hölder exponent for $\Phi_{\chi,x}$ and $\Phi_{\chi,x}^{-1}$. Given $y\in W_{F}^{\chi}(x)$ with $\intd_{\chi}(x,y)\leq 1$ and $n\geq 0$ we have
\begin{align*}
\intd_{\chi}(g^{n}(x),g^{n}(y))\leq & C\cdot\norm{\Phi_{\chi,g^{n}(x)}(g^{n}x) - \Phi_{\chi,g^{n}(x)}(g^{n}y)}^{\beta} = \\ &
C\cdot\norm{M^{n}\Phi_{\chi,x}(x) - M^{n}\Phi_{\chi,x}(y)}^{\beta} = \\ &
Ce^{n\cdot\chi(M)\beta}\cdot\norm{\Phi_{\chi,x}(x) - \Phi_{\chi,x}(y)}^{\beta}
\end{align*}
for some $C$ independent of $x,y$ and $n$. Similarly, we have
\begin{align*}
e^{n\chi(M)}\cdot\norm{\Phi_{\chi,x}(x) - \Phi_{\chi,x}(y)} = & \norm{\Phi_{\chi,g^{n}x}(g^{n}x) - \Phi_{\chi,g^{n}x}(g^{n}y)}\leq \\ &
C\intd_{\chi}(g^{n}(x),g^{n}(y))^{\beta}
\end{align*}
and combined
\begin{align*}
& c_{1}e^{n\chi(M)/\beta}\cdot\norm{\Phi_{\chi,x}(x) - \Phi_{\chi,x}(y)}^{1/\beta}\leq\intd_{\chi}(g^{n}x,g^{n}y)\leq \\ &
c_{2}e^{n\chi(M)\beta}\cdot\norm{\Phi_{\chi,x}(x) - \Phi_{\chi,x}(y)}^{\beta}.
\end{align*}
We can now apply \cite[Lemma 25]{CentralizerRigidity} and the lemma follows since we can choose $\beta$ as close to $1$ as we want by letting $f$ be sufficiently $C^{1}-$close to $L$.
\end{proof}
Assume now that $L$ has no three eigenvalues of the same modulus. For a $Z^{\infty}(f) = \Gamma-$periodic point $p\in\mathbb{T}^{d}$, let $\chi_{p}:\Gamma\to\mathbb{R}$ be a Lyapunnov exponent of the derivative cocycle $D_{p}\Gamma|_{E_{f}^{\chi}}$ at $p$, restricted to $E_{f}^{\chi}(p)$. If $g\in\Gamma$ is homotopic to $M_{g}$, then Lemma \ref{L:PropLyapFunc} implies that
\begin{align}
(1-\varepsilon)\chi(M_{g})\leq\chi_{p}(g)\leq(1+\varepsilon)\chi(M_{g}).
\end{align}
In particular, $\chi_{p}(g)$ is positive if and only if $\chi(M_{g})$ is positive. Using Theorem \ref{Thm:InjectivityOfHom} (Theorem \ref{Thm:InjectivityOfHom} is proved in Section \ref{Sec:InjectivityOfHom} and guarantees that $Z^{\infty}(f)$ is virtually abelian) the following lemma follows.
\begin{lemma}\label{L:LyapFuncPosProp}
Let $L\in{\rm GL}(d,\mathbb{Z})$ be irreducible with $2-$dimensional isometric center and no three eigenvalues of the same modulus. Let $f\in{\rm Diff}^{\infty}(\mathbb{T}^{d})$ be $C^{1}-$close to $L$ and let $\Gamma = Z^{\infty}(f)$. If $p$ is a periodic point for $Z^{\infty}(f)$ then the Lyapunov exponents along $E_{f}^{\chi}$ at $p$ are positively proportional to $\chi$.
\end{lemma}
\begin{lemma}\label{L:VolumeExpansionEst}
Let $L\in{\rm GL}(d,\mathbb{Z})$ be irreducible with $2-$dimensional isometric center and no three eigenvalues of the same modulus. Let $f\in{\rm Diff}_{\rm vol}^{\infty}(\mathbb{T}^{d})$ be $C^{1}-$close to $L$ and let $g\in Z^{\infty}(f)$ be homotopic to $M$. Then there is some small $\varepsilon > 0$, that can be chosen arbitrarily close to $0$ by letting $f$ be sufficiently $C^{1}-$close to $L$, and constant $c > 0$ such that
\begin{align*}
|\det(D_{x}g^{n}|_{E_{f}^{c}})|\geq & c\cdot\det(M^{n}|_{E_{L}^{c}})^{1-\varepsilon}\cdot\left[\prod_{\substack{\chi(M) > 0 \\ \chi(L)\neq0}}\det(M^{n}|_{E_{L}^{\chi}})\right]^{-2\varepsilon}.
\end{align*}
\end{lemma}
\begin{proof}
We begin by noting that $g\in Z^{\infty}(f)$ preserves volume. This would be immediate if we could apply the main result from \cite{StabelErgodicity}, but we can not since we are only perturbing in the $C^{1}-$topology. Instead, let $u(x)$ be any continuous $f-$invariant function. Given $c\in\text{Im}(u)$ we define the closed set $X_{c} = u^{-1}(c) = \{x\text{ : }u(x) = c\}$. Since $u$ is $f-$invariant $fX_{c} = X_{c}$. Moreover, $f-$invariance implies that $u$ is constant on the $su-$saturation of any $x\in\mathbb{T}^{d}$. It follows that $X_{c}$ is $su-$saturated. By Theorem \ref{Thm:RHsuMinimality} the level sets $X_{c}$ are dense. But $X_{c}$ is also closed which shows $X_{c} = \mathbb{T}^{d}$. So $u$ is constant. Now let $\sigma:\mathbb{T}^{d}\to\mathbb{R}$ be the Jacobian of $g$, $(g^{*}\text{vol})_{x} = \sigma(x)\text{vol}_{x}$. Then $\sigma$ is $f-$invariant since $f$ preserves volume. It follows that $\sigma$ is constant, which implies that $g$ is volume preserving.

Volume preservation of $g$ implies that we have
\begin{align*}
\det(D_{x}g^{n}|_{E_{f}^{c}})\cdot\prod_{\chi(L)\neq0}\det(D_{x}g^{n}|_{E_{f}^{\chi}}) = 1
\end{align*}
or 
\begin{align*}
|\det(D_{x}g^{n}|_{E_{f}^{c}})| = \left[\prod_{\chi(L)\neq0}|\det(D_{x}g^{n}|_{E_{f}^{\chi}})|\right]^{-1}.
\end{align*}
However, Lemma \ref{L:PropLyapFunc} implies that for $\chi(M)\geq0$
\begin{align*}
|\det(D_{x}g^{n}|_{E_{f}^{\chi}})|\leq & C\cdot e^{n(1+\varepsilon)\chi(M)\cdot\dim(E_{f}^{\chi})} =
C\cdot \left[e^{n\chi(M)\dim(E_{f}^{\chi})}\right]^{1+\varepsilon}= \\ &
C|\det(M|_{E_{L}^{\chi}})|^{n(1+\varepsilon)}
\end{align*}
and similarly for $\chi(M) < 0$. For $\chi(L)\neq0$
\begin{align*}
& |\det(D_{x}g^{n}|_{E_{f}^{\chi}})|\leq C|\det(M|_{E_{L}^{\chi}})|^{n(1+\varepsilon)},\quad \chi(M) > 0, \\
& |\det(D_{x}g^{n}|_{E_{f}^{\chi}})|\leq C|\det(M|_{E_{L}^{\chi}})|^{n(1-\varepsilon)},\quad \chi(M) < 0.
\end{align*}
It follows that we have some $c > 0$ such that
\begin{align*}
|\det(D_{x}g^{n}|_{E_{f}^{c}})|\geq & c\left[\frac{1}{\prod_{\substack{\chi_{M} > 0 \\ \chi(L)\neq0}}|\det(M^{n}|_{E_{L}^{\chi}})|}\right]^{1+\varepsilon}\cdot \\ &
\left[\frac{1}{\prod_{\substack{\chi_{M} < 0 \\ \chi(L)\neq0}}|\det(M^{n}|_{E_{L}^{\chi}})|}\right]^{1-\varepsilon}
\end{align*}
or since $|\det(M)| = 1$
\begin{align*}
\prod_{\substack{\chi(M) < 0 \\ \chi(L)\neq0}}|\det(M^{n}|_{E_{L}^{\chi}})| = \frac{1}{|\det(M^{n}|_{E_{L}^{c}})|}\cdot\frac{1}{\prod_{\substack{\chi(M) > 0 \\ \chi(L)\neq0}}|\det(M^{n}|_{E_{L}^{\chi}})|}
\end{align*}
which implies
\begin{align*}
|\det(D_{x}g^{n}|_{E_{f}^{c}})|\geq & c|\det(M^{n}|_{E_{L}^{c}})|^{1-\varepsilon}\cdot\left[\prod_{\substack{\chi(M) > 0 \\ \chi(L)\neq0}}|\det(M^{n}|_{E_{L}^{\chi}})|\right]^{-2\varepsilon}
\end{align*}
proving the lemma.
\end{proof}

%% file: HomologyRepInj.tex
\section{Elements in the centralizer homotopic to identity}
\label{Sec:InjectivityOfHom}

Let $L\in\text{GL}(d,\mathbb{Z})$ be irreducible with $2-$dimensional isometric center. In this section we prove Theorem \ref{Thm:InjectivityOfHom}. Let $S_{N}$ be the symmetric group on $N$ symbols.
\begin{lemma}\label{L:SymmetricGroupMap}
Let $f$ be a $C^{1}-$small perturbation of $L$. Let $N = |\det(L-I)|$ be the number of fixed points of $L$. There is a natural homomorphism $\Pi:Z^{0}(f)\to S_{N}$ such that the kernel of $\Pi$ consists precisely of those $g\in Z^{0}(f)$ which fix all the fixed points of $f$.
\end{lemma}
\begin{proof}
The graph $G(L) = \{(x,Lx)\text{ : }x\in\mathbb{T}^{d}\}$ is transverse to the diagonal $\Delta = \{(x,x)\text{ : }x\in\mathbb{T}^{d}\}$ since $L$ is ergodic. The graph of $f$ is $C^{1}-$close to $G(L)$ and it follows by transversality that $G(f)$ is transverse to $\Delta$ and $\#G(L)\cap\Delta = \#G(f)\cap\Delta$. But $G(L)\cap\Delta$ and $G(f)\cap\Delta$ are the fixed points of $L$ and $f$ respectively. So $f$ and $L$ have the same number of fixed points, and the map $L$ has $|\det(L-I)|$ fixed points \cite{KatokIntroDynSyst}.

We denote the fixed points of $f$ by $\{x_{1},...,x_{N}\}$. If $g\in Z^{0}(f)$, then for any $x_{j}\in\{x_{1},...,x_{N}\}$ we have $fgx_{j} = gfx_{j} = gx_{j}$ so $gx_{j}$ is fixed by $f$. We define $\Pi(g)\in S_{N}$ by $gx_{j} = x_{\Pi(g)(j)}$. One sees
\begin{align*}
x_{\Pi(g')\Pi(g)(j)} = g'x_{\Pi(g)(j)} = g'(gx_{j}) = (g'\circ g)(x_{j}) = x_{\Pi(g'\circ g)(j)}
\end{align*}
so $\Pi$ is a homomorphism. It is immediate from the definition that $g\in Z^{0}(f)$ fixes the fixed points of $f$ if $\Pi(g) = \text{id}$.
\end{proof}
\begin{proof}[Proof of Theorem \ref{Thm:InjectivityOfHom}]
Let $f$ be as in the theorem and let $\pi_{f}:Z^{1}(f)\to{\rm GL}(d,\mathbb{Z})$ be the homology representation. We note that $\ker\pi_{f}$ are precisely the elements $g\in Z^{1}(f)$ homotopic to identity. Let $\Pi:Z^{1}(f)\to S_{N}$ be the map from Lemma \ref{L:SymmetricGroupMap}. We claim that the map $\Pi:\ker\pi_{f}\to S_{N}$ is injective, which proves the theorem since $S_{N}$ is a finite group. To show this, pick $g\in\ker\pi_{f}$. In particular it holds that
\begin{align*}
{\rm Fix}(g) = \{x\text{ : }gx = x\}\neq\emptyset.
\end{align*}
Let $x_{0}\in{\rm Fix}(g)$. By Claim \ref{Claim:ExistenceFMcoordinates} we have for $\sigma = s,u$ and $x\in W_{f}^{\sigma}(x_{0})$
\begin{align*}
\Phi_{\sigma,x_{0}}(g(x)) = \Phi_{\sigma,g(x_{0})}(g(x)) = {\rm Id}\cdot\Phi_{\sigma,x_{0}}(x) = \Phi_{\sigma,x_{0}}(x).
\end{align*}
Since $\Phi_{\sigma,x_{0}}:W_{F}^{\sigma}(x_{0})\to E_{L}^{\sigma}$ is a homeomorphism, Claim \ref{Claim:ExistenceFMcoordinates}, it follows that $gx = x$ for $x\in W_{F}^{\sigma}(x_{0})$. That is, ${\rm Fix}(g)$ is $\sigma-$saturated. Since this holds for $\sigma = s$ and $\sigma = u$ we conclude that ${\rm Fix}(g)$ is closed, $su-$saturated and non-empty. Theorem \ref{Thm:InjectivityOfHom} now follows from Theorem \ref{Thm:RHsuMinimality}.
\end{proof}

%% file: TopConjForHypCenter.tex
\section{Smooth conjugacies for perturbations with large centralizer}
\label{Sec:FullCentralizer}

\subsection{Perturbations with virtually full centralizer}

In this section, we prove Theorem \ref{Thm:FullCentralizer}. For the remainder of this section, let $L\in{\rm GL}(d,\mathbb{Z})$ be irreducible with $2-$dimensional isometric center and no three eigenvalues of the same modulus. Let $f\in{\rm Diff_{vol}^{\infty}}(\mathbb{T}^{d})$ be $C^{1}-$close to $L$. Recall that we assume, without loss of generality, that $f(0) = 0$.
\begin{lemma}\label{L:AlgCenterDomination}
If $\Gamma\leq Z_{\rm Aut}(L)$ has full rank, ${\rm rank}(\Gamma) = {\rm rank}(Z_{\rm Aut}(L))$, then for any $Q > 0$ there is $\gamma\in\Gamma$ satisfying
\begin{align}\label{Eq:CenterDomination}
\left|\det(\gamma|_{E_{L}^{c}})\right|\geq \left|\frac{\det(\gamma|_{E_{\gamma}^{u}})}{\det(\gamma|_{E_{L}^{c}})}\right|^{Q}
\end{align}
\end{lemma}
\begin{proof}
Let $\chi_{1},...,\chi_{n}$, $n = r_{1} + r_{2}$, be the Lyapunov exponents of $\Gamma$ such that $\chi_{1}(\gamma)$ is the Lyapunov exponent of $\gamma$ along $E_{L}^{c}$. The map $\Lambda(\gamma) = (\chi_{2}(\gamma),...,\chi_{n}(\gamma))$ maps $\Gamma$ onto a lattice in $\mathbb{R}^{n-1}$ by Lemma \ref{L:AlgebraicFullCentralizer}. The set
\begin{align*}
U := \{(x_{2},...,x_{n})\text{ : }x_{2},...,x_{n} < 0\}
\end{align*}
is a open convex cone, so it intersects any lattice. It follows that we find some $\gamma\in\Gamma$ such that $\Lambda(\gamma)\in U$ or $\chi_{j}(\gamma) < 0$ for $j = 2,...,n$. Since $\chi_{1}(\gamma)+...+\chi_{n}(\gamma) = 0$ we have $\chi_{1}(\gamma) > 0$ and $\chi_{j}(\gamma) < 0$ for $j > 1$. For this $\gamma$ we have $E_{\gamma}^{u} = E_{L}^{c}$ and the lemma follows.
\end{proof}
\begin{remark}
To prove the results of this section, it would suffice to be able to choose $Q > 8\varepsilon$, where $\varepsilon$ only depends on the $C^{1}-$distance between $f$ and $L$, in Equation \ref{Eq:CenterDomination} for some $\gamma\in\Gamma$ (see the proofs of Lemmas \ref{L:NonlinearCenterDomination} and \ref{L:FrankManningIsInjective}). However, for any $r < {\rm rank}(Z_{\rm Aut}(L))$ and $Q_{0}\in\mathbb{R}_{>0}$ one can construct a subgroup $\Gamma\leq Z_{\rm Aut}(L)$ with ${\rm rank}(\Gamma) = r$ and such that Equation \ref{Eq:CenterDomination} does not hold with $Q > Q_{0}$ for any ${\rm id}\neq\gamma\in\Gamma$, Lemma \ref{L:NecessityOfLargeAction}. So, if one uses the methods of this section, then either one must assume that $\Gamma$ has maximal rank or one must make an assumption on $\Gamma$ involving more properties than the rank of $\Gamma$.
\end{remark}
Combining Lemma \ref{L:AlgCenterDomination} with Lemma \ref{L:VolumeExpansionEst} and Theorem \ref{Thm:InjectivityOfHom}, we obtain the following lemma.
\begin{lemma}\label{L:NonlinearCenterDomination}
If $f\in{\rm Diff_{vol}^{\infty}}(\mathbb{T}^{d})$ is a $C^{1}-$small perturbation of $L$ satisfying ${\rm rank}(Z^{\infty}(f)) = {\rm rank}(Z_{\rm Aut}(L))$, then there is $g\in Z^{\infty}(f)$, homotopic to some hyperbolic $M$, and $\lambda > 1$ such that
\begin{align*}
\det(D_{x}g|_{E_{f}^{c}})\geq\lambda > 1
\end{align*}
for all $x\in\mathbb{T}^{d}$. Moreover, given $\delta > 0$ we can choose $f$ sufficiently $C^{1}-$close to $L$ such that $\lambda$ satisfies $\det(M|_{E_{L}^{c}})^{1-\delta}\leq\lambda$.
\end{lemma}
\begin{proof}
Let $Q > 0$ and $\eta > 0$. By Theorem \ref{Thm:InjectivityOfHom} and Lemma \ref{L:AlgCenterDomination} we find some $g\in Z^{\infty}(f)$ homotopic to $M$ where $\det(M|_{E_{L}^{c}}) > 1$ and
\begin{align*}
|\det(M|_{E_{L}^{c}})|\geq\left|\frac{\det(M|_{E_{M}^{u}})}{\det(M|_{E_{L}^{c}})}\right|^{Q}.
\end{align*}
By Lemma \ref{L:VolumeExpansionEst} this implies
\begin{align*}
|\det(D_{x}g^{n}|_{E_{f}^{c}})|\geq & c\cdot |\det(M^{n}|_{E_{L}^{c}})|^{1-\varepsilon}\cdot\left|\frac{\det(M^{n}|_{E_{M}^{u}})}{\det(M^{n}|_{E_{L}^{c}})}\right|^{-2\varepsilon}\geq \\ &
c\cdot |\det(M^{n}|_{E_{L}^{c}})|^{1-\varepsilon-\eta}\cdot\left|\frac{\det(M^{n}|_{E_{M}^{u}})}{\det(M^{n}|_{E_{L}^{c}})}\right|^{Q\eta}\cdot \\ &
\left|\frac{\det(M^{n}|_{E_{M}^{u}})}{\det(M^{n}|_{E_{L}^{c}})}\right|^{-2\varepsilon} = \\ &
|\det(M^{n}|_{E_{L}^{c}})|^{1-\varepsilon-\eta}\cdot\left[c\cdot\left|\frac{\det(M^{n}|_{E_{M}^{u}})}{\det(M^{n}|_{E_{L}^{c}})}\right|^{Q\eta - 2\varepsilon}\right].
\end{align*}
We fix $\delta > \varepsilon > 0$ and $\eta$ such that $1 - \varepsilon - \eta > 1 - \delta$ or $0 < \eta < \delta - \varepsilon$. We also want $Q\eta - 2\varepsilon > 0$ so we choose
\begin{align*}
Q > \frac{2\varepsilon}{\eta} > \frac{2\varepsilon}{\delta - \varepsilon}
\end{align*}
which can be done by Lemma \ref{L:AlgCenterDomination}. With this choice of $Q$ and $n$ large enough we obtain
\begin{align*}
c\cdot\left|\frac{\det(M^{n}|_{E_{M}^{u}})}{\det(M^{n}|_{E_{L}^{c}})}\right|^{Q\eta - 2\varepsilon}\geq1.
\end{align*}
The lemma follows since $\varepsilon$, and therefore also $\delta$, can be chosen as small as we want by letting $f$ be $C^{1}-$close to $L$.
\end{proof}
\begin{lemma}\label{L:CenterExpansion}
Let $f$ be as in Lemma \ref{L:NonlinearCenterDomination}, $g\in Z^{\infty}(f)$ satisfy $|\det(D_{x}g|_{E_{f}^{c}})|\geq\lambda > 1$ and $g$ fix all fixed points of $f$. Either $D_{0}g|_{E_{f}^{c}}$ has a pair of complex eigenvalues with modulus at least $\sqrt{\lambda} > 1$ or $D_{0}g|_{E_{f}^{c}(0)}$ is diagonalizable with one eigenvalue of multiplicity two and modulus at least $\sqrt{\lambda} > 1$.
\end{lemma}
\begin{proof}
Since $f(0) = 0$ we have $g(0) = 0$. The lemma follows by noting that $D_{0}g|_{E_{f}^{c}}:E_{f}^{c}(0)\to E_{f}^{c}(0)$ can not have a $1-$dimensional real eigenspace. Indeed, if $D_{0}g|_{E_{f}^{c}}$ has no $1-$dimensional real eigenspace then either $D_{0}g|_{E_{f}^{c}}$ has complex conjugate eigenvalues or $D_{0}g|_{E_{f}^{c}}$ has one eigenvalue with multiplicity $2$ and is diagonalizable. The size of the eigenvalues follows by $\det(D_{0}g|_{E_{f}^{c}})\geq\lambda > 1$, so the product of the eigenvalues must be at least $\lambda$ and since the eigenvalues have the same modulus we get the estimate of their size.

If $f$ is $C^{1}-$close to $L$ and fix $0$ then $E_{L}^{c}(0)$ is close to $E_{f}^{c}(0)$. It follows that $D_{0}f:E_{f}^{c}(0)\to E_{f}^{c}(0)$ is close to $L:E_{L}^{c}\to E_{L}^{c}$. In particular $D_{0}f|_{E_{f}^{c}}$ does not have real eigenvalues. Since the eigenspaces of $D_{0}g|_{E_{f}^{c}}$ are preserved by $D_{0}f|_{E_{f}^{c}}$ by commutativity it follows that $D_{0}g|_{E_{f}^{c}}$ can not have a $1-$dimensional eigenspace with multiplicity $1$ since this would imply that $D_{0}f|_{E_{f}^{c}}$ had a real eigenvalue.
\end{proof}
\begin{lemma}\label{L:ExistenceStationaryNormalForm}
Let $f$ and $g$ be as in Lemma \ref{L:CenterExpansion}. There is a diffeomorphism $\psi:{\rm Basin}(0)\to T_{0}W_{F}^{c}(0)$ where
\begin{align*}
{\rm Basin(0)} = \{x\in W_{F}^{c}(0)\text{ : }g^{-n}(x)\to 0\}
\end{align*}
such that
\begin{align*}
& \psi(g(x)) = D_{0}g|_{E_{f}^{c}}(\psi(x)), \\
& \psi(f(x)) = D_{0}f|_{E_{f}^{c}}(\psi(x))
\end{align*}
and $D_{0}\psi = \text{id}$.
\end{lemma}
\begin{proof}
Local existence of $\psi$ follows from \cite[Theorem 2]{StationaryNormalForm} for $g$. One then extends $\psi$ to $\text{Basin}(0)$ by the relation
\begin{align*}
\psi(x) = (D_{0}g|_{E_{f}^{c}})^{n}\psi(g^{-n}(x)).
\end{align*}
That $\psi$ also gives a normal form for $f$ follows since $\psi$ is unique. Indeed, if $\psi_{1},\psi_{2}:\text{Basin}(0)\to T_{0}W_{F}^{c}(0)$ are two diffeomorphisms, $\psi_{j}(0) = 0$, such that $D_{0}\psi_{1} = D_{0}\psi_{2}$ and $\psi_{j}(g(x)) = D_{0}g(\psi_{j}(x))$ then $\psi_{1} = \psi_{2}$. We denote $D_{0}g|_{E_{f}^{c}} = A$. By considering $g^{-1}$ instead of $g$ we may assume that $\norm{A} < 1$. We write $\psi(x) = \psi_{1}\circ\psi_{2}^{-1}(x)$ and have $\psi(Ax) = A\psi(x)$, $\psi(0) = 0$ and $D_{0}\psi = \text{id}$. It follows
\begin{align*}
\psi(x) = & A^{-n}\psi(A^{n}x) = A^{-n}\left(\psi(0) + A^{n}x + \mathcal{O}(\norm{A^{n}x}^{2})\right) = \\ &
x + A^{-n}\mathcal{O}(\norm{A^{n}x}^{2}).
\end{align*}
Since $A$ has two eigenvalues of the same modulus we have $\norm{A^{-n}}\leq c\norm{A}^{-n}$ for some constant $c$. We obtain an estimate
\begin{align*}
\psi(x) = x + \mathcal{O}(\norm{A}^{-n}\cdot\norm{A}^{2n}) = x + \mathcal{O}(\norm{A}^{n})
\end{align*}
and letting $n\to\infty$ we have $\psi(x) = x$. So $\psi_{1} = \psi_{2}$. Noting that $\psi'(x) = D_{0}f^{-1}\circ\psi\circ f$ also defines a normal form for $g$ with $D_{0}\psi' = \text{id}$ it follows by the uniqueness that $\psi' = \psi$ or $\psi(f(x)) = D_{0}f(\psi(x))$.
\end{proof}
Since $g$ is homotopic to some hyperbolic $M$ there is a Hölder semi-conjugacy $H:\mathbb{T}^{d}\to\mathbb{T}^{d}$ homotopic to identity such that $H(g(x)) = MH(x)$ \cite{FranksManningConjugacy} (or \cite[Theorem 2.6.1]{KatokIntroDynSyst}). It holds that $H:W_{f}^{c}(0)\to H(0) + E_{L}^{c}$, we assume for simplicity that $H$ is chosen such that $H(0) = 0$ which implies $H:W_{F}^{c}(0)\to E_{L}^{c}$. We also write $H:\mathbb{R}^{d}\to\mathbb{R}^{d}$ for the lift such that $H(0) = 0$.
\begin{lemma}\label{L:ExtendingFranksManningSemiConj}
The map $H:\mathbb{T}^{d}\to\mathbb{T}^{d}$ satisfy $H(fx) = LH(x)$, and the lifted map $H:\mathbb{R}^{d}\to\mathbb{R}^{d}$ satisfy $H(Fx) = LH(x)$.
\end{lemma}
\begin{proof}
The map $H$ is the unique map within its homotopy class that satisfies $H(0) = 0$ and $H(gx) = MH(x)$. Indeed, if $H_{1}$ and $H_{2}$ are two maps satisfying 
\begin{enumerate}[label = (\roman*)]
    \item $H_{j}(x) = x + \varphi_{j}(x)$, $\varphi_{j}:\mathbb{T}^{d}\to\mathbb{R}^{d}$, 
    \item $H_{j}(gx) = MH_{j}(x)$, 
    \item $H_{j}(0) = 0$,
\end{enumerate}
then $\varphi_{1}(x)$ and $\varphi_{2}(x)$ both solve a hyperbolically twisted cohomological equation $M\varphi_{j}(x) - \varphi_{j}(gx) = v(x) + m_{j}$ with $m_{j}\in\mathbb{Z}^{d}$ and $gx = Mx + v(x)$. Taking the difference $\psi(x) = \varphi_{2}(x) - \varphi_{1}(x)$ we see that $\psi$ solves
\begin{align}\label{Eq:HyperbolicTwist1}
M\psi(x) - \psi(gx) = m_{2} - m_{1}.
\end{align}
We note that $0 = H_{j}(0) = 0 + \varphi_{j}(0) = \varphi_{j}(0)$ on $\mathbb{T}^{d}$, so $\varphi_{j}(0)\in\mathbb{Z}^{d}$ for $j = 1,2$. It follows that $\psi(0)\in\mathbb{Z}^{d}$. Since Equation \ref{Eq:HyperbolicTwist1} has a unique solution $\psi$ which is constant, we conclude that $\psi(x)\in\mathbb{Z}^{d}$ for all $x$ so $\varphi_{2}(x) = \varphi_{1}(x) + m$ for some $m\in\mathbb{Z}^{d}$. This in turn implies that $H_{1}(x) = H_{2}(x)$ for all $x\in\mathbb{Z}^{d}$. The map $\Tilde{H}(x) = L^{-1}H(fx)$ satisfy $(i)$ since $f$ is homotopic to $L$ and it satisfies $(iii)$ since $f(0) = 0$. Moreover, using that $M$ commute with $L$ and $g$ commute with $f$ we obtain
\begin{align*}
\Tilde{H}(gx) = & L^{-1}H(fgx) = L^{-1}H(gfx) = L^{-1}MH(fx) = ML^{-1}H(fx) = \\ &
M\Tilde{H}(x)
\end{align*}
so $\Tilde{H}$ also satisfies $(ii)$. Since $\Tilde{H}$ satisfy $(i)$, $(ii)$ and $(iii)$ we conclude that $\Tilde{H}(x) = H(x)$ or $L^{-1}H(fx) = H(x)$, and the lemma follows for $H:\mathbb{T}^{d}\to\mathbb{T}^{d}$. On $\mathbb{R}^{d}$ the lemma follows for the case on $\mathbb{T}^{d}$ since the lift is normalized with $H(0) = 0$.
\end{proof}
In the remainder we will abuse notation and write $H$ for the map $H:W_{F}^{c}(0)\to E_{L}^{c}$.
\begin{lemma}\label{L:FrankManningIsInjective}
Let $f$ and $g$ be as in Lemma \ref{L:CenterExpansion}. The map $H:W_{F}^{c}(0)\to E_{L}^{c}$ is a homeomorphism when it is restricted to the set
\begin{align*}
{\rm Basin}(0) = \{x\in W_{F}^{c}(0)\text{ : }\lim_{n\to\infty}g^{-n}(x) = 0\}.
\end{align*}
\end{lemma}
\begin{proof}
Denote by $A := D_{0}G|_{E_{F}^{c}}$ and $B := D_{0}F|_{E_{F}^{c}}$. We begin by showing that $H^{-1}(0)$ has zero area. Since $H:\mathbb{T}^{d}\to\mathbb{T}^{d}$ is homotopic to identity, any lift of $H$ can be written as $x + h(x)$ for some $h:\mathbb{T}^{d}\to\mathbb{R}^{d}$ which factors through a torus. In particular, there is $C\geq0$ such that if $H(x) = 0$ then $\intd(x,0)\leq C$. Since $W_{F}^{c}(0)$ is uniformly close to $E_{L}^{c}$, any $x\in W_{F}^{c}(0)$ satisfying $H(x) = 0$ also satisfy $\intd_{c}(x,0)\leq C$ for some, possibly larger, constant $C$. Since $H(Gx) = MH(x)$ it follows that $H^{-1}(0)$ is $G-$invariant. The estimate $\det(D_{x}G|_{E_{f}^{c}})\geq\lambda > 1$ imply that for any measurable $V\subset W_{F}^{c}(0)$ we have $\text{Area}(G(V))\geq\lambda\text{Area}(V)$ so, in particular, any $G-$invariant set must have either $0$ area or $\infty$ area. Since $H^{-1}(0)\subset W_{F,C}^{c}(0)$ the set $H^{-1}(0)$ must have zero area.

Let $\psi:\text{Basin}(0)\to T_{0}W_{F}^{c}(0)$ be the normal form from Lemma \ref{L:ExistenceStationaryNormalForm}. That is, we have $\psi(G(x)) = A(\psi(x))$ and $\psi(F(x)) = B(\psi(x))$. For $n\in\mathbb{Z}$
\begin{align*}
H\circ\psi^{-1}(B^{n}(z)) = H\circ F^{n}\circ\psi^{-1}(z) = L^{n}H\circ\psi^{-1}(z)
\end{align*}
which implies that $|H\circ\psi^{-1}(B^{n}(z))| = |H(\psi^{-1}(z))|$ since the eigenvalues of $L$ along $E_{L}^{c}$ lie on the unit circle. If the eigenvalues of $B$ do not lie on the unit circle, then by letting $n$ be either positive or negative we have $\lim_{n\to\pm\infty}B^{n}(z) = 0$. Since $H(\psi^{-1}(0)) = H(0) = 0$ this would imply $|H\circ\psi^{-1}(B^{n}(z))|\to 0$. So, if $B$ do not have eigenvalues on the unit circle, then $H(\psi^{-1}(z)) = 0$ for all $z\in T_{0}W_{F}^{c}(0)$. Since $\psi^{-1}$ is a diffeomorphism, this implies that $H^{-1}(0)$ has a non-empty interior and in particular positive area, which is a contradiction. We conclude that $B$ has eigenvalues on the unit circle. 

Let $E\subset U$ be an ellipse about $0$ such that $B(E) = E$ and $H(\psi^{-1}(E))\neq\{0\}$ (such an ellipse exists since there are $B-$invariant ellipses foliating an open neighbourhood about $0$ in $U$, so if $H(\psi^{-1}(E)) = \{0\}$ for all such $E$ then $H^{-1}(0)$ would have non-empty interior and positive area). We denote by $E' = H(\psi^{-1}(E))$ the closed curve image of $E$. The map $B:E\to E$ is semiconjugated to $L:E'\to E'$ by $H\circ\psi^{-1}$. Since $L|_{E_{L}^{c}}$ is an irrational rotation, it follows that $B:E\to E$ has no finite orbit. So, $B:E\to E$ is an irrational rotation. It follows that $E'$ is a circle. Indeed, if we fix any $y\in E$ then $\{B^{n}y\}$ is dense in $E$, so $\{H(\psi^{-1}(B^{n}y))\} = \{L^{n}H(\psi^{-1}(y))\}$ is dense in $E'$. But $\{L^{n}H(\psi^{-1}(y))\}$ is also dense on the circle of radius $|H(\psi^{-1}(y))|$, so $E'$ is a circle of radius $|H(\psi^{-1}(y))|$. After a linear coordinate change of $T_{0}W_{F}^{c}(0)$ the map $H\circ\psi^{-1}$ takes the form $H(\psi^{-1}(re^{2\pi i\theta})) = \rho(r)e^{2\pi i\eta(r,\theta)}$, where we have identified $E_{L}^{c}$ and $T_{0}W_{F}^{c}(0)$ with $\mathbb{C}$. Let $\alpha,\omega\in\mathbb{R}$ be such that $Bz = ze^{2\pi i\alpha}$ and $Lz = ze^{2\pi i\omega}$. For each $\rho(r)\neq0$ the relation $H(\psi^{-1}(Bz) = Lz$ implies $\eta(r,\theta + \alpha) = \eta(r,\theta) + \omega$. It follows that $\eta(r,\theta + \alpha) - \eta(r,\theta) = \omega$. Writing $\eta(r,\theta) = m_{r}\theta + u_{r}(\theta)$ with $m_{r}\in\mathbb{Z}$ and $u_{r}:\mathbb{R}/\mathbb{Z}\to\mathbb{R}$ we see that
\begin{align*}
\eta(r,\theta + \alpha) - \eta(r,\theta) = & m_{r}(\theta + \alpha) + u_{r}(\theta + \alpha) - m_{r}\theta - u_{r}(\theta) = \\ &
m_{r}\alpha + (u_{r}(\theta + \alpha) - u_{r}(\theta)) = \omega
\end{align*}
which implies that $m_{r}\alpha = \omega$ for some $m_{r}\in\mathbb{Z}$ and $u_{r}(\theta) = \theta_{0}(r)$ is constant since $\alpha$ is irrational. Since $e^{2\pi i\alpha}$ is an eigenvalue of $B = D_{0}f|_{E_{f}^{c}}$, $e^{2\pi i\omega}$ is an eigenvalue of $L$ and $B$ is close to $L|_{E_{L}^{c}}$ we have $m_{r} = 1$. So, $\eta(r,\theta) = \theta + \theta_{0}(r)$ and $\alpha = \omega$. We claim that $\rho(r)\neq0$ for all $r\neq0$. Indeed, for $r > 0$ we write $\gamma(\theta) = \psi^{-1}(re^{2\pi i\theta})$. Let $O$ be the (non-empty) interior of $\gamma$. Since $\text{Area}(G^{n}(O))\geq\lambda^{n}\text{Area}(O)\to\infty$ there is some sequence $x_{n}\in\partial G^{n}(O)$ such that $x_{n}\to\infty$. But $\partial G^{n}(O) = \text{Im}(G^{n}\circ\gamma)$, so we find $\theta_{n}\in\mathbb{R}/\mathbb{Z}$ with $x_{n} = G^{n}(\psi^{-1}(re^{2\pi i\theta_{n}}))$. If $H(\gamma(\theta)) = 0$ then
\begin{align*}
0 = M^{n}(H(\psi^{-1}(re^{2\pi i\theta_{n}})) = H(G^{n}\psi^{-1}(re^{2\pi i\theta_{n}}) = H(x_{n})
\end{align*}
which is a contradiction since $\intd_{c}(H(x_{n}),0)\geq\intd_{c}(x_{n},0) - C$. It follows that $\rho(r)\neq0$ for $r\neq0$, so the formula $H(\psi^{-1}(re^{2\pi i\theta})) = \rho(r)e^{2\pi i(\theta + \theta_{0}(r))}$ holds with both $\rho$ and $\theta_{0}$ continuous. Finally, we claim that $\rho(r)$ is increasing. It suffices to show that $\rho$ is injective since $\rho(0) = 0$. Let $0 < r_{1} < r_{2}$ be such that $\rho(r_{1}) = \rho(r_{2}) = \rho_{0}\neq0$. Consider the annulus $S$ in $W_{F}^{c}(0)$ bounded by the curves $\gamma_{1}(\theta) = \psi^{-1}(r_{1}e^{2\pi i\theta})$ and $\gamma_{2}(\theta) = \psi^{-1}(r_{2}e^{2\pi i\theta})$. Since $H(\text{Im}(\gamma_{1})) = H(\text{Im}(\gamma_{2})) = \rho_{0}\cdot S^{1}$ and $\norm{H(x) - x}\leq\norm{h}_{C^{0}} = C < \infty$ we have
\begin{align*}
|G^{n}\circ\gamma_{j}(\theta)|\leq & |H(G^{n}\circ\gamma_{j}(\theta))| + C = |M^{n}(\rho_{0}e^{2\pi i(\theta + \theta_{0}(r_{j}))})| + C = \\ &
\eta^{n}\rho_{0} + C,
\end{align*}
where $\eta = \sqrt{|\det(M|_{E_{L}^{c}})|} > 1$ is the modulus of the eigenvalue of $M$ along $E_{L}^{c}$. Similarly
\begin{align*}
|G^{n}\circ\gamma_{j}(\theta)|\geq & |H(G^{n}\circ\gamma_{j}(\theta))| - C = |M^{n}(\rho_{0}e^{2\pi i(\theta + \theta_{0}(r_{j}))})| - C = \\ &
\eta^{n}\rho_{0} - C.
\end{align*}
Since the boundary of $G^{n}S$ is given by the curves $G^{n}\circ\gamma_{j}$, $j = 1,2$, it follows that $G^{n}S$ lie within $W_{F}^{c}(0)\cap B_{2C}(\rho_{0}\eta^{n}\cdot S^{1})$. Noting that $\text{Area}(W_{F}^{c}(0)\cap B_{2C}(x))$ is uniformly bounded in $x$, we have for some constant $K$
\begin{align*}
\lambda^{n}\text{Area}(S)\leq\text{Area}(G^{n}S)\leq K\cdot\eta^{n}.
\end{align*}
By Lemma \ref{L:NonlinearCenterDomination}, we can, given $\delta > 0$, choose $f$ sufficiently $C^{1}-$close to $L$ such that $\lambda\geq|\det(M|_{E_{L}^{c}})|^{1-\delta} = \eta^{2(1-\delta)}$. It follows that we have
\begin{align*}
\eta^{2(1-\delta)n}\text{Area}(S)\leq K\eta^{n}.
\end{align*}
With $\delta < 1/2$ and letting $n\to\infty$ we obtain $\text{Area}(S) = 0$. It follows that $\text{Im}(\gamma_{1}) = \text{Im}(\gamma_{2})$ so $\rho(r)$ is indeed injective. We write $H(\psi^{-1}(re^{2\pi i\theta})) = \rho(r)e^{2\pi i(\theta + \theta_{0}(r))}$ where $\rho$ is strictly increasing. Since $\psi^{-1}$ is invertible this implies that $H$ is invertible.
\end{proof}
Finally we show that $H:W_{F}^{c}(0)\to E_{L}^{c}$ is invertible by showing that $\text{Basin}(0) = W_{F}^{c}(0)$.
\begin{lemma}
We have
\begin{align*}
{\rm Basin}(0) = W_{F}^{c}(0)
\end{align*}
so $H:W_{F}^{c}(0)\to E_{L}^{c}$ is invertible.
\end{lemma}
\begin{proof}
Let $\gamma:\mathbb{T}\to\text{Basin}(0)$ be a smooth, simple and closed $F-$invariant curve traversing $0$ (we can take $\gamma(\theta) = \psi^{-1}(re^{2\pi i\theta})$ for any $r > 0$ where we have chosen coordinates on $T_{0}W_{F}^{c}(0)$ as in the proof of Lemma \ref{L:FrankManningIsInjective}). As in the proof of Lemma \ref{L:FrankManningIsInjective} we see that $H(\text{Im}(\gamma))$ is a circle of some radius $\rho > 0$ in $E_{L}^{c}$. Consider the sequence of $F-$invariant curves $\gamma_{n}(\theta) = G^{n}\circ\gamma(\theta)$. One sees that $H(\text{Im}(\gamma_{n}))$ is a circle of radius $\eta^{n}\rho$ where $\eta$ is the modulus of the eigenvalue of $M$ along $E_{L}^{c}$. Since $\norm{x} - C\leq\norm{H(x)}\leq\norm{x} + C$ it follows that $\gamma_{n}$ is a smooth, simple and closed curve traversing $0$ that lies within $B_{C}^{c}(\eta^{n}\rho\cdot S^{1})$. If we denote by $U_{n}$ the interior of $\gamma_{n}$ in $W_{F}^{c}(0)$ it follows that there is some constant $K$ such that $W_{F,\eta^{n}\rho - K}^{c}(0)\subset U_{n}$. On the other hand, $U_{n}\subset\text{Basin}(0)$ since $\gamma$ lie in $\text{Basin}(0)$ and $\text{Basin}(0)$ is $G-$invariant. Letting $n\to\infty$ the lemma follows.
\end{proof}
Recall that $T_{n}:W_{F}^{c}(0)\to W_{F}^{c}(0)$ are defined in Definition \ref{Def:TnMapsFromFederico}. Since $f$ is semiconjugate to $L$ the diffeomorphism $f$ is not accessible, Claim \ref{Claim:AccessibilityFactor}, so $T_{n}$ forms a $\mathbb{Z}^{d}-$action on $W_{F}^{c}(0)$. Since $H$ is invertible, the action of $T_{n}$ is minimal.
\begin{lemma}
The map $H:W_{F}^{c}(0)\to E_{L}^{c}$ is a $C^{1}-$diffeomorphism. 
\end{lemma}
\begin{proof}
Let $\mathcal{F}$ be the foliation on $W_{F}^{c}(0)\setminus\{0\}$ with leaves $\psi^{-1}(r\cdot S^{1})$. We denote by $X_{\theta}$ the unit vector field along $\mathcal{F}$. Given $n\in\mathbb{Z}^{d}\setminus\{0\}$ we denote $Y_{\theta}(T_{n}(x)) = D_{x}T_{n}(X_{\theta}(x))$. Since $n\neq0$ we have $T_{n}(0)\neq0$ and we find a leaf $\mathcal{F}(x)$ with $x$ sufficiently close to $0$ such that $T_{n}(\mathcal{F}(x))$ do not wind around $0$. We claim that there is some $y\in\mathcal{F}(x)$ such that $D_{y}T_{n}(X_{\theta}(y))$ is not proportional to $X_{\theta}(T_{n}(y))$. Indeed, if $D_{y}T_{n}(X_{\theta}(y))$ was proportional to $X_{\theta}(T_{n}(y))$ for all $y\in\mathcal{F}(x)$ then $T_{n}(\mathcal{F}(x))$ would be a leaf of $\mathcal{F}$, but this can not be since $T_{n}(\mathcal{F}(x))$ do not wind around $0$. Let $z_{0} = T_{n}(y_{0})$, $y_{0}\in\mathcal{F}(x)$, be chosen such that $X_{\theta}(z_{0})$ and $Y_{\theta}(z_{0}) = D_{y_{0}}T_{n}(X_{\theta}(y_{0}))$ are linearly independent. We find some open set $z_{0}\in U$ such that $U$ has transverse $C^{1}-$smooth foliations $\mathcal{F} = \mathcal{F}|_{U}$ and $\mathcal{G}$ where $\mathcal{G}(T_{n}(y))$ is the connected component of $U\cap T_{n}(\mathcal{F}(y))$ with $y\in T_{n}^{-1}(U)$. We claim that $H$ is smooth along $\mathcal{F}$. Indeed, from the proof of Lemma \ref{L:FrankManningIsInjective}, $H(\psi^{-1}(re^{2\pi i\theta})) = \rho(r)e^{2\pi i(\theta + \theta_{0}(r))}$ so by the definition of $\mathcal{F}$ it follows that $H$ is differentiable along $\mathcal{F}$. It also holds that $H$ is differentiable along $\mathcal{G}$ in $U$. Indeed, $H(T_{n}(x)) = H(x) + n^{c}$ so since $\mathcal{G}$ is the image of $\mathcal{F}$ under $T_{n}$ and $H$ is differentiable along $\mathcal{F}$ it follows that $H$ is differentiable along $\mathcal{G}$. Since $\mathcal{F}$ and $\mathcal{G}$ are transverse it follows that $H$ is $C^{1}$ in $U$. Moreover, since $H|_{U}$ is a homeomorphism onto its image, it follows by Sard's theorem that we may assume that $\det(D_{x}H)\neq0$ on $U$.

Let $z\in W_{F}^{c}(0)$ be arbitrary. The action $T_{n}$ of $\mathbb{Z}^{d}$ on $W_{F}^{c}(0)$ is minimal, so we find some $n_{0}$ such that $T_{n_{0}}(z)\in U$. Since $H(z) = H(T_{n_{0}}(z)) - n_{0}^{c}$ and since $T_{n_{0}}$ is a $C^{1}-$diffeomorphism it follows that $H$ is differentiable at $z$ with continuous derivative and $\det(D_{z}H)\neq0$. The lemma follows by the inverse function theorem.
\end{proof}
We can now prove Theorem \ref{Thm:FullCentralizer}.
\begin{proof}[Proof of Theorem \ref{Thm:FullCentralizer}]
We begin by showing that $H:\mathbb{T}^{d}\to\mathbb{T}^{d}$ is a homeomorphism. It suffices to show that $H:\mathbb{R}^{d}\to\mathbb{R}^{d}$ is a homeomorphism. Since $H$ is surjective, it suffices to show that $H$ is injective. Let $\pi^{su}:\mathbb{R}^{d}\to W_{F}^{c}(0)$ be the holonomy along $W_{F}^{su}$ (note that $E_{f}^{s}\oplus E_{f}^{u}$ is integrable since $f$ is not accessible \cite[Theorem 5.1]{StabelErgodicity}) and let $\pi_{0}^{su}:\mathbb{R}^{d}\to E_{L}^{c}$ be the holonomy along $W_{L}^{su}(x) = x + E_{L}^{su}$. If $H(p) = H(q)$ then $H(\pi^{su}(p)) = \pi_{0}^{su}(H(p)) = \pi_{0}^{su}(H(q)) = H(\pi^{su}(q))$. Since $H:W_{F}^{c}(0)\to E_{L}^{c}$ is a homeomorphism it follows that $\pi^{su}(p) = \pi^{su}(q)$. That $p = q$ now follows from Claim \ref{Claim:ExistenceFMcoordinates} since $H|_{W_{F}^{\chi}(x)} = \Phi_{\chi,x} + \Phi_{\chi}(x)$ by uniqueness of $\Phi_{\chi}$.

Since $H:W_{F}^{c}(0)\to E_{L}^{c}$ is a $C^{1}-$diffeomorphism and $\pi^{su}$ are diffeomorphisms along $W_{F}^{c}$ it follows that $H$ is smooth along $W_{F}^{c}$ with a derivative $D_{x}H|_{E_{f}^{c}}$ that is invertible. Since $H(g(x)) = MH(x)$ and $M$ is expanding along $E_{L}^{c}$ it follows that $g$ is expanding along $E_{f}^{c}$. Since $g$ has the same hyperbolic behaviour along $E_{f}^{\chi}$ as $M$ have along $E_{L}^{\chi}$ for all $\chi(L)\neq 0$ by Lemma \ref{L:PropLyapFunc}, it follows that $g$ is Anosov. That $H$ is a $C^{\infty}-$diffeomorphism now follows from Theorems \ref{Thm:RigidityOfHigherRankAnosov}, \ref{Thm:InjectivityOfHom} and \ref{Thm:AllSubgroupsHigerRank}.
\end{proof}

\subsection{Centralizer contains element homotopic to hyperbolic}

In this section, we want to show that, in some cases, accessibility is preserved for topological factors. More precisely, we show the following.
\begin{claim}\label{Claim:AccessibilityFactor}
Let $f\in{\rm Diff}^{\infty}(M)$ and $g\in{\rm Diff}^{\infty}(N)$ be partially hyperbolic diffeomorphisms on the manifold $M$ and $N$ respectively. Moreover, assume that $g$ is dynamically coherent with isometric center. If $g$ is a topological factor of $f$ and $f$ is accessible, then $g$ is also accessible.
\end{claim}
\begin{lemma}\label{L:IsoCenterEstimateAlongCenter}
Let $f:M\to M$ be dynamically coherent with isometric center. For $y\in W^{c}_{\delta}(x)$, we have
\begin{align*}
\frac{1}{C}\intd(x,y)\leq\intd(f^{n}x,f^{n}y)\leq C\intd(x,y)
\end{align*}
for $n\in\mathbb{Z}$.
\end{lemma}
\begin{proof}
The lemma is immediate in the $\intd_{c}$ metric along the leaves of $W^{c}$. Since the $\intd_{c}$ is comparable to $\intd$ for $x,y\in M$ close, the lemma follows.
\end{proof}
\begin{lemma}\label{L:CharacterizationStableUnstable}
Let $f:M\to M$ be partially hyperbolic, dynamically coherent and isometric along the center with each $W^{cs},W^{cu}$ uniquely integrable. The foliations $W^{s}$ and $W^{u}$ have a dynamical and topological characterisation
\begin{align*}
& W^{s}(x) = \{y\in M\text{ : }\intd(f^{n}x,f^{n}y)\to 0,\text{ }n\to\infty\}, \\
& W^{u}(x) = \{y\in M\text{ : }\intd(f^{-n}x,f^{-n}y)\to 0,\text{ }n\to\infty\}.
\end{align*}
\end{lemma}
\begin{proof}
We prove the lemma for $W^{s}(x)$ and the proof for $W^{u}(x)$ follows by reversing time. It is immediate that
\begin{align*}
W^{s}(x)\subset\{y\in M\text{ : }\lim_{n\to\infty}\intd(f^{n}x,f^{n}y) = 0\} =: \Tilde{W}_{x}
\end{align*}
so it suffices to show $\Tilde{W}_{x}\subset W^{s}(x)$. Let $y\in \Tilde{W}_{x}$. If $f^{n}y\in W^{s}(f^{n}x)$ for some $n$ then $y\in W^{s}(x)$ since $W^{s}(x)$ is $f-$invariant. So we may assume without loss of generality that $\intd(f^{n}x,f^{n}y) < \delta$ for all $n\geq0$. That is, we assume that $y$ is such that
\begin{align*}
& \lim_{n\to\infty}\intd(f^{n}x,f^{n}y) = 0, \\
& \intd(f^{n}x,f^{n}y) < \delta,\quad n\geq0.
\end{align*}
For any $\varepsilon > 0$ there is $\varepsilon > \delta > 0$ such that $f^{n}(W^{cs}_{\delta}(x))\subset W^{cs}_{\varepsilon}(f^{n}x)$ for $n\geq0$. Let $x'\in W^{u}(x)$, $x\neq x'$, be close to $x$ such that $y\in W^{cs}(x')$ is at most a distance $\delta$ from $x'$. For $\varepsilon>0$ sufficiently small we find $n > 0$ such that $\intd(f^{n}x',f^{n}x) > 2\varepsilon$, but then
\begin{align*}
\intd(f^{n}x,f^{n}y)\geq\intd(f^{n}x,f^{n}x') - \intd(f^{n}x',f^{n}y) > 2\varepsilon - \varepsilon = \varepsilon
\end{align*}
which contradicts $\intd(f^{n}x,f^{n}y) < \delta$. So $x = x'$ or $y\in W^{cs}(x)$. Now, let $x''\in W^{c}(x)$, $x\neq x''$, be close to $x$ such that $y\in W^{s}(x'')$. We have
\begin{align*}
\intd(f^{n}x,f^{n}y)\geq\intd(f^{n}x,f^{n}x'') - \intd(f^{n}x'',f^{n}y).
\end{align*}
By Lemma \ref{L:IsoCenterEstimateAlongCenter} $\intd(f^{n}x,f^{n}x'')$ is bounded away from $0$ and 
\begin{align*}
\intd(f^{n}x'',f^{n}y)\to 0
\end{align*}
which implies that $\intd(f^{n}x,f^{n}y)$ is bounded away from zero. This implies that $\intd(f^{n}x,f^{n}y)\not\to0$ which is a contradiction, so $x = x''$. It follows that $y\in W^{s}(x)$.
\end{proof}
We can now prove Claim \ref{Claim:AccessibilityFactor}.
\begin{proof}[Proof of Claim \ref{Claim:AccessibilityFactor}]
Let $h:M\to N$ be a surjective continuous map such that $hf = gh$. Since $g$ is dynamically coherent with isometric center, it follows from Lemma \ref{L:CharacterizationStableUnstable} that $h$ maps the stable foliation of $f$ to the stable foliation of $g$ and unstable foliation of $f$ to the unstable foliation of $g$. In particular, $h$ maps $su-$paths of $f$ to $su-$paths of $g$. Let $p,q\in N$ and $x\in h^{-1}p$ and $y\in h^{-1}q$. Since $f$ is accessible we find an $su-$path for $f$ from $x$ to $y$, composing this $su-$path with $h$ gives $su-$path for $g$ from $p$ to $q$.
\end{proof}

\subsubsection{Proof of Theorem \ref{Thm:HomotopicToHyperbolic}}

We can now prove Theorem \ref{Thm:HomotopicToHyperbolic}.
\begin{lemma}\label{L:ConjugacyWhenHaveHyperbolic}
Let $f$ be a $C^{5}-$small perturbation of $L$ and $g\in Z^{\infty}(f)$. If $g$ is homotopic to a hyperbolic automorphism, then $Z^{\infty}(f)$ contains an Anosov diffeomorphism.
\end{lemma}
\begin{proof}
Since $g$ is homotopic to a hyperbolic $M$ we have a Franks-Manning semiconjugacy \cite{FranksManningConjugacy} from $g$ to $M$. We denote this conjugacy by $H:\mathbb{T}^{d}\to\mathbb{T}^{d}$. From Lemma \ref{L:ExtendingFranksManningSemiConj} we also have $Hf = LH$. By Claim \ref{Claim:AccessibilityFactor} and the fact that $L$ is not accessible, it follows that $f$ is not accessible. From Theorem \ref{Thm:RHaccessibilityDichom} it follows that $f$ is conjugated to $L$ by a bi-Hölder conjugacy that is a $C^{1}-$diffeomorphim when restricted to $W_{f}^{c}$. We denote this conjugacy by $\Tilde{H}$. We claim that $\Tilde{H}(x) = H(x) + p_{0}$ with $p_{0}\in\mathbb{R}^{d}$. To see this, note that $LH\Tilde{H}^{-1} = Hf\Tilde{H}^{-1} = H\Tilde{H}^{-1}L$. The claim now follows from \cite{ConjRigidity} since $H$ and $\Tilde{H}$ are both homotopic to identity (in \cite{ConjRigidity} the authors consider homeomorphisms commuting with $L$, but the proof also works in the case of non-invertible maps).

The homeomorphism $H$ maps $W_{f}^{\sigma}$ to $W_{L}^{\sigma}$. Since $M$ is hyperbolic and commutes with $L$, it follows that $M$ preserves $W_{L}^{c}$. Since $W_{L}^{c}$ correspond to a complex eigenvalue, it follows that $M$ either contracts $W_{L}^{c}$ or expands $W_{L}^{c}$. Since $H$ is a $C^{1}-$diffeomorphism along $W_{f}^{c}$ we can differentiate $H$ along $W_{f}^{c}$ and get $D_{x}g|_{E_{f}^{c}} = D_{Mh(x)}h^{-1}\circ M\circ D_{x}h$. It follows that $D_{x}g$ either expands or contracts $E_{f}^{c}$ depending on whether $M$ expands or contracts $E_{L}^{c}$. So, $g$ is Anosov by Lemma \ref{L:PropLyapFunc}.
\end{proof}
\begin{proof}[Proof of Theorem \ref{Thm:HomotopicToHyperbolic}]
By Theorems \ref{Thm:AllSubgroupsHigerRank} and \ref{Thm:RigidityOfHigherRankAnosov}, it suffices to show that $Z^{\infty}(f)$ contains an Anosov element. By Lemma \ref{L:ConjugacyWhenHaveHyperbolic}, it suffices to show that $Z^{\infty}(f)$ contains a diffeomorphism homotopic to a hyperbolic automorphism. The theorem now follows from Theorem \ref{Thm:InjectivityOfHom} and Lemma \ref{L:LargestSubgroupWithoutHyperbolic}.
\end{proof}

%% file: SymplecticPerturbations.tex
\section{Symplectic perturbations}
\label{Sec:SymplecticPerturbations}

In this section we prove Theorems \ref{Thm:LeafConjugacy} and \ref{Thm:SymplecticRigidity}. Fix $L\in\text{GL}(d,\mathbb{Z})$, $d\geq6$, with property $(P)$ and a $C^{1}-$small perturbation $f$ of $L$. We begin by showing that if the perturbation $f$ is symplectic, then $f$ is $C^{1+\alpha}-$leaf conjugated to $L$.

\subsection{Regularity of coordinates along invariant foliations}

We denote the non-zero Lyapunov exponents of $L$ by $\chi_{1},...,\chi_{N},-\chi_{1},...,-\chi_{N}$, $N = (d-2)/2$. We also recall that if $\Gamma\leq Z_{\text{Aut}}(L)$ is an abelian subgroup we denote by $\chi_{j}:\Gamma\to\mathbb{R}$ the Lyapunov functional of $\Gamma$ with $\chi_{j}(L) = \chi_{j}$. We claim that the maps $\Phi_{\chi}:\mathbb{R}^{d}\to E_{L}^{\chi}$ from Claim \ref{Claim:ExistenceFMcoordinates} are regular if the centralizer of $f$ is large.
\begin{claim}\label{Claim:RegularityCoordinates}
Let $f\in{\rm Diff}_{\omega}^{\infty}(\mathbb{T}^{d})$ be a $C^{5}-$small symplectic perturbation of $L$. If ${\rm rank}(Z^{\infty}(f)) > 1$ then each $\Phi_{\chi}:\mathbb{R}^{d}\to E_{L}^{\chi}$ is $C^{1+\alpha}$. Moreover, if $L$ has $r-$spread spectrum, $r\geq 2$, then $\Phi_{\chi}:\mathbb{R}^{d}\to E_{L}^{\chi}$ is $C^{r-\varepsilon}$ for every $\varepsilon > 0$.
\end{claim}
We will use the following lemma, which is an immediate consequence of \cite{KatokPeriodicPoints}.
\begin{lemma}\label{L:DenseSetOfHyperbolic}
Let $X$ be a closed, smooth manifold, $f\in{\rm Diff}^{1+\alpha}(X)$ and $\mu$ a $f-$invariant ergodic volume. If $\mu$ is a hyperbolic measure for $f$ then $f$ has a dense set of hyperbolic periodic points.
\end{lemma}
\begin{proof}
Let $\Lambda_{\chi,\ell}^{k}\subset X$ be Pesin blocks, see \cite[Section 2]{KatokPeriodicPoints}. By \cite[Main Lemma Section 3]{KatokPeriodicPoints} there is $\psi(k,\chi,\ell,\delta)$ such that if $x\in\Lambda_{\chi,\ell}^{k}$, $f^{n}x\in\Lambda_{\chi,\ell}^{k}$ and $\intd(x,f^{n}x) < \psi$ then there is a $f-$hyperbolic periodic point $p\in X$ (of period $n$) satisfying $\intd(x,p) < \delta$. By choosing $k$ as the dimension of the stable direction for $f$ with respect to $\mu$ and $\ell$ sufficiently large we have $\mu(\Lambda_{\chi,\ell}^{k}) > 0$. After dropping to the set of Lebesgue density points of $\Lambda_{\chi,\ell}^{k}$, we may assume that $\mu(\Lambda_{\chi,\ell}^{k}\cap B_{r}(x)) > 0$ for all $x\in\Lambda_{\chi,\ell}^{k}$ and $r > 0$. For $x\in\Lambda_{\chi,\ell}^{k}$, let $\varepsilon = \min(\delta,\psi)$ and $D_{x} = \Lambda_{\chi,\ell}^{k}\cap B_{\varepsilon}(x)$. Since $\mu(D_{x}) > 0$, the Poincaré recurrence theorem implies that there is some $y\in D_{x}$ and $n > 0$ such that $f^{n}y\in D_{x}$. By \cite[Main Lemma Section 3]{KatokPeriodicPoints} there is some $f-$hyperbolic periodic point $p\in X$ such that $\intd(p,y) < \delta$. It follows that $\intd(x,p)\leq\intd(x,y)+\intd(y,p) < 2\delta$. If $\delta = 1/n$ and $p_{n}$ is the hyperbolic periodic point constructed above, then $\intd(x,p_{n})\to0$. That is, $x$ lie in the closure of the $f-$hyperbolic periodic points. So $\Lambda_{\chi,\ell}^{k}$ lie in the closure of the $f-$hyperbolic periodic points. Since the union $\cup_{\ell > 0}\Lambda_{\chi,\ell}^{k}$ has full volume, it follows that the closure of the hyperbolic periodic points is dense and the lemma follows.
\end{proof}
\begin{lemma}\label{L:DenseHigherRankPeriodicPoints}
With $f$ as in Claim \ref{Claim:RegularityCoordinates}, the action of $Z^{\infty}(f)$ on $\mathbb{T}^{d}$ have a dense set of periodic points.
\end{lemma}
\begin{proof}
By \cite{AvilaViana} $f$ is either non-uniformly hyperbolic with respect to volume or $f$ is topologically conjugated to $L$. In the second case, the lemma is immediate. In the first case, we define
\begin{align*}
\Lambda_{h,\delta}^{n} = \{p\in\mathbb{T}^{d}\text{ : }f^{n}p = p,\text{ }\intd(\sigma(D_{p}f^{n}|_{E_{f}^{c}}),S^{1})\geq\delta\}.
\end{align*}
That is, $\Lambda_{h,\delta}^{n}$ consists of periodic points of period $n$ where the hyperbolicity has a uniform lower bound. It is clear that if $g\in Z^{\infty}(f)$ then $g(\Lambda_{h,\delta}^{n}) = \Lambda_{h,\delta}^{n}$. Moreover, the set $\Lambda_{h,\delta}^{n}$ is discrete and closed. It is closed since after identifying $E_{f}^{c}$ with a trivial vector bundle the spectrum of $D_{x}f^{n}|_{E_{f}^{c}}$ vary continuously in $x\in\mathbb{T}^{d}$. The fact that $\Lambda_{h,\delta}^{n}$ is discrete follows from the Hartman-Grobman theorem \cite{KatokIntroDynSyst}. Since $\Lambda_{h,\delta}^{n}$ is closed and discrete it follows that $\Lambda_{h,\delta}^{n}$ is finite, it follows that every $p\in\Lambda_{h,\delta}^{n}$ is $g-$periodic. So, $\Lambda_{h,\delta}^{n}$ consists of periodic points for the action of $Z^{\infty}(f)$. By Lemma \ref{L:DenseSetOfHyperbolic} $f$ has a dense set of hyperbolic periodic points, so
\begin{align*}
\Lambda = \bigcup_{n\geq 1,\delta > 0}\Lambda_{h,\delta}^{n}
\end{align*}
is dense.
\end{proof}
Let $\chi,\lambda\in{\rm Lyap}(L)$ and $W_{f}^{\chi}$, $W_{f}^{\lambda}$ be the associated Lyapunov foliations. For the proof of the next Lemma, it will be convenient define a map $\Phi_{\chi,p}^{\lambda}$ as
\begin{align}\label{Eq:OfFoliationFMCoordinates}
\Phi_{\chi,p}^{\lambda}:W_{f}^{\lambda}(p)\to E_{L}^{\chi},\quad\Phi_{\chi,p}^{\lambda}(q) = \Phi_{\chi}(q) - \Phi_{\chi}(p).
\end{align}
That is, $\Phi_{\chi,p}^{\lambda}$ is defined precisely as $\Phi_{\chi,p}$ from Claim \ref{Claim:ExistenceFMcoordinates}, but instead of restricting to $W_{f}^{\chi}(p)$ we restrict to $W_{f}^{\lambda}(p)$ (when $\lambda = \chi$ then clearly we have $\Phi_{\chi,p}^{\chi} = \Phi_{\chi,p}$). Note that, as in the proof of Claim \ref{Claim:ExistenceFMcoordinates}, $\Phi_{\chi,p}^{\lambda}$ is well-defined for $p\in\mathbb{T}^{d}$ and is continuous (in the Hölder topology) in $p$ for some Hölder exponent $\theta$. Moreover, if $g\in Z^{\infty}(f)$ is homotopic to $A_{g}\in{\rm GL}(d,\mathbb{Z})$ then the functional equation
\begin{align}
\Phi_{\chi,gp}^{\lambda}(g(q)) = e^{\chi(A_{g})}\Phi_{\chi,p}^{\lambda}(q),\quad q\in W_{f}^{\lambda}(p)
\end{align}
still holds as for $\Phi_{\chi,p}$. However, as opposed to $\Phi_{\chi,p}$, the map $\Phi_{\chi,p}^{\lambda}$ may not be a homeomorphism (in fact, we will show that if $\lambda\neq\chi$ then $\Phi_{\chi,p}^{\lambda} = 0$) and the Hölder exponent $\theta$ of $\Phi_{\chi,p}^{\lambda}$ may not be close to $1$. 

We will also use non-stationary linearizations of $f$ (and $Z^{\infty}(f)$) along $W_{f}^{\lambda}$ for each Lyapunov functional $\lambda$. We find a continuous family of $C^{r}$ diffeomorphisms 
\begin{align}\label{Eq:NonStationaryNormalForm}
\Psi_{\lambda,p}:T_{p}W_{f}^{\lambda}\to W_{f}^{\lambda}(p),\quad \Psi_{\lambda,p}(0) = p\text{ and }D_{0}\Psi_{\lambda,p} = {\rm id}_{T_{p}W_{f}^{\lambda}},
\end{align}
with a functional equation
\begin{align}
\Psi_{\lambda,fp}(D_{p}f(v)) = f\Psi_{\lambda,p}(v),
\end{align}
see \cite[Corollary 4.8]{RealNormalFormsExist} (to see a construction of $\Psi_{\lambda,p}$, in the non-uniformly contracting or expanding setting, along $W_{f}^{\lambda}$, see \cite[Lemma 3.2, Lemma 3.4]{NormalFormsExist}). That is, $\Psi_{\lambda,p}$ (or rather $\Psi_{\lambda,p}^{-1}$) linearize $f$ along $W_{f}^{\lambda}$. The family of maps $\Psi_{\lambda,p}$ also linearize every $g\in Z^{\infty}(f)$, see \cite[Theorem 4.6, (3)]{RealNormalFormsExist}. We will use that the family $\Psi_{\lambda,p}$ define an affine structure on the leaves of $W_{f}^{\lambda}$. That is, given $q\in W_{f}^{\lambda}(p)$ the map
\begin{align}
\Psi_{q,\lambda}^{-1}\circ\Psi_{p,\lambda}:E_{f}^{\lambda}(p)\to E_{f}^{\lambda}(q)
\end{align}
is an affine map, see \cite[Corollary 4.8]{RealNormalFormsExist}. We fix, once and for all, a nowhere vanishing vector field $v:\mathbb{T}^{d}\to E_{f}^{\lambda}$ and define for $q\in W_{f}^{\lambda}(p)$
\begin{align}
\Psi_{q,\lambda}^{-1}\circ\Psi_{p,\lambda}(tv(p)) = \left[\sigma_{p,q}t  + s_{p,q}\right]v(q)
\end{align}
where $\sigma_{p,q},s_{p,q}\in\mathbb{R}$.
\begin{lemma}\label{L:RegularityOfFMalongLeaves}
Let $f$ be as in Claim \ref{Claim:RegularityCoordinates}. If $W_{F}^{\chi}$ have (uniformly) $C^{r}$ leaves, then $\Phi_{\chi}$ is uniformly $C^{r}$ along $W_{F}^{\chi}$. If $\lambda\neq\chi$ are distinct Lyapunov functionals then $\Phi_{\chi}$ is constant along $W_{F}^{\lambda}$.
\end{lemma}
\begin{remark}
The proof of this lemma closely follows proofs from \cite{KatokLewis}.
\end{remark}
\begin{proof}
We will split the proof into three parts. First, we consider the case when the Lyapunov functional $\chi$ is either not proportional to the Lyapunov functional $\lambda$, or negatively proportional to $\lambda$ (the exponents $\chi$ and $\lambda$ are proportional if $\chi = c\lambda$ for some $c$, negatively proportional if $c < 0$, and positively proportional if $c > 0$). Second, we show that if $\lambda$ and $\chi$ are positively proportional, but not equal, then $\Phi_{\chi}$ is constant along $W_{F}^{\lambda}$. This finishes the proof of the second part of the lemma. Finally we show that when $\lambda = \chi$, then $\Phi_{\chi}$ is (uniformly) $C^{r}$ along $W_{F}^{\chi}$.

From the construction of $\Phi_{\chi,p}$ in Claim \ref{Claim:ExistenceFMcoordinates}, and $\Phi_{\chi,p}^{\lambda}$ in Equation \ref{Eq:OfFoliationFMCoordinates}, it is immediate that $\Phi_{\chi}$ is smooth along $W_{F}^{\chi}$ if and only if $\Phi_{\chi,p}^{\chi} = \Phi_{\chi,p}$ is smooth for each $p\in\mathbb{T}^{d}$, and $\Phi_{\chi}$ is constant along $W_{F}^{\lambda}$ if and only if $\Phi_{\chi,p}^{\lambda}$ is identically $0$ for every $p\in\mathbb{T}^{d}$. As a consequence, we can study the map $\Phi_{\chi,p}^{\lambda}$ instead of $\Phi_{\chi}$ for the remainder of the proof.

\textit{Independent or negatively proportional exponents}. Let $\lambda$ and $\chi$ be either independent (as elements of the vector space of linear functionals on $Z^{\infty}(f)\cong\mathbb{Z}^{k}$), or negatively proportional. Either way, we find some $g\in Z^{\infty}(f)$, homotopic to $A_{g}\in{\rm GL}(d,\mathbb{Z})$, such that $\chi(A_{g}) > 0$ and $\lambda(A_{g}) < 0$ (if $\chi$ and $\lambda$ are negatively proportional we can choose $g = f$ or $g = f^{-1}$, if they are independent then the set which satisfy $\chi(A_{g}) > 0$ and $\lambda(A_{g}) < 0$ define an open cone in $Z^{\infty}(f)\cong\mathbb{Z}^{k}$ which is, in particular, non-empty). We fix such a $g$. Let $\Psi_{\lambda,p}$ be the map from Equation \ref{Eq:NonStationaryNormalForm} and $\Phi_{\chi,p}^{\lambda}$ the map from Equation \ref{Eq:OfFoliationFMCoordinates}. For any $n\geq0$ we have
\begin{align*}
\Phi_{\chi,g^{n}p}^{\lambda}\left(\Psi_{\lambda,g^{n}p}(D_{p}g^{n}(v))\right) = \Phi_{\chi,g^{n}p}^{\lambda}\left(g^{n}\Psi_{\lambda,p}(v)\right) = e^{n\chi(A_{g})}\Phi_{\chi,p}^{\lambda}(\Psi_{\lambda,p}(v))
\end{align*}
or after dividing by $e^{n\chi(A_{g})}$
\begin{align}\label{Eq:FormulaForNotPosPropExpAlongLyapFoliation}
\Phi_{\chi,p}^{\lambda}(\Psi_{\lambda,p}(v)) = e^{-n\chi(A_{g})}\Phi_{\chi,g^{n}p}^{\lambda}\left(\Psi_{\lambda,g^{n}p}(D_{p}g^{n}(v))\right).
\end{align}
Note that $\norm{D_{p}g^{n}(v)}\to 0$ as $n\to\infty$ since, from Lemma \ref{L:PropLyapFunc}, $\lambda(A_{g}) < 0$ implies that $g$ contract $E_{f}^{\lambda}$. It follows that $\intd\left(\Psi_{\lambda,g^{n}p}(D_{p}g^{n}(v)),g^{n}p\right)\to0$, and since $\Phi_{\chi,g^{n}p}^{\lambda}(g^{n}(p)) = 0$ we conclude that $\Phi_{\chi,g^{n}p}^{\lambda}\left(\Psi_{\lambda,g^{n}p}(D_{p}g^{n}(v))\right)\to0$ as $n\to\infty$. Using that $\chi(A_{g}) > 0$ we obtain
\begin{align}
\Phi_{\chi,p}^{\lambda}(\Psi_{\lambda,p}(v)) = \lim_{n\to\infty}e^{-n\chi(A_{g})}\Phi_{\chi,g^{bn}p}^{\lambda}\left(\Psi_{\lambda,g^{n}p}(D_{p}g^{n}(v))\right) = 0
\end{align}
showing that $\Phi_{\chi,p}^{\lambda}$ identically vanishes along $W_{f}^{\lambda}(p)$ for every $p$. It follows that $\Phi_{\chi}$ is constant along $W_{F}^{\lambda}(p)$ for every $p\in\mathbb{R}^{d}$, finishing the first part of the proof.

\textit{Positively proportional exponents.}
We now turn to the case when $\lambda$ and $\chi$ are positively proportional, but not equal. Recall that $v:\mathbb{T}^{d}\to E_{f}^{\lambda}$ is a nowhere vanishing vector field. Let
\begin{align*}
X_{\pm} := \left\{p\in\mathbb{R}^{d}\text{ : }\Phi_{\chi,p}^{\lambda}(\Psi_{\lambda,p}(tv(p))) = \pm c_{p}^{\pm}|t|^{\eta_{p}^{\pm}},\quad t\in\mathbb{R}_{\pm}\right\}
\end{align*}
where $\mathbb{R}_{+} = \{t > 0\}$, $\mathbb{R}_{-} = \{t < 0\}$, $c_{p}^{\pm}\in E_{L}^{\chi}$ and $\eta_{p}^{\pm} > 0$. Our first goal is to show that $X_{\pm}$ (and $X_{+}\cap X_{-}$) are dense in $\mathbb{T}^{d}$.

We claim that every $Z^{\infty}(f)-$periodic point $p_{0}\in\mathbb{T}^{d}$ lie in $X_{\pm}$, which shows that $X_{+}$, $X_{-}$ and $X_{+}\cap X_{-}$ are all dense by Lemma \ref{L:DenseHigherRankPeriodicPoints}. It suffices to consider the case $X_{+}$, since if we replace the vector field $v$ by $-v$ then we transform $X_{+}$ to $X_{-}$. At any periodic point $p_{0}$ we define $\Gamma\leq Z^{\infty}(f)$ as the finite index subgroup that fix $p_{0}$ and preserve the orientation of $W_{f}^{\lambda}(p_{0})$. If $\lambda_{p_{0}}$ is the Lyapunov functional at $p_{0}$ along $W_{f}^{\lambda}$, then $\lambda_{p_{0}}$ is positively proportional to $\lambda$ by Lemma \ref{L:LyapFuncPosProp}, so since $\lambda$ is positively proportional to $\chi$ we find $c > 0$ such that $\chi = c\lambda_{p_{0}}$. We note, for future reference, that the constant $c > 0$ is not $1$ if $\chi$ and $\lambda$ are distinct, since $\lambda_{p_{0}}$ is close to $\lambda$ by Lemma \ref{L:PropLyapFunc}. Given any $g\in\Gamma$ we have
\begin{align*}
\Phi_{\chi,p_{0}}^{\lambda}\left(\Psi_{\lambda,p_{0}}\left(e^{\lambda_{p_{0}}(g)}v(p_{0})\right)\right) = & \Phi_{\chi,p_{0}}^{\lambda}\left(\Psi_{\lambda,p_{0}}\left(D_{p_{0}}g\left[v(p_{0})\right]\right)\right) = \\ &
\Phi_{\chi,p_{0}}^{\lambda}\left(g\Psi_{\lambda,p_{0}}\left(v(p_{0})\right)\right) = \\ &
e^{\chi(A_{g})}\Phi_{\chi,p_{0}}^{\lambda}\left(\Psi_{\lambda,p_{0}}\left(v(p_{0})\right)\right) = \\ &
e^{c\lambda_{p_{0}}(g)}\Phi_{\chi,p_{0}}^{\lambda}\left(\Psi_{\lambda,p_{0}}\left(v(p_{0})\right)\right)
\end{align*}
where the last equality use that $\chi$ and $\lambda_{p_{0}}$ are positively proportional. That is, we have an equality
\begin{align}\label{Eq:NormalFormEquationAtPerPoint}
\Phi_{\chi,p_{0}}^{\lambda}\left(\Psi_{\lambda,p_{0}}\left(e^{\lambda_{p_{0}}(g)}v(p_{0})\right)\right) = \left[e^{\lambda_{p_{0}}(g)}\right]^{c}\Phi_{\chi,p_{0}}^{\lambda}\left(\Psi_{\lambda,p_{0}}\left(v(p_{0})\right)\right).
\end{align}
By Lemma \ref{L:AlgebraicFullCentralizer} the kernel of $\chi$, and therefore also $\lambda_{p_{0}}$, is trivial. Since ${\rm rank}(\Gamma) = {\rm rank}(Z^{\infty}(f)) > 1$ it follows that the image of $\lambda_{p_{0}}$ is dense in $\mathbb{R}$. From Equation \ref{Eq:NormalFormEquationAtPerPoint} and density of ${\rm Im}(\lambda_{p_{0}})$ it follows that $p_{0}\in X_{+}$. The periodic points are dense in $\mathbb{T}^{d}$ by Lemma \ref{L:DenseHigherRankPeriodicPoints}, so $X_{+}$ is dense. For any $p\in X_{+}$ the function $c_{p}^{+}$ can be expressed as
\begin{align}
c_{p}^{+} = \Phi_{\chi,p}^{\lambda}(\Psi_{\lambda,p}(v(p)))
\end{align}
which implies, in particular, that $c_{p}^{+}$ vary continuously in $p$. Moreover, if $c_{p}^{+}\neq0$ then $c_{q}^{+}\neq0$ in some ball $q\in B_{r}(p)\cap X_{+}$, so we can solve for $\eta_{q}^{+}$ in $B_{r}(p)\cap X_{+}$, to see that $\eta_{q}^{+}$ vary continuously in $q\in B_{r}(q)\cap X_{+}$. In particular, if $p\in X_{+}$ is such that $c_{p}^{+}\neq0$ and $q_{n}\in B_{r}(p)\cap X_{+}$ is a sequence such that $q_{n}\to q\in B_{r}(p)\cap X_{+}$ then we also have $q\in X_{+}$. Since $X_{+}$ is dense this implies that $X_{+}\cap B_{r}(p) = B_{r}(p)$, so $p$ is an interior point in $X_{+}$. We fix $p\in X_{+}$ such that $c_{p}^{+}\neq0$ and let $r > 0$ be such that $B_{r}(p)\subset X_{+}$ and $c_{q}^{+}\neq0$ for all $q\in B_{r}(p)$. Given any $q\in W_{f}^{\lambda}(p)$ sufficiently close to $p$ we have $q\in B_{r}(p)$, so we can write
\begin{align*}
c_{q}^{+}t^{\eta_{q}^{+}} = & \Phi_{\chi,q}^{\lambda}(\Psi_{\lambda,q}(tv(q))) = \Phi_{\chi,q}^{\lambda}(\Psi_{\lambda,p}\Psi_{\lambda,p}^{-1}\Psi_{\lambda,q}(tv(q))) = \\ &
\Phi_{\chi,p}^{\lambda}(\Psi_{\lambda,p}\Psi_{\lambda,p}^{-1}\Psi_{\lambda,q}(tv(q))) + \Phi_{\chi}(p) - \Phi_{\chi}(q) = \\ &
\Phi_{\chi,p}^{\lambda}(\Psi_{\lambda,p}\left(\left[(\sigma_{q,p}t + s_{q,p})v(p)\right]\right) + \Phi_{\chi}(p) - \Phi_{\chi}(q) = \\ &
c_{p}^{+}\left[\sigma_{q,p}t + s_{q,p}\right]^{\eta_{p}^{+}} + \Phi_{\chi}(p) - \Phi_{\chi}(q)
\end{align*}
where we have used that $\Psi_{\lambda,p}$ defines an affine structure on $W_{f}^{\lambda}$. Considering the asymptotics as $t\to\infty$ the equality
\begin{align}\label{Eq:AsymptoticEquation}
c_{q}^{+}t^{\eta_{q}^{+}} = c_{p}^{+}\left[\sigma_{q,p}t + s_{q,p}\right]^{\eta_{p}^{+}} + \Phi_{\chi}(p) - \Phi_{\chi}(q)
\end{align}
implies that $\eta_{p}^{+} = \eta_{q}^{+}$ and $c_{q}^{+} = c_{p}^{+}\sigma_{q,p}^{\eta_{p}^{+}}$. Moreover, if $a\in\mathbb{R}_{>0}$ then the function
\begin{align*}
\mathbb{R}_{>0}\ni x\mapsto(x + a)^{\eta_{p}^{+}} - x^{\eta_{p}^{+}}
\end{align*}
is constant if and only if $\eta_{p}^{+} = 1$ (note that if $\eta_{p}^{+}\neq1$ then the derivative in $x$ is not $0$). Using this in Equation \ref{Eq:AsymptoticEquation}, we conclude that $\eta_{p}^{+} = 1$. As noted above, if $p\in\mathbb{T}^{d}$ is periodic, then $\eta_{p}^{+}\neq1$ so the set of $p\in X_{+}$ such that $c_{p}^{+}\neq0$ is open and disjoint from the $Z^{\infty}(f)-$periodic points. We conclude that $c_{p}^{+} = 0$ for every $p\in X_{+}$ since the periodic points are dense by Lemma \ref{L:DenseHigherRankPeriodicPoints}. Since $\Phi_{\chi,p}^{\lambda}$ is zero when $c_{p}^{+} = c_{p}^{-} = 0$, and since $X_{+}\cap X_{-}$ is dense we conclude that $\Phi_{\chi,p}^{\lambda}$ is identically zero for every $p\in\mathbb{T}^{d}$.

\textit{Equal exponents.} If $\lambda = \chi$ then $\Phi_{\chi,p}^{\chi} = \Phi_{\chi,p}^{\chi}$ is invertible along $W_{f}^{\chi}(p)$ by Claim \ref{Claim:ExistenceFMcoordinates}. It follows that $c_{p}^{\pm}\neq0$ for all $p\in X_{\pm}$. In fact, since 
\begin{align*}
c_{p}^{\pm} = \Phi_{\chi,p}^{\lambda}(\Psi_{\chi,p}(\pm v(p)))
\end{align*}
it holds that $c_{p}^{\pm}$ are uniformly bounded away from $0$. We conclude that we can always solve for $\eta_{p}^{\pm}$, and the argument above shows that $\eta_{p}^{\pm} = 1$. It follows that $c_{p}^{\pm}$ and $\eta_{p}^{\pm}$ both vary continuously in $p$, so $X_{\pm}$ are both closed sets. Since $X_{\pm}$ are both also dense we have $X_{\pm} = \mathbb{T}^{d}$. Since $\Psi_{\chi,p}$ is smooth the formulas
\begin{align}
tc_{p}^{\pm} = \Phi_{\chi,p}(\Psi_{\chi,p}(tv(p)))
\end{align}
implies that $\Phi_{\chi,p}$ is smooth along $W_{f}^{\chi}$ at all $q\in W_{f}^{\chi}(p)\setminus\{p\}$ for every $p\in\mathbb{T}^{d}$. By choosing some $q\in W_{F}^{\chi}(p)\setminus\{p\}$ it also follows that $\Phi_{\chi}$ is smooth along $W_{f}^{\chi}$ at $p$. But this can only hold if $c_{p}^{+} = c_{p}^{-} = c_{p}$, so we have
\begin{align}
c_{p}t = \Phi_{\chi,p}(\Psi_{\chi,p}(tv(p))),\quad t\in\mathbb{R}
\end{align}
showing that $\Phi_{\chi,p}$ is uniformly $C^{r}$ along $W_{f}^{\chi}$.
\end{proof}
\begin{proof}[Proof of Claim \ref{Claim:RegularityCoordinates}]
Let $\chi\in{\rm Lyap}(L)$. Assume, without loss of generality, that $\chi > 0$. We begin by proving that $\Phi_{\chi}$ is smooth along $W_{F}^{cs}$. We claim that $\Phi_{\chi}$ is constant along $W_{F}^{cs}$, so in particular smooth. Given $y\in W_{F}^{cs}(x)$
\begin{align*}
\norm{\Phi_{\chi}(y) - \Phi_{\chi}(x)} = & e^{-n\chi}\norm{\Phi_{\chi}(F^{n}y) - \Phi_{\chi}(F^{n}x)}\leq \\ &
e^{-n\chi}\norm{(F^{n}y)^{\chi} - (F^{n}x)^{\chi}} + 2e^{-n\chi}\norm{\varphi_{\chi}}_{C^{0}}\leq \\ &
Ce^{-n\chi}\cdot\left(\intd_{cs}(F^{n}x,F^{n}y) + 1\right)\to 0
\end{align*}
where the last inequality use Lemma \ref{L:QuasiIsometryForPert}. To show that $\Phi_{\chi}$ is smooth, it suffices, by Journé's lemma \cite{Journe}, to show that $\Phi_{\chi}$ is smooth along $W_{F}^{u}$. Let $r$ be the number from the $r-$spread assumption, or $1 + \alpha$ if we do not assume that the spectrum of $L$ is $r-$spread. We get regularity for $\Phi_{\chi}$ along $W_{F}^{u}$ in three steps, first we show that $\Phi_{\chi}$ is $C^{\infty}$ along $W_{F}^{\lambda > \chi}$. Second we show that $\Phi_{\chi}$ is $C^{r}$ along $W_{F}^{0 < \lambda < \chi}$. Finally we show that $\Phi_{\chi}$ is $C^{r}$ along $W_{F}^{\chi}$. Applying Journé's lemma we see that $\Phi_{\chi}$ is $C^{r}$ along $W_{F}^{0 < \lambda\leq\chi}$ and applying Journé's lemma once more we see that $\Phi_{\chi}$ is $C^{r}$ along $W_{F}^{u}$.

Since the foliations $W_{F}^{\lambda}$, $\lambda > \chi$, gives a product structure in $W_{F}^{\lambda > \chi}$ and $\Phi_{\chi}$ is constant along each $W_{F}^{\lambda}$, Lemma \ref{L:RegularityOfFMalongLeaves}, it follows that $\Phi_{\chi}$ is constant, and therefore smooth, along $W_{F}^{\lambda > \chi}$ (note that the leaves $W_{F}^{\lambda > \chi}$ are unstable manifolds for a partially hyperbolic diffeomorphism and therefore $C^{\infty}$, so $\Phi_{\chi}$ is $C^{\infty}$ along $W_{F}^{\lambda > \chi}$). Similarly we see that $\Phi_{\chi}$ is constant, and therefore also $C^{r}$ along $W_{F}^{\lambda < \chi}$ (since the leaves $W^{0 < \lambda < \chi}$ are not necessarily more regular than $C^{r}$, we obtain regularity from Lemma \ref{L:SmoothnesSpreadSpec}, we do not get $C^{\infty}$ regularity along $W_{F}^{0 < \lambda < \chi}$). That $\Phi_{\chi}$ is $C^{r}$ along $W_{F}^{\chi}$ is immediate from Lemmas \ref{L:RegularityOfFMalongLeaves} and \ref{L:SmoothnesSpreadSpec}. The claim now follows from Journé's lemma \cite{Journe}.
\end{proof}
We now prove Theorem \ref{Thm:LeafConjugacy}.
\begin{proof}[Proof of Theorem \ref{Thm:LeafConjugacy}]
By Claim \ref{Claim:RegularityCoordinates} each $\Phi_{\chi}:\mathbb{R}^{d}\to E_{L}^{\chi}$ is $C^{1+\alpha}$, constant along $W_{F}^{\chi'},W_{F}^{c}$ with $\chi\neq\chi'$ and has non-vanishing derivative along $W_{F}^{\chi}$. We define
\begin{align*}
H_{{\rm leaf}}(p) = \sum_{0\neq\chi\in{\rm Lyap}(L)}\Phi_{\chi}(p) + p^{c}
\end{align*}
and note that for $n\in\mathbb{Z}^{d}$
\begin{align*}
H_{{\rm leaf}}(p + n) = & \sum_{0\neq\chi\in{\rm Lyap}(L)}\Phi_{\chi}(p + n) + (p+n)^{c} = \\ &
\sum_{0\neq\chi\in{\rm Lyap}(L)}\left(\Phi_{\chi}(p) + n^{\chi}\right) + p^{c} + n^{c} = \\ = &
H_{{\rm leaf}}(p) + n
\end{align*}
so $H_{{\rm leaf}}$ descends to a $C^{1+\alpha}-$map of $\mathbb{T}^{d}$ homotopic to identity. Moreover, since $D_{p}H_{{\rm leaf}}$ is invertible for all $p\in\mathbb{T}^{d}$ and $H_{\rm leaf}$ has degree $1$ it follows that $H_{{\rm leaf}}$ is in fact a $C^{1+\alpha}-$diffeomorphism. Each $\Phi_{\chi}$ is constant along $W_{F}^{c}(p)$ so for $q\in W_{F}^{c}(p)$ we have
\begin{align*}
H_{{\rm leaf}}(q) = H_{{\rm leaf}}(p) + (q - p)^{c}\in H_{{\rm leaf}}(p) + E_{L}^{c}
\end{align*}
or $H_{\rm leaf}(W_{F}^{c}(x)) = H_{\rm leaf}(x) + E_{L}^{c}$. Finally, for any $x\in\mathbb{R}^{d}$
\begin{align*}
H_{\rm leaf}(Fx) = & \sum_{0\neq\chi\in{\rm Lyap}(L)}\Phi_{\chi}(Fx) + (Fx)^{c} = \\ &
\sum_{0\neq\chi\in{\rm Lyap}(L)}L\Phi_{\chi}(x) + (Lx + v(x))^{c} = \\ &
\left[\sum_{0\neq\chi\in{\rm Lyap}(L)}L\Phi_{\chi}(x) + Lx^{c}\right] + v^{c}(x) = \\ &
LH_{\rm leaf}(x) + v^{c}(x)\in LH_{\rm leaf}(x) + E_{L}^{c}
\end{align*}
where $v:\mathbb{T}^{d}\to\mathbb{R}^{d}$ is defined such that $Fx = Lx + v(x)$. So, since 
\begin{align*}
H_{\rm leaf}(W_{F}^{c}(x)) = H_{\rm leaf}(x) + E_{L}^{c},
\end{align*}
we have $H_{\rm leaf}(W_{F}^{c}(Fx)) = L(H_{\rm leaf}(x) + E_{L}^{c})$.
\end{proof}

\subsection{Proof of Theorem \ref{Thm:SymplecticRigidity}}

In this section, we improve the result of the previous section when the spectrum of $L$ is $3-$spread proving Theorem \ref{Thm:SymplecticRigidity}.
\begin{proof}[Proof of Theorem \ref{Thm:SymplecticRigidity}]
By Claim \ref{Claim:RegularityCoordinates} each $\Phi_{\chi}$ is $C^{3-\varepsilon}$ which implies that we can define $C^{2-\varepsilon}$ non-vanishing $1-$forms $\omega_{\chi} = \intd\Phi_{\chi}$ on $\mathbb{T}^{d}$. Since $E_{f}^{c}$ is the intersections of the kernels of all $\omega_{\chi}$, $E_{f}^{c}$ is also $C^{2-\varepsilon}$. Since $E_{f}^{s}\oplus E_{f}^{u}$ is the symplectic dual of $E_{f}^{c}$ it follows that $E_{f}^{s}\oplus E_{f}^{u}$ is also $C^{2-\varepsilon}$. That $f$ is $C^{\infty}-$conjugated to $L$ follows from $(6)$ in \cite[Corollary 1.8]{CenterFoliationRig}.
\end{proof}

%% file: PropertiesOfPHAut.tex
\newpage

\appendix

\section{Some algebraic properties of partially hyperbolic automorphisms of the torus}
\label{Sec:AlgebraicProperties}

In this appendix, we prove some basic properties of automorphisms of the torus, some of which are well-known. For a more thorough treatment of centralizers of toral automorphisms, see \cite{Baake2001,Baake2013}. We will assume throughout the appendix that $d > 2$. Before starting the proofs, recall that any irreducible $L\in{\rm GL}(d,\mathbb{Z})$ is determined by its characteristic polynomial. This is essentially shown in \cite{CentAut}. Let $L\in{\rm GL}(d,\mathbb{Z})$ be irreducible. If $p\in\mathbb{Q}[t]$ is a rational polynomial and $p(L)v = 0$ for some $0\neq v\in\mathbb{Q}^{d}$ then $\ker p(L)$ is a rational non-trivial $L-$invariant subspace. But if $L$ is irreducible, then any rational non-trivial $L-$invariant subspace is $\mathbb{R}^{d}$ so $\ker p(L) = \mathbb{R}^{d}$ or $p(L) = 0$. Given some $0\neq v_{0}\in\mathbb{Q}^{d}$ we have a natural map $\eta:\mathbb{Q}[L]\to\mathbb{Q}^{d}$ defined by $p(L)\mapsto p(L)v_{0}$, and since $p(L)v_{0} = 0$ if and only if $p(L) = 0$ it follows that this map is injective. The dimension of $\mathbb{Q}[L]$ as a vector space over $\mathbb{Q}$ is $d$, so the map $p(L)\mapsto p(L)v_{0}$ is an isomorphism. One sees that $\eta$ conjugates $L:\mathbb{Q}^{d}\to\mathbb{Q}^{d}$ to the multiplication action $p(L)\mapsto L\cdot p(L)\in\mathbb{Q}[L]$. Finally, $\mathbb{Q}[L]$ is isomorphic to the number field $K = \mathbb{Q}[t]/\langle p_{L}(t)\rangle$ where $p_{L}(t)$ is the characteristic polynomial of $L$. It follows that $L$ is conjugate (over the rationals) to the map $K\to K$ defined by $k\mapsto t\cdot k$.

We begin by showing that irreducible automorphisms with $2-$dimensional isometric center are \textit{pseudo Anosov} in the sense of \cite{StabelErgodicity}.
\begin{lemma}\label{L:PseudoIsAllTrue}
If $L\in\rm GL(d,\mathbb{Z})$ is irreducible with precisely two eigenvalues on $S^{1}$ then $p_{L}(t)$ can not be written as a polynomial in $t^{n}$ for any $n\geq 2$.
\end{lemma}
\begin{proof}
Let $\lambda,\overline{\lambda}\in S^{1}\subset\mathbb{C}$ be the roots of $p_{L}(t)$ on the unit circle. Let $Q(t)$ be such that $p_{L}(t) = Q(t^{n})$ and let $\zeta_{n}$ be a $n'$th root of unity. We have 
\begin{align*}
p_{L}(\zeta_{n}\lambda) = Q((\zeta_{n}\lambda)^{n}) = Q(\lambda^{n}) = p_{L}(\lambda) = 0
\end{align*}
which implies, since $L$ has only two eigenvalues on $S^{1}$, that $\zeta_{n}\lambda = \lambda$ or $\zeta_{n}\lambda = \overline{\lambda}$. The first case implies $\zeta_{n} = 1$. For the other case, we take everything to the $n'$th power and obtain $\lambda^{n} = \lambda^{-n}$ or $\lambda^{2n} = 1$. So $\lambda$ is a root of unity. If $\lambda$ is a root of unity, then the irreducibility of $p_{L}(t)$ implies that every root of $p_{L}(t)$ is a root of unity. This is a contradiction since $p_{L}(t)$ only have two roots on $S^{1}$ and $d > 2$. We conclude that if $p_{L}(t) = Q(t^{n})$ then $n = 1$.
\end{proof}
\begin{lemma}\label{L:AlgebraicFullCentralizer}
Let $L\in{\rm GL}(d,\mathbb{Z})$ be irreducible with affine centralizer $Z_{\rm Aut}(L)$. For any subgroup $\Gamma\leq Z_{\rm Aut}(L)$ we define
\begin{align*}
\Lambda:\Gamma\to\mathbb{R}^{r_{1} + r_{2} - 1},\quad\Lambda(\gamma) = (\chi_{1}(\gamma),...,\chi_{r_{1} + r_{2}-1}(\gamma))
\end{align*}
where each $\chi_{j}$ is the Lyapunov functional along some eigendirection. If $\Gamma\leq Z_{\rm Aut}(L)$ has rank $r_{1} + r_{2} - 1$ then the image of $\Lambda$ is a lattice. Moreover, if $\chi_{j}$ correspond to a real eigenvalue and $\chi_{j}(\gamma) = 0$ then $\gamma = \pm I$.
\end{lemma}
\begin{proof}
This lemma is essentially a reformulation of the Dirichlet unit theorem. If $p_{L}(t)$ is the characteristic polynomial of $L$, then we obtain the corresponding number field $K = \mathbb{Q}[t]/\langle p_{L}(t)\rangle$. The action of $L$ on $\mathbb{T}^{d}$ is, up to finite index, conjugated to the action of $q(t)\mapsto tq(t)$ on $(K/\mathcal{O}_{K})\otimes\mathbb{R}$ where $\mathcal{O}_{K}$ are the integers in $K$. We can describe $Z_{\rm Aut}(L)$ in these coordinates, it corresponds to the units $U_{K}$ of $\mathcal{O}_{K}$ and they act on $K/\mathcal{O}_{K}$ by $q(t)\mapsto u\cdot q(t)$ for $u\in U_{K}$ \cite{WangNumberTheoryDescription}. Let $\sigma_{1}^{\mathbb{R}},...,\sigma_{r_{1}}^{\mathbb{R}}:K\to\mathbb{R}$ be the real embeddings of $K$ and $\sigma_{1}^{\mathbb{\mathbb{C}}},...,\sigma_{r_{2}}^{\mathbb{\mathbb{C}}}:K\to\mathbb{C}$ complex embeddings such that any other complex embedding can be written as $\sigma = \overline{\sigma_{j}}^{\mathbb{C}}$ for some $j = 1,...,r_{2}$. The eigenvalues of $u\in U_{K}$ considered as an automorphism on $K/\mathcal{O}_{K}$ are given by $\{\sigma_{i}^{\mathbb{R}}(u),\sigma_{j}^{\mathbb{C}}(u),\overline{\sigma}_{j}^{\mathbb{C}}(u)\}_{i,j}$. If we define
\begin{align*}
& \Lambda':U_{K}\to\mathbb{R}^{r_{1} + r_{2}},\\
& \Lambda'(u) = (\log|\sigma_{1}^{\mathbb{R}}(u)|,...,\log|\sigma_{r_{1}}^{\mathbb{R}}(u)|,\log|\sigma_{1}^{\mathbb{C}}(u)|,...,\log|\sigma_{r_{2}}^{\mathbb{C}}(u)|)
\end{align*}
then $\Lambda'$ take values in $V = \{x_{1}+...+x_{r_{1}} + 2y_{1} + ... + 2y_{r_{2}} = 0\}$ (this can be seen by noting that $\eta:U_{K}\to\mathbb{R}\setminus0$, $\eta(u) = \sigma_{1}^{\mathbb{R}}(u)\cdot...\cdot\sigma_{r_{1}}^{\mathbb{R}}(u)\cdot|\sigma_{1}^{\mathbb{C}}(u)|^{2}\cdot...\cdot|\sigma_{r_{2}}^{\mathbb{C}}(u)|^{2}$ is a homomorphism taking values in $\mathbb{Z}$ so it must only take values $\pm1$). So we can define $\Lambda:U_{K}\to V$ and the (proof of) Dirichlet unit theorem \cite{SerreDirichletUnit} implies that the image of $\Lambda$ is a lattice in $V$.

To see the last part, suppose that $\gamma$ satisfies $\chi_{j}(\gamma) = 0$. Let $v_{\chi}$ be an eigendirection for $E_{L}^{\chi}$ then we have $\gamma v_{\chi} = \pm v_{\chi}$. So $\ker(\gamma\mp I)$ is non-empty and rational. Irreducibility of $L$ imply that $\ker(\gamma\mp I) = \mathbb{R}^{d}$ so $\gamma(x) = \pm x$.
\end{proof}
\begin{lemma}\label{L:NecessityOfLargeAction}
Let $L\in{\rm GL}(d,\mathbb{Z})$ be irreducible and ergodic with precisely one pair of eigenvalues on the unit circle. Denote by $E_{L}^{c}$ the eigendirection for $L$ corresponding to the eigenvalues on the unit circle. For any $r < r_{1} + r_{2} - 1$ and $Q > 0$ there is a subgroup $\Gamma\leq Z_{\rm Aut}(L)$ such that $\Gamma\cong\mathbb{Z}^{r}$ and for all $\gamma\in\Gamma$ we have
\begin{align}\label{Eq:CounterExampleEquation}
|\det(\gamma|_{E_{L}^{c}})|\leq\left|\frac{\det(\gamma|_{E_{\gamma}^{u}})}{\det(\gamma|_{E_{L}^{c}})}\right|^{Q}
\end{align}
\end{lemma}
\begin{proof}
It suffices to show the lemma with $r = r_{1} + r_{2} - 2$. We number the Lyapunov functionals of $\Gamma$ by $\chi_{0},\chi_{1},...,\chi_{n}:Z_{\rm Aut}(L)\to\mathbb{R}$ where $n = r_{1} + r_{2} - 1$ and $\chi_{0}$ is the Lyapunov functional along $E_{L}^{c}$. We will, for the remainder of the proof, identify $Z_{\rm Aut}(L)$ with a lattice in $\mathbb{R}^{n}$, Lemma \ref{L:AlgebraicFullCentralizer}. We let $d_{j}$ be the dimension of the Lyapunov space corresponding to the Lyapunov functional $\chi_{j}$. In particular, $d_{0} = 2$. If we take the logarithm of Equation \ref{Eq:CounterExampleEquation} then we obtain
\begin{align}\label{Eq:CounterExample1}
2\chi_{0}(\gamma)\leq Q\sum_{\chi_{j}(\gamma) > 0}d_{j}\chi_{j}(\gamma) - 2Q\chi_{0}(\gamma).
\end{align}
Using that each $\gamma\in\Gamma$ satisfy $|\det(\gamma)| = 1$ we have
\begin{align*}
\sum_{j = 0}^{n}d_{j}\chi_{j}(\gamma) = 0,\quad\sum_{\chi_{j}(\gamma) > 0}d_{j}\chi_{j}(\gamma) = \frac{1}{2}\sum_{j = 0}^{n}d_{j}|\chi_{j}(\gamma)|.
\end{align*}
So equation \ref{Eq:CounterExample1} can be written
\begin{align}\label{Eq:CounterExample2}
\chi_{0}(\gamma)\leq\frac{Q}{4(Q+1)}\cdot\sum_{j = 0}^{n}d_{j}|\chi_{j}(\gamma)|.
\end{align}
We define a norm on $\mathbb{R}^{n}$ by
\begin{align*}
\norm{v}_{L} = \sum_{j = 0}^{n}d_{j}|\chi_{j}(v)|.
\end{align*}
Let $W_{0} = \ker\chi_{0}\subset\mathbb{R}^{n}$ be the kernel of $\chi_{0}$. We choose a linear functional $\omega\in(\mathbb{R}^{n})^{*}$ close to $0$ such that $W := \ker(\chi_{0} + \omega)\subset\mathbb{R}^{n}$ is rational and close to $W_{0}$. Since $W\subset\mathbb{R}^{n}$ is rational of dimension $n-1$ the intersection $\Gamma = W\cap Z_{\rm Aut}(L)$ has rank $n-1$. For any $\gamma\in\Gamma$ we have
\begin{align*}
\chi_{0}(\gamma)\leq|\chi_{0}(\gamma)| = |\omega(\gamma)|\leq\norm{\omega}\cdot\norm{\gamma}_{L} = \norm{\omega}\cdot\sum_{j = 0}^{n}d_{j}|\chi_{j}(v)|
\end{align*}
so with $\norm{\omega}\leq Q/4(Q+1)$ the lemma follows.
\end{proof}

\subsection{Automorphisms with property (P)}

In this section, we show some basic properties of automorphisms with property $(P)$ (see Definition \ref{Def:PropertyP}).
\begin{lemma}\label{L:PropertiesOfP}
Let $L\in{\rm GL}(d,\mathbb{Z})$ have property $(P)$. The following properties hold
\begin{enumerate}[label = (\roman*)]
    \item $L$ is ergodic, 
    \item $p_{L}(t)$ is not a polynomial in $t^{n}$ for $n\geq 2$,
    \item if $A\in Z_{{\rm GL}(d,\mathbb{Z})}(L)$ lie in the centralizer of $L$, $A$ is not hyperbolic and $A\neq\pm I$ then $A$ is irreducible,
    \item the real eigenvalues of $L$ can be  written as $\lambda_{1},...,\lambda_{N},\lambda_{1}^{-1},...,\lambda_{N}^{-1}$ where $|\lambda_{1}| > |\lambda_{2}| > ... > |\lambda_{N}| > 1$ for $N = (d-2)/2$,
    \item there is a rational symplectic structure $\omega$ such that $L^{*}\omega = \omega$.
\end{enumerate}
\end{lemma}
\begin{proof}
The real eigenvalues of $L$ can not be $\pm1$ since that would contradict that $L$ was irreducible. If $L$ had an eigenvalue that was a root of unity, then irreducibility of $L$ implies that all roots of $p_{L}(t)$ are roots of unity. This is a contradiction since $L$ only has $2$ non-real eigenvalues.

Property $(ii)$ follows from Lemma \ref{L:PseudoIsAllTrue}.

For $(iii)$, let $A\in{\rm GL}(d,\mathbb{Z})$ commute with $L$. If $\pm1$ is an eigenvalue for $A$ then $\ker(A\pm I)$ defines a rational subspace invariant by $L$ and by irreducibility of $L$ this implies that $\ker(A\pm I) = \mathbb{R}^{d}$. So if $A$ is not equal to $\pm I$ and not hyperbolic, then $A$ must have at least one pair of complex eigenvalues. Since $A$ preserves the eigendirections of $L$, it follows that $A$ has $d-2$ real eigenvalues and the pair of complex eigenvalues for $A$ must be along $E_{L}^{c}$. Suppose that we can write $p_{A}(t) = P(t)Q(t)$ where we assume that the eigenvalues of $A$ along $E_{L}^{c}$ is a root of $Q(t)$. It follows that $P(t)$ has only real roots. Let $W = \ker P(A)$, which is not $\mathbb{R}^{d}$ since $E_{L}^{c}\not\subset W$. Since $P(t)$ is rational, $W$ is rational and since $L$ commutes with $P(A)$, $L$ preserves $W$. It follows by the irreducibility of $L$ that $W = 0$. This can only hold if $P(t)$ is a constant polynomial, so $p_{A}(t)$ is indeed irreducible in $\mathbb{Q}[t]$.

Part $(iv)$ follows from the fact that $L$ is irreducible with an eigenvalue on $S^{1}$. Let $\omega\in S^{1}$ be one of the roots of $p_{L}(t)$ on $S^{1}$. Then $\overline{\omega}$ is also a root, so $\omega$ is a root of $q(t) := t^{d}p_{L}(t^{-1})$. Since the polynomials $q(t)$ and $p_{L}(t)$ have the same degree, $p_{L}(t)$ is irreducible and both $p_{L}(t)$ and $q(t)$ are monic it follows that $q(t) = p_{L}(t)$. So the roots of $p_{L}(t)$ comes in pairs $\lambda,\lambda^{-1}$. To see that no two distinct real roots of $p_{L}(t)$ have the same modulus, we note that if $|\lambda| = |\mu|$ for real $\lambda\neq\mu$ then $\lambda = -\mu$. It follows that $\lambda$ is a root of $q(t) = p_{L}(-t)$ and again using irreducibility of $p_{L}(t)$ and the fact that $\deg(q(t)) = \deg(p_{L}(t))$ with both $q(t)$ and $p_{t}(t)$ monic it follows that $p_{L}(t) = p_{L}(-t)$. But if $p_{L}(t) = p_{L}(-t)$ then all odd coefficients in the polynomial $p_{L}(t)$ vanish so $p_{L}(t)$ is a polynomial in $t^{2}$ contradicting $(ii)$.

The last part follows from \cite[Theorem A.1]{ExistenceSymplecticStructureForInteger} and the fact that any two $L\in{\rm GL}(d,\mathbb{Z})$ with the same irreducible characteristic polynomial are conjugate over the rationals.
\end{proof}
\begin{remark}
In the remainder, if we have an automorphism with property $(P)$ then we will always order the real eivenvalues as in $(iv)$. Moreover we order the Lyapunov exponents similarly with $\chi_{j} = \log|\lambda_{j}|$.
\end{remark}
\begin{lemma}\label{L:IrreducibleSymPoly}
If $p(t)\in\mathbb{Q}[t]$ is irreducible with one root $\lambda\in(-2,2)$ then $q(t) = t^{d}p(t + t^{-1})$, $d = \deg(p(t))$, is irreducible.
\end{lemma}
\begin{proof}
Let $p(t)$ and $q(t)$ be as in the lemma. We find some $\omega\in S^{1}\setminus\{\pm1\}$ such that $\omega + 1/\omega = \lambda\in(-2,2)$ and $q(\omega) = 0$. Let $q(t) = Q_{1}(t)Q_{2}(t)...Q_{N}(t)$, $Q_{1}(t),...,Q_{N}(t)\in\mathbb{Q}[t]$, be the decomposition of $q(t)$ into irreducible factors. We assume that we have numbered the factors of $q(t)$ such that $\omega$ is a root of $Q_{1}(t)$. Since $Q_{1}(t)$ is irreducible with a root on $S^{1}$, it follows that $\deg(Q_{1}(t)) = 2n$ is even and we have some $P(z)\in\mathbb{Q}[z]$ such that $Q_{1}(t) = t^{n}P(t + t^{-1})$. We note that $P(z)$ is irreducible since if $P(z)$ splits into a product, then so would $Q_{1}(t)$, which would contradict the irreducibility of $Q_{1}(t)$. In particular, this implies that $P(t)$ only has roots away from $0$. We claim that $P(t)$ divides $p(t)$. To see this, let $p(z)/P(z) = \gamma(z) + r(z)/P(z)$ where $\deg(r(z)) < \deg(P(z))$ and $\gamma(z)$ is some polynomial in $\mathbb{Q}[z]$ of degree $d - n$. It follows that
\begin{align*}
Q_{2}(t)...Q_{N}(t) = & \frac{q(t)}{Q_{1}(t)} = t^{d-n}\cdot\frac{p(t + t^{-1})}{P(t + t^{-1})} = \\ &
t^{d-n}\cdot\left(\gamma(t + t^{-1}) + \frac{r(t + t^{-1})}{P(t + t^{-1})}\right).
\end{align*}
Since $\deg(r(z)) < \deg(P(z))$ we find some root, $z_{0}\neq 0$, of $P(z)$ such that $\lim_{z\to z_{0}}r(z)/P(z) = \infty$ provided that $r(z)$ is not identically $0$. Since the map $z\mapsto z + 1/z$ is a surjective map on the Riemannsphere that map $0$ and $\infty$ to $\infty$ it follows that we find $w_{0}\neq0,\infty$ such that $z_{0} = w_{0} + 1/w_{0}$. Since $w_{0}\neq0,\infty$ and $t^{d-n}\gamma(t + t^{-1})$ is a polynomial
\begin{align*}
\infty = & \lim_{w\to w_{0}}|w^{d-n}|\cdot\left|\frac{r(w + w^{-1})}{P(w + w^{-1})}\right| = \\ &
\lim_{w\to w_{0}}\left|Q_{2}(w)...Q_{N}(w) - w^{d-n}\gamma(w + w^{-1})\right| = \\ &
\left|Q_{2}(w_{0})...Q_{N}(w_{0}) - w_{0}^{d-n}\gamma(w_{0} + w_{0}^{-1})\right| < \infty
\end{align*}
which gives a contradiction. It follows that $r(z)$ is identically zero. In extension $P(z)$ divides $p(z)$. By irreducibility of $p(z)$, it follows that $p(z) = c\cdot P(z)$ which implies that $q(t)$ is proportional to $Q_{1}(t)$ and $q(t)$ is irreducible.
\end{proof}
\begin{lemma}\label{L:ExistenceSpreadSpectrum}
For any even $d\geq 4$ and any $r\in\mathbb{N}$ there is $L\in{\rm GL}(d,\mathbb{Z})$ with property $(P)$ and satisfying $\chi_{j} > r\chi_{j+1}$ for $j = 1,...,(d-2)/2 - 1$.
\end{lemma}
\begin{proof}
Let $d = 2n$ and let $K$ be a degree $n$ totally real number field with (real) embeddings $\sigma_{1},...,\sigma_{n}:K\to\mathbb{R}\subset\mathbb{C}$. From (the proof of) the Dirichlét unit theorem \cite{SerreDirichletUnit}, the logarithmic embedding
\begin{align*}
& \Lambda:U_{K}\to V = \{(x_{1},...,x_{n})\text{ : }x_{1}+...+x_{n} = 0\}\subset\mathbb{R}^{n}, \\
& \Lambda(u) = (\log|\sigma_{1}(u)|,...,\log|\sigma_{n}(u)|)
\end{align*}
defines a lattice in $V$. That is, $\Lambda(U_{K})\subset V$ is a lattice. We define the open set $W\subset V$
\begin{align*}
W := \{&x_{j} > rx_{j+1},\text{ }x_{n-1} > \log(2)\text{ : }j = 1,...,n-2,\text{ }\\ &
x_{n} = -x_{1}-...-x_{n-1}\}.
\end{align*}
The open set $W$ intersects $\Lambda(U_{K})$. Indeed, $W$ is non-empty and open, and if we have some $w\in W$, then $Rw\in W$ for all $R > 0$. Moreover, if $\rho(w) > 0$ is the largest radius of a ball such that $B_{\rho(w)}(w)\subset W$, then it is clear that $\rho(Rw)\to\infty$ as $R\to\infty$. It follows that $W$ contains balls of arbitrarily large radius. There is some radius $R_{0} > 0$ such that $\Lambda(U_{K})$ intersect any ball of radius $\geq R_{0}$ since $\Lambda(U_{K})$ is a lattice. We conclude that $\Lambda(U_{K})\cap W\neq\emptyset$. Let $u\in U_{K}$ be such that $\Lambda(u)\in W$. The minimal polynomial, $p(t)$, of $u$ is irreducible. Indeed, suppose that $p(t) = r_{1}(t)r_{2}(t)$ splits. Since $p(t)$ is primitive over $\mathbb{Z}$, we may assume that both $r_{j}(t)$ is defined over $\mathbb{Z}$ and since the constant and highest term of $p(t)$ are both $1$ it follows that the same must hold for both $r_{j}(t)$. Now, $p(t)$ only has one root with modulus less than $1$, so one of $r_{j}(t)$ only has roots with modulus larger than $1$, suppose that $r_{2}(t)$ only have roots of modulus larger than $1$. If $r_{2}(t)$ is not constant, then its constant term is the product of its roots, but $r_{2}(t)$ by assumption have only roots of modulus larger. We conclude that $r_{2}(t)$ is constant, so $p(t)$ is irreducible. We can write $p(t)$ as
\begin{align*}
p(t) = t^{n} + c_{n-1}t^{n-1}+...+c_{1}t\pm1.
\end{align*}
We define $q(t) = t^{n}p(t + t^{-1})\in\mathbb{Z}[t]$ which has the form
\begin{align*}
q(t) = t^{d} + c_{d-1}'t^{d-1} +...+1.
\end{align*}
If we denote by $L$ the companion matrix of $q(t)$, then $L\in\rm GL(d,\mathbb{Z})$. Since $p(t)$ has only real roots and one root with modulus less than $2$, it follows that $q(t)$ has one pair of complex roots on $S^{1}$ and all other roots real. Moreover, by Lemma \ref{L:IrreducibleSymPoly}, the polynomial $q(t)$ is irreducible. Finally, the real roots of $q(t)$ that are larger than $1$ in modulus satisfy $|\lambda_{j}| + |\lambda_{j}|^{-1} = e^{x_{j}}$, $|\lambda_{j}| > 1$, where $\Lambda(u) = (x_{1},...,x_{n-1})$. Since $|\lambda_{j}| + |\lambda_{j}|^{-1} > |\lambda_{n-1}| + |\lambda_{n-1}|^{-1} > 2$ we obtain a formula for $|\lambda_{j}|$ in terms of $x_{j}$ by
\begin{align*}
|\lambda_{j}| = \frac{e^{x_{j}}}{2} + \sqrt{\frac{e^{2x_{j}}}{4} - 1}
\end{align*}
so with $\chi_{j} = \log|\lambda_{j}|$ we obtain $x_{j} - \log(2)\leq\chi_{j}\leq x_{j}$. In extension, we have
\begin{align*}
\chi_{j}\geq x_{j} - \log(2) > rx_{j+1} - \log(2)\geq r\chi_{j+1} - \log(2)
\end{align*}
the lemma now follows after possibly replacing $u$ by $u^{N}$ for some large $N\in\mathbb{N}$.
\end{proof}
\begin{lemma}\label{L:LargestSubgroupWithoutHyperbolic}
Let $L\in{\rm GL}(d,\mathbb{Z})$ have property $(P)$. If $\Gamma\leq Z_{{\rm GL}(d,\mathbb{Z})}(L)$ is a subgroup that contains no hyperbolic matrix then we can order the non-zero Lyapunov functionals $\chi_{1},...,\chi_{N},\mu_{1},...,\mu_{N}:\Gamma\to\mathbb{R}$, $N = (d-2)/2$, such that $\mu_{j} = -\chi_{j}$. In particular, if $\Gamma\leq Z_{{\rm GL}(d,\mathbb{Z})}(L)$ contain no hyperbolic matrix then ${\rm rank}(\Gamma)\leq (d-2)/2$.
\end{lemma}
\begin{proof}
Let $\Gamma_{0}\leq Z_{\rm Aut}(L)$ consist of all $A$ that are not hyperbolic. If $A\neq\pm I$ is not hyperbolic, it follows from Lemma \ref{L:PropertiesOfP} that $A$ has property $(P)$. In particular, if we denote by $\mu_{1},...,\mu_{d-2}:Z_{\rm Aut}(L)\to\mathbb{R}$ the Lyapunov exponents then for each $A\in\Gamma_{0}$ we find some permutation $\sigma\in S_{d-2}$ such that $\mu_{j}(A) = -\mu_{\sigma(A)j}(A)$. We claim that $\sigma(A) = \sigma(L)$ for all $A\in\Gamma_{0}$, which proves the first part of the lemma. To see this, we have
\begin{align}\label{Eq:LyapunovExponentTrick}
\mu_{j}(A) + n\mu_{j}(L) = & \mu_{j}(L^{n}A) = -\mu_{\sigma(L^{n}A)j}(L^{n}A) = \\ &
-\mu_{\sigma(L^{n}A)j}(A) - n\mu_{\sigma(L^{n}A)j}(L).
\end{align}
Dividing by $n$ and letting $n$ be sufficiently large we obtain
\begin{align*}
\mu_{j}(L) = -\mu_{\sigma(L^{n}A)j}(L) = -\mu_{\sigma(L)j}(L)
\end{align*}
so $\sigma(L^{n}A) = \sigma(L)$ if $n$ is large enough. Applying this in equation \ref{Eq:LyapunovExponentTrick} we have
\begin{align*}
\mu_{j}(A) = -\mu_{\sigma(L)j}(A) = -\mu_{\sigma(A)j}(A)
\end{align*}
which shows that $\sigma(L) = \sigma(A)$.

We order the Lyapunov exponents by $\chi_{1},...,\chi_{N},-\chi_{1},...,-\chi_{N}:\Gamma_{0}\to\mathbb{R}$ where $\chi_{1}(L) > ... > \chi_{N}(L) > 0$. For the second part of the lemma, it suffices to show that ${\rm rank}(\Gamma_{0})\leq N$. We claim that the map $\Lambda:\Gamma\to\mathbb{R}^{N}$, $\Lambda(\gamma) = (\chi_{1}(\gamma),...,\chi_{N}(\gamma))$ has discrete image which proves the lemma. There is $\varepsilon > 0$ such that if $|\chi_{1}(\gamma)|,...,|\chi_{N}(\gamma)| < \varepsilon$, $\gamma\in\Gamma_{0}$, then all eigenvalues of $\gamma$ are close to $1$ in modulus and for $\varepsilon$ sufficiently small this implies that $\gamma$ only have eigenvalues on $S^{1}$. So any $\gamma\in\Gamma_{0}$, that has image $\Lambda(\gamma)$ sufficiently close to $0$, is mapped to $0$. So, the image of $\Lambda$ is discrete.
\end{proof}